%% file: morrey4.tex
\documentclass[12pt,twoside]{article}
\usepackage{amsmath,amssymb}
\usepackage{amsthm}
\usepackage{upref}
\usepackage{citeref}
\usepackage{color}
\numberwithin{equation}{section}
\usepackage{enumerate}
\usepackage{graphicx}

\theoremstyle{plain}
\newtheorem{theorem}{Theorem}[section]
\newtheorem{proposition}[theorem]{Proposition}
\newtheorem{definition}[theorem]{Definition}
\newtheorem{corollary}[theorem]{Corollary}
\newtheorem{conjecture}[theorem]{Conjecture}
\newtheorem{slope}[theorem]{Slope Rules}
 
\theoremstyle{definition}
\newtheorem{remark}[theorem]{Remark}

\newtheorem{example}[theorem]{Example}

\begin{document}

\newcommand{\eq}{equation}
\newcommand{\real}{\ensuremath{\mathbb R}}
\newcommand{\comp}{\ensuremath{\mathbb C}}
\newcommand{\rn}{\ensuremath{{\mathbb R}^n}}
\newcommand{\tn}{\ensuremath{{\mathbb T}^n}}
\newcommand{\rnp}{\ensuremath{\real^n_+}}
\newcommand{\rnn}{\ensuremath{\real^n_-}}
\newcommand{\Rn}{\ensuremath{{\mathbb R}^{n-1}}}
\newcommand{\Zn}{\ensuremath{{\mathbb Z}^{n-1}}}
\newcommand{\no}{\ensuremath{\nat_0}}
\newcommand{\ganz}{\ensuremath{\mathbb Z}}
\newcommand{\zn}{\ensuremath{{\mathbb Z}^n}}
\newcommand{\zom}{\ensuremath{{\mathbb Z}_{\Om}}}
\newcommand{\zOm}{\ensuremath{{\mathbb Z}^{\Om}}}
\newcommand{\As}{\ensuremath{A^s_{p,q}}}
\newcommand{\Ae}{\ensuremath{A^{s_1}_{p_1,q_1}}}
\newcommand{\Az}{\ensuremath{A^{s_2}_{p_2,q_2}}}
\newcommand{\Bs}{\ensuremath{B^s_{p,q}}}
\newcommand{\Be}{\ensuremath{B^{s_1}_{p_1,q_1}}}
\newcommand{\Bz}{\ensuremath{B^{s_2}_{p_2,q_2}}}
\newcommand{\Fs}{\ensuremath{F^s_{p,q}}}
\newcommand{\Fe}{\ensuremath{F^{s_1}_{p_1,q_1}}}
\newcommand{\Fz}{\ensuremath{F^{s_2}_{p_2,q_2}}}
\newcommand{\At}{\ensuremath{A^{s,\tau}_{p,q}}}
\newcommand{\Bt}{\ensuremath{B^{s,\tau}_{p,q}}}
\newcommand{\Ft}{\ensuremath{F^{s,\tau}_{p,q}}}
\newcommand{\MA}{\ensuremath{{\mathcal A}^s_{u,p,q}}}
\newcommand{\MB}{\ensuremath{{\mathcal N}^s_{u,p,q}}}
\newcommand{\MF}{\ensuremath{{\mathcal E}^s_{u,p,q}}}
\newcommand{\Fsr}{\ensuremath{F^{s,\rloc}_{p,q}}}
\newcommand{\nat}{\ensuremath{\mathbb N}}
\newcommand{\Om}{\ensuremath{\Omega}}
\newcommand{\di}{\ensuremath{{\mathrm d}}}
\newcommand{\Dd}{\ensuremath{{\mathrm D}}}
\newcommand{\sn}{\ensuremath{{\mathbb S}^{n-1}}}
\newcommand{\Ac}{\ensuremath{\mathcal A}}
\newcommand{\Acs}{\ensuremath{\Ac^s_{p,q}}}
\newcommand{\Bc}{\ensuremath{\mathcal B}}
\newcommand{\Cc}{\ensuremath{\mathcal C}}
\newcommand{\cc}{{\scriptsize $\Cc$}${}^s (\rn)$}
\newcommand{\ccd}{{\scriptsize $\Cc$}${}^s (\rn, \delta)$}
\newcommand{\Fc}{\ensuremath{\mathcal F}}
\newcommand{\Lc}{\ensuremath{\mathcal L}}
\newcommand{\Mc}{\ensuremath{\mathcal M}}
\newcommand{\Ec}{\ensuremath{\mathcal E}}
\newcommand{\Pc}{\ensuremath{\mathcal P}}
\newcommand{\Efr}{\ensuremath{\mathfrak E}}
\newcommand{\Mfr}{\ensuremath{\mathfrak M}}
\newcommand{\Mbf}{\ensuremath{\mathbf M}}
\newcommand{\Dbb}{\ensuremath{\mathbb D}}
\newcommand{\Lbb}{\ensuremath{\mathbb L}}
\newcommand{\Pbb}{\ensuremath{\mathbb P}}
\newcommand{\Qbb}{\ensuremath{\mathbb Q}}
\newcommand{\Rbb}{\ensuremath{\mathbb R}}
\newcommand{\vp}{\ensuremath{\varphi}}
\newcommand{\hra}{\ensuremath{\hookrightarrow}}
\newcommand{\supp}{\ensuremath{\mathrm{supp \,}}}
\newcommand{\ssupp}{\ensuremath{\mathrm{sing \ supp\,}}}
\newcommand{\dist}{\ensuremath{\mathrm{dist \,}}}
\newcommand{\unif}{\ensuremath{\mathrm{unif}}}
\newcommand{\ve}{\ensuremath{\varepsilon}}
\newcommand{\vk}{\ensuremath{\varkappa}}
\newcommand{\vr}{\ensuremath{\varrho}}
\newcommand{\pa}{\ensuremath{\partial}}
\newcommand{\oa}{\ensuremath{\overline{a}}}
\newcommand{\ob}{\ensuremath{\overline{b}}}
\newcommand{\of}{\ensuremath{\overline{f}}}
\newcommand{\LA}{\ensuremath{L^r\!\As}}
\newcommand{\LcA}{\ensuremath{\Lc^{r}\!A^s_{p,q}}}
\newcommand{\LcdA}{\ensuremath{\Lc^{r}\!A^{s+d}_{p,q}}}
\newcommand{\LcB}{\ensuremath{\Lc^{r}\!B^s_{p,q}}}
\newcommand{\LcF}{\ensuremath{\Lc^{r}\!F^s_{p,q}}}
\newcommand{\Lf}{\ensuremath{L^r\!f^s_{p,q}}}
\newcommand{\La}{\ensuremath{\Lambda}}
\newcommand{\Lob}{\ensuremath{L^r \ob{}^s_{p,q}}}
\newcommand{\Lof}{\ensuremath{L^r \of{}^s_{p,q}}}
\newcommand{\Loa}{\ensuremath{L^r\, \oa{}^s_{p,q}}}
\newcommand{\Lcoa}{\ensuremath{\Lc^{r}\oa{}^s_{p,q}}}
\newcommand{\Lcob}{\ensuremath{\Lc^{r}\ob{}^s_{p,q}}}
\newcommand{\Lcof}{\ensuremath{\Lc^{r}\of{}^s_{p,q}}}
\newcommand{\Lca}{\ensuremath{\Lc^{r}\!a^s_{p,q}}}
\newcommand{\Lcb}{\ensuremath{\Lc^{r}\!b^s_{p,q}}}
\newcommand{\Lcf}{\ensuremath{\Lc^{r}\!f^s_{p,q}}}
\newcommand{\id}{\ensuremath{\mathrm{id}}}
\newcommand{\tr}{\ensuremath{\mathrm{tr\,}}}
\newcommand{\trd}{\ensuremath{\mathrm{tr}_d}}
\newcommand{\trL}{\ensuremath{\mathrm{tr}_L}}
\newcommand{\ext}{\ensuremath{\mathrm{ext}}}
\newcommand{\re}{\ensuremath{\mathrm{re\,}}}
\newcommand{\Rea}{\ensuremath{\mathrm{Re\,}}}
\newcommand{\Ima}{\ensuremath{\mathrm{Im\,}}}
\newcommand{\loc}{\ensuremath{\mathrm{loc}}}
\newcommand{\rloc}{\ensuremath{\mathrm{rloc}}}
\newcommand{\osc}{\ensuremath{\mathrm{osc}}}
\newcommand{\pr}{\pageref}
\newcommand{\wh}{\ensuremath{\widehat}}
\newcommand{\wt}{\ensuremath{\widetilde}}
\newcommand{\ol}{\ensuremath{\overline}}
\newcommand{\os}{\ensuremath{\overset}}
\newcommand{\Li}{\ensuremath{\overset{\circ}{L}}}
\newcommand{\Ai}{\ensuremath{\os{\, \circ}{A}}}
\newcommand{\Ci}{\ensuremath{\os{\circ}{\Cc}}}
\newcommand{\dom}{\ensuremath{\mathrm{dom \,}}}
\newcommand{\SA}{\ensuremath{S^r_{p,q} A}}
\newcommand{\SB}{\ensuremath{S^r_{p,q} B}}
\newcommand{\SF}{\ensuremath{S^r_{p,q} F}}
\newcommand{\Hc}{\ensuremath{\mathcal H}}
\newcommand{\Nc}{\ensuremath{\mathcal N}}
\newcommand{\Lci}{\ensuremath{\overset{\circ}{\Lc}}}
\newcommand{\bmo}{\ensuremath{\mathrm{bmo}}}
\newcommand{\BMO}{\ensuremath{\mathrm{BMO}}}
\newcommand{\cm}{\\[0.1cm]}
\newcommand{\Aa}{\ensuremath{\os{\, \ast}{A}}}
\newcommand{\Ba}{\ensuremath{\os{\, \ast}{B}}}
\newcommand{\Fa}{\ensuremath{\os{\, \ast}{F}}}
\newcommand{\Aas}{\ensuremath{\Aa{}^s_{p,q}}}
\newcommand{\Bas}{\ensuremath{\Ba{}^s_{p,q}}}
\newcommand{\Fas}{\ensuremath{\Fa{}^s_{p,q}}}
\newcommand{\Ca}{\ensuremath{\os{\, \ast}{{\mathcal C}}}}
\newcommand{\Cas}{\ensuremath{\Ca{}^s}}
\newcommand{\Car}{\ensuremath{\Ca{}^r}}
\newcommand{\bl}{$\blacksquare$}
\newcommand{\bli}{\begin{enumerate}[{\upshape\bfseries (i)}]}
\newcommand{\eli}{\end{enumerate}}
\newcommand{\ignore}[1]{}
\newcommand{\rhoAs}{\ensuremath\vr\text{-} \!\As}
\newcommand{\rhoAe}{\ensuremath\vr\text{-} \!\Ae}
\newcommand{\rhoAz}{\ensuremath\vr\text{-} \!\Az}
\newcommand{\rhoeAe}{\ensuremath\vr_1\text{-} \!\Ae}
\newcommand{\rhozAz}{\ensuremath\vr_2\text{-} \!\Az}
\newcommand{\rhoBs}{\ensuremath\vr\text{-} \!\Bs}
\newcommand{\rhoBe}{\ensuremath\vr\text{-} \!\Be}
\newcommand{\rhoBz}{\ensuremath\vr\text{-} \!\Bz}
\newcommand{\bs}{\ensuremath{b^s_{p,q}}}
\newcommand{\rhobs}{\ensuremath\vr\text{-} \!\bs}
\newcommand{\rhoxAx}[5]{\ensuremath\vr_{#1}\text{-} \!{#2}^{#3}_{#4,#5}}
\newcommand{\rhoAx}[3]{\ensuremath\vr\text{-} \!A^{#1}_{#2,#3}}
\newcommand{\rhoFs}{\ensuremath\vr\text{-} \!\Fs}
\newcommand{\rhoFe}{\ensuremath\vr\text{-} \!\Fe}
\newcommand{\rhoFz}{\ensuremath\vr\text{-} \!\Fz}
\newcommand{\Lloc}{L_1^{\mathrm{loc}}}

\newcommand{\ddh}[1]{{\color{red} #1}}%
\newcommand{\new}[1]{{\color{blue} #1}}%
\renewcommand{\new}[1]{{#1}}%
\newcommand{\open}[1]{\smallskip\noindent\fbox{\parbox{\textwidth}{\color{blue}\bfseries\begin{center}
      #1 \end{center}}}\\ \smallskip}

\newcommand{\unterbild}[1]{{\noindent\refstepcounter{figure}\upshape\bfseries
    Figure {\thefigure}{\label{#1}}}\addcontentsline{lof}{figure}{Figure \ref{#1}}}%
\newcommand{\egv}[1]{\ensuremath{{\mathcal E}\rule[-0.7ex]{0em}{2.4ex}_{\!\mathsf G}^{#1}}}
\newcommand{\egX}{\egv{X}}
\newcommand{\fundfn}[1]{\varphi\rule[-0.6ex]{0ex}{2ex}_{#1}}
\newcommand{\M}{\ensuremath{{\cal M}_{u,p}}}
\newcommand{\at}{{A}_{p,q}^{s,\tau}}
\newcommand{\envg}{\ensuremath{{\mathfrak E}\rule[-0.5ex]{0em}{2.3ex}_{\!\mathsf G}}}
\newcommand{\uGindex}[1]{u\rule[-0.6ex]{0ex}{1.2ex}_{\mathsf G}^{#1}}
\newcommand{\uGx}{\uGindex{X}}


\thispagestyle{empty}

\begin{center}
{\Large Morrey smoothness spaces: A new approach}
\\[0.5cm]
\today
\\[1cm]
{\large Dorothee D. Haroske and Hans Triebel}
\end{center}

\begin{abstract}
  \new{
    In the recent years so-called Morrey smoothness spaces attracted a lot of interest. They can (also) be understood as generalisations of the classical spaces $\As (\rn)$, $A\in \{B,F\}$, in \rn, where the parameters satisfy 
    $s\in \real$ (smoothness), $0<p \le \infty$ (integrability) and $0<q \le
    \infty$ (summability). In the case of Morrey smoothness spaces additional parameters are involved. In our opinion, among the various approaches at least two scales enjoy special attention, also in view of applications: the scales $\Ac^s_{u,p,q} (\rn)$, with $\Ac\in \{\Nc, \Ec\}$, $u\geq p$, and $A^{s, \tau}_{p,q} (\rn)$, with $\tau\geq 0$.

    We reorganise these two prominent types of Morrey smoothness  spaces by adding to $(s,p,q)$ the so--called slope parameter $\vr$, preferably 
    (but not exclusively) with $-n \le \vr <0$.} 
  It comes out that $|\vr|$ replaces $n$, and $\min
(|\vr|,1)$ replaces 1 in slopes of (broken) lines in the $\big( \frac{1}{p}, s\big)$--diagram characterising distinguished properties 
of the spaces $\As (\rn)$ \new{and their Morrey counterparts}. Special attention will be paid to low--slope spaces with $-1 <\vr <0$, where corresponding properties are 
quite often independent  of $n\in \nat$. 
    
Our aim is two--fold. On the one hand we reformulate  some assertions already available in the literature (many of them  are quite
recent). On the other hand we establish on this basis new properties, a few of them became visible only in the context of the offered
new approach, governed, now, by the four parameters $(s,p,q,\vr)$.
\end{abstract}


\newpage

\tableofcontents


\section{Introduction}    \label{S1}
The nowadays classical spaces
\begin{\eq}   \label{1.1}
\As(\rn) \qquad \text{with} \quad A\in \{B,F \}, \quad s \in \real \quad \text{and} \quad 0<p,q \le \infty,
\end{\eq}
have been extended in several directions, most notably into two types of Morrey smoothness spaces,
\begin{\eq}  \label{1.2}
\Nc^s_{u,p,q} (\rn) \quad \text{and} \quad \Ec^s_{u,p,q} (\rn) \qquad \text{with} \quad p \le u <\infty,
\end{\eq}
on the one hand and
\begin{\eq}  \label{1.3}
A^{s, \tau}_{p,q} (\rn), \quad A\in \{B,F\}, \qquad \text{with} \quad 0 \le \tau <\infty,
\end{\eq}
on the other hand. With $p=u$ in \eqref{1.2} and $\tau =0$, $p<\infty$ in \eqref{1.3} one obtains the corresponding spaces in 
\eqref{1.1}.

\new{These scales of spaces \eqref{1.2}, \eqref{1.3} were investigated intensively in recent years, we add some further historic remarks below. A strong motivation to study extensions from the scales of function spaces \eqref{1.1} to this Morrey setting came from possible applications to PDEs, as it was already the case for the `basic' Morrey spaces extending $L_p$. Some Besov-Morrey spaces were first introduced by Netrusov in \cite{Ne-10} by means  of differences.  One of the milestones in this direction is the famous paper by Kozono and Yamazaki  \cite{KY} where they used spaces of type \eqref{1.2} to study Navier-Stokes equations; we also refer in this context to the papers by Mazzucato \cite{Maz03}, Ferreira and Postigo \cite{FP}, or by Yang, Fu, and Sun \cite{YFS}, by Lemari\'{e}-Rieusset \cite{LMR-5,LMR-4}, as well as to the monographs \cite{FHS-MS-1,FHS-MS-2,YSY10}.

}

It is natural to ask for counterparts of distinguished properties of  the spaces $\As (\rn)$ in the context of these
Morrey smoothness spaces. Typical examples are embeddings in $L_\infty (\rn)$ and in $L^{\loc}_1 (\rn)$ (regular distributions)
or traces on hyper--planes. One may also be interested to find out those spaces to which the $\delta$--distribution  or the 
characteristic function $\chi_Q$ of the unit cube $Q = (0,1)^n$, $n\in \nat$, belong. Some final answers have been obtained  in recent
times. But the related conditions for the above questions are not very appealing, producing quite often curved lines in the well--known
$\big( \frac{1}{p},s \big)$--diagram, \new{where any space of type $\As$ is indicated by its smoothness parameters $s$ and the integrability $p$, neglecting the fine index $q$ for the moment; we refer to \cite{HMS16} for such examples}. This suggests to search for a re--parametrisation of the spaces in \eqref{1.2}, \eqref{1.3} such that the outcome produces natural and transparent conditions for distinguished properties of these spaces. \new{It turns out that the previous, separately obtained results, based on independent arguments, can thus not only be understood in a better way, detached from the (sometimes quite involved) technical requirements. But one might also observe more intrinsic reasons for common phenomena. This will also constitute some basis for unified results and, occasionally, lead to appropriate conjectures.}

The classical parameters $s$ 
(smoothness), $p$ (integrability) and $q$ (summability) in \eqref{1.1}, but also in \eqref{1.2}, \eqref{1.3}, are untouchable. But
we replace $u$ in \eqref{1.2} and $\tau$ in \eqref{1.3} by the  common parameter $\vr$, typically (but not exclusively) with $-n \le
\vr <0$. The corresponding spaces
\begin{\eq}   \label{1.4}
\rhoAs (\rn) = \big\{   \La^{\vr} \As (\rn),  \La_{\vr} \As (\rn): \  A\in \{B,F\}\big\},
\end{\eq}
cover all space in \eqref{1.2}, \eqref{1.3} with
\begin{\eq}   \label{1.5}
\La^{-n} \As (\rn) = \La_{-n} \As (\rn) = \As (\rn).
\end{\eq}
We call $\vr$ the slope parameter because $|\vr|$ quite often takes over the r\^ole of the slope $n$ in distinguished (broken) lines in
the $\big( \frac{1}{p},s\big)$--diagram. For example, the sharp embedding
\begin{\eq}   \label{1.6}
\As (\rn) \hra C(\rn) \qquad \text{if} \quad s > \frac{n}{p}, \quad 0<p<\infty,
\end{\eq}
as far as the breaking line is concerned, has now the sharp counterpart
\begin{\eq}   \label{1.7}
\rhoAs(\rn) \hra C(\rn) \quad \text{if} \quad s > \frac{|\vr|}{p}, \quad 0<p<\infty.
\end{\eq}
\new{We refer to Theorems~\ref{T5.3} and \ref{T5.18} and Figure~\ref{fig-1a} on p.~\pageref{fig-1a} below.}

This applies also to other distinguished properties, many of them are discovered only recently. We touch on two of them. Let
\begin{\eq}   \label{1.8}
\sigma^t_p = t \Big( \max \left( \frac{1}{p},1 \right) - 1 \Big), \qquad t\ge 0, \quad 0<p<\infty.
\end{\eq}
Then the sharp inclusion
\begin{\eq}   \label{1.9}
\As (\rn) \subset L^{\loc}_1 (\rn) \qquad \text{if} \quad s > \sigma^n_p, \quad 0<p<\infty,
\end{\eq}
as far as the breaking line is concerned, has now the sharp counterpart
\begin{\eq}   \label{1.10}
\rhoAs(\rn) \subset L^{\loc}_1 (\rn)  \quad \text{if} \quad s>  \sigma^{|\vr|}_p, \quad 0<p<\infty.
\end{\eq}
\new{This can be found in Theorem~\ref{T5.7} and Figure~\ref{fig-1b} on p.~\pageref{fig-1b} below.}

For the characteristic function $\chi_Q$ of the cube $Q = (0,1)^n$ the sharp assertion 
\begin{\eq}  \label{1.11}
\chi_Q \in \As (\rn) \qquad \text{if} \quad s < \frac{1}{p}, \quad 0<p<\infty,
\end{\eq}
as far the breaking line is concerned, has now the sharp counterpart
\begin{\eq}   \label{1.12}
\chi_Q \in \rhoAs(\rn) \qquad \text{if} \quad s < \frac{1}{p} \min (|\vr|, 1 ), \quad 0<p<\infty.
\end{\eq}
The generalisation of the slope $n$ in \eqref{1.6}, \eqref{1.9} by $|\vr|$ in \eqref{1.7}, \eqref{1.10} obeys the so--called 
{\em Slope--$n$--Rule}, whereas the replacement of $1$ in \eqref{1.11} by $\min(|\vr|,1)$ in \eqref{1.12} is a typical example of the
{\em Slope--$1$--Rule}, \new{as formulated in Section~\ref{S2.2} below}. It is one of the main aims of this paper  to reformulate already existing assertions for the spaces in \eqref{1.2},
\eqref{1.3} in terms of these slope rules for the spaces in \eqref{1.4}. This will be complemented by some new properties. Then 
detailed proofs will be given.

The paper is organized as follows. In Section~\ref{S2} we introduce the above--mentioned spaces in \eqref{1.1}--\eqref{1.4} and 
discuss in detail their interrelations, including coincidences and diversities. It is central for our approach to structure the
somewhat bewildering  plethora of the Morrey smoothness spaces in common use according to \eqref{1.2}, \eqref{1.3} into $\vr$--clans
consisting, roughly speaking, of the spaces in \eqref{1.4} such that the above--mentioned slope rules can be applied. But the
reorganisation of the above Morrey smoothness spaces into $\vr$--clans is not only an efficient technical device. Any $\vr$--clan  is
an organic entity on equal footing with the classical spaces in \eqref{1.1}. Its families in \eqref{1.4} have different abilities
complementing each other symbiotically. In Section~\ref{S3} we collect some tools: wavelet characterisations, interpolations, lifts,
characterisations by derivatives and the useful Fatou property. Section~\ref{S4} deals with the so--called key properties (smooth
pointwise multipliers, diffeomorphisms, extensions of corresponding spaces in $\real^n_+$ to their counterparts in \rn, and traces on
hyper--planes). First we try to justify in Section~\ref{S4.1} why the classical spaces in \eqref{1.1} and now their Morrey 
generalisations in \eqref{1.4} (as reformulations of \eqref{1.2}, \eqref{1.3}) deserve to be studied. It comes out that the first
three of the above--mentioned 
key properties can be treated rather quickly based on what is already known. The situation is different as far as traces of
suitable spaces in \eqref{1.4} on $\Rn$ are concerned. This gives us the possibility  to show how our approach can be used to obtain
new substantial assertions about traces and related extensions. \new{In Section~\ref{S5} we discuss essential features like the precise
versions of \eqref{1.7}, \eqref{1.10}, \eqref{1.12}. This will be complemented by some further topics such as truncations for spaces 
on $\rn$, expansions by Haar wavelets, Faber expansions. Finally, in Section~\ref{S7} we concentrate on embeddings again, beginning with the situation of embeddings of spaces on bounded domains, including related compactness and entropy number results. We comment on growth envelope functions which represent some tool to `measure' unboundedness in function spaces. In particular, we always pay attention to these results formulated 
in the context of the $\vr$-clans according to the slope rules. We end our paper by an outlook when -- different from our approach so far -- we consider different $\vr$-clans, for instance, embeddings between spaces belonging to the $\vr_1$- and $\vr_2$-clan with $\vr_1\neq \vr_2$. In that sense we `cross borders' between the $\vr$-clan structure we propose here. In our opinion the results, also illustrated in some further diagrams there, are very much convincing to use the new $\vr$-clan setting and follow this line of research further.
}

\section{Definitions and basic properties}   \label{S2}
\subsection{Definitions}    \label{S2.1}
We use standard notation. Let $\nat$ be the collection of all natural numbers and $\no = \nat \cup \{0 \}$. Let $\rn$ be the Euclidean $n$-space, where $n \in \nat$. Furthermore, we set $\real = \real^1$, and $\comp$ is the complex plane. Let $S(\rn)$ be the usual Schwartz space of all complex-valued rapidly decreasing
infinitely differentiable functions on $\rn$, and let $S' (\rn)$ be the  dual space of all
tempered distributions on \rn. Let $L_p (\rn)$ with $0<p \le \infty$ be the standard complex quasi-Banach space with respect to the Lebesgue measure in \rn, quasi-normed by
\begin{\eq}   \label{2.1}
\| f \, | L_p (\rn) \| = \Big( \int_{\rn} |f(x)|^p \, \di x \Big)^{1/p}
\end{\eq}
with the natural modification for $p= \infty$. Similarly we define $L_p (M)$ where $M$ is a Lebesgue--measurable subset of \rn.
As usual $\ganz$ is the collection of all integers; and $\zn$ where $n\in \nat$ denotes the lattice of all points $m = (m_1, \ldots, m_n) \in \rn$ with $m_j \in \ganz$. If $\vp \in S(\rn)$, then
\begin{\eq}   \label{2.2}
\wh{\vp} (\xi) = (F \vp)(\xi) = (2 \pi)^{-n/2} \int_{\rn} e^{-ix \xi} \, \vp (x) \, \di  x, \qquad \xi \in \rn,
\end{\eq}
denotes the Fourier transform of \vp. As usual, $F^{-1} \vp$ and $\vp^\vee$ stand for the inverse Fourier transform, which is given by the right-hand side of 
\eqref{2.2} with $i$ instead of $-i$. Note that $x \xi = \sum^n_{j=1} x_j \xi_j$ 
stands for the scalar product in \rn. Both $F$ and $F^{-1}$ are extended to $S' (\rn)$ in the standard way. 

Let $\vp_0 \in S(\rn)$ with
\begin{\eq}   \label{2.3}
\vp_0 (x) =1 \ \text{if $|x| \le 1$} \qquad \text{and} \qquad \vp_0 (x) =0 \ \text{if $|x| \ge 3/2$},
\end{\eq}
and let
\begin{\eq}   \label{2.4}
\vp_k (x) = \vp_0 \big( 2^{-k} x) - \vp_0 \big( 2^{-k+1} x \big), \qquad x \in \rn, \quad k \in \nat.
\end{\eq}
Since
\begin{\eq}    \label{2.5}
\sum^\infty_{j=0} \vp_j (x) = 1 \qquad \text{for} \quad x\in \rn,
\end{\eq}
the $\vp_j$ form a dyadic resolution of unity. The entire analytic functions $(\vp_j \wh{f} )^\vee (x)$ make sense pointwise in $\rn$ for any $f \in S' (\rn)$. Let
\begin{\eq}   \label{2.6}
Q_{J,M} = 2^{-J} M + 2^{-J} (0,1)^n, \qquad J \in \ganz, \quad M \in \zn,
\end{\eq}
be the usual dyadic cube in \rn, $n\in \nat$, with sides of length $2^{-J}$ parallel to the coordinate axes and with $2^{-J}M$
as lower left corner. If $Q$ is a cube in $\rn$ and $d>0$, then $dQ$ is the cube in $\rn$ concentric with $Q$ and whose the side--length is
$d$ times the side--length of $Q$. Let $|\Om|$ be the Lebesgue measure of the Lebesgue measurable  set $\Om$ in \rn. Let $a^+ =
\max (a,0)$ for $a\in \real$. Furthermore
\begin{\eq}  \label{2.7}
a_i \sim b_i \qquad \text{for} \quad i \in I \quad \text{(equivalence)}
\end{\eq}
for two sets of positive numbers $\{a_i : i\in I \}$ and $\{b_i : i\in I \}$ means that there are two positive numbers $c_1$ and $c_2$
such that
\begin{\eq}   \label{2.8}
c_1 a_i \le b_i \le c_2 \new{a_i} \qquad \text{for all} \quad i\in I.
\end{\eq}

\begin{definition}   \label{D2.1}
Let $n\in \nat$ and let $\vp = \{ \vp_j \}^\infty_{j=0}$ be the above dyadic resolution  of unity.
\bli
\item Let
\begin{\eq}   \label{2.9}
0<p \le \infty, \qquad 0<q \le \infty, \qquad s \in \real.
\end{\eq}
Then $\Bs (\rn)$ is the collection of all $f \in S' (\rn)$ such that
\begin{\eq}   \label{2.10}
\| f \, | \Bs (\rn) \|_{\vp} = \Big( \sum^\infty_{j=0} 2^{jsq} \big\| (\vp_j \widehat{f})^\vee \, | L_p (\rn) \big\|^q \Big)^{1/q} 
\end{\eq}
is finite $($with the usual modification if $q= \infty)$. 
\item
Let
\begin{\eq}   \label{2.11}
0<p<\infty, \qquad 0<q \le \infty, \qquad s \in \real.
\end{\eq}
Then $F^s_{p,q} (\rn)$ is the collection of all $f \in S' (\rn)$ such that
\begin{\eq}   \label{2.12}
\| f \, | F^s_{p,q} (\rn) \|_{\vp} = \Big\| \Big( \sum^\infty_{j=0} 2^{jsq} \big| (\vp_j \wh{f})^\vee (\cdot) \big|^q \Big)^{1/q} \big| L_p (\rn) \Big\|
\end{\eq}
is finite $($with the usual modification if $q=\infty)$.
\item
Let $0<q \le \infty$ and $s\in \real$. Then $F^s_{\infty,q} (\rn)$ is the collection of all $f\in S'(\rn)$ such that
\begin{\eq}  \label{2.13}
\| f \, | F^s_{\infty,q} (\rn) \|_{\vp} = \sup_{J\in \ganz, M\in \zn} 2^{Jn/q} \Big(\int_{Q_{J,M}} \sum_{j \ge J^+} 2^{jsq} \big|
(\vp_j \wh{f} )^\vee (x) \big|^q \, \di x \Big)^{1/q}
\end{\eq}
is finite $($with the usual modification if $q=\infty$ as explained below$)$.
\eli
\end{definition}

\begin{remark}    \label{R2.2}
These are the classical spaces $\As (\rn)$ according to \eqref{1.1}. The above definition coincides with {\cite[Definition 
1.1, p.\,2]{T20}}, {including} part (iii). There one finds also some discussions and (historical) references, especially about the
spaces $F^s_{\infty,q} (\rn)$. \new{Let us mention here, in particular, the series of monographs \cite{T83,T92,T06} and the long paper \cite{FrJ90}.}
In particular, it is well known that the spaces in the above definition are independent of the chosen
resolution of unity $\vp$ (equivalent quasi--norms). This justifies our omission of the subscript $\vp$ in \eqref{2.10}, \eqref{2.12}
and \eqref{2.13} in the sequel. They are quasi--Banach spaces (and Banach spaces if $p \ge 1$, $q \ge 1$). As remarked in 
\cite[p.\,3]{T20} one can replace $J \in \ganz$ in \eqref{2.13} by $J \in \no$ (equivalent quasi--norms),
\begin{\eq}  \label{2.14}
\| f \, | F^s_{\infty,q} (\rn) \| = \sup_{J\in \no, M\in \zn} 2^{Jn/q} \Big(\int_{Q_{J,M}} \sum_{j \ge J} 2^{jsq} \big|
(\vp_j \wh{f} )^\vee (x) \big|^q \, \di x \Big)^{1/q}.
\end{\eq}
If $q= \infty$, then \eqref{2.13} and \eqref{2.14} must be understood as
\begin{\eq}   \label{2.15}
\| f\, | F^s_{\infty, \infty} (\rn) \| \sim \sup_{j\in \no, x\in \rn} 2^{js} \big| \big(\vp_j \wh{f} \big)^\vee (x) \big| = \|f \, |
B^s_{\infty, \infty} (\rn) \|.
\end{\eq}
In other words, the {\em H\"{o}lder--Zygmund spaces} $\Cc^s (\rn) = B^s_{\infty, \infty} (\rn)$ can be incorporated into the $F$--scale as
\begin{\eq}  \label{2.16}
F^s_{\infty, \infty} (\rn) = B^s_{\infty, \infty} (\rn) = \Cc^s (\rn), \qquad s\in \real.
\end{\eq}
\end{remark}

The spaces in \eqref{1.2} generalise the above spaces $\As (\rn)$ with $p<\infty$ replacing there the Lebesgue spaces $L_p (\rn)$ by
the well--known Morrey spaces which we adapt to our later purposes as follows. As usual, $f\in L^{\loc}_p (\rn)$ means that the
restriction of $f$ to any bounded Lebesgue--measurable set $M$ in $\rn$ belongs to $L_p (M)$, $0<p \le \infty$. Let again $Q_{J,M}$ be
the cubes in \eqref{2.6}.

\begin{definition}   \label{D2.3}
Let $n\in \nat$, $0<p<\infty$ and $-n \le \vr <0$. Then $\La^{\vr}_p (\rn)$ collects all $f \in L^{\loc}_p (\rn)$ such that
\begin{\eq}   \label{2.17}
\| f \, | \La^{\vr}_p (\rn) \| = \sup_{J\in \ganz, M\in \zn} 2^{\frac{J}{p} (n + \vr)} \| f \, | L_p (Q_{J,M} ) \|
\end{\eq}
is finite.
\end{definition}

\begin{remark}  \label{R2.4}
Obviously, $\La^{\vr}_p (\rn)$ are quasi--Banach spaces (and Banach spaces if $p \ge 1$). Furthermore,
\begin{\eq}   \label{2.18}
\La^{-n}_p (\rn) = L_p (\rn), \qquad 0<p<\infty.
\end{\eq}
The {\em Morrey space} $\Mc^{u}_p (\rn)$ with $0<p \le u <\infty$ in common use collects all $f\in L^{\loc}_p (\rn)$ such that
\begin{\eq}   \label{2.19}
\| f \, | \Mc^{u}_p (\rn) \|= \sup_{J\in \ganz, M\in \zn} 2^{Jn (\frac{1}{p} - \frac{1}{u})} \| f \, |L_p (Q_{J,M}) \|
\end{\eq}
is finite. Compared with \eqref{2.17} one obtains
\begin{\eq}    \label{2.20}
\La^{\vr}_p (\rn) = \Mc^{u}_p (\rn) \qquad \text{with} \quad 0<p<\infty, \quad  u \vr  + np =0.
\end{\eq}
\end{remark}

\begin{remark}
  The idea of Morrey spaces $\Mc^{u}_p (\rn)$, $0< p \le u < \infty$, goes back to Morrey in \cite{Mor}, dealing with the regularity
of solutions of some partial differential equations. They are part of a wider class of Morrey-Campanato spaces, cf. \cite{Pee}, and can be seen as a complement to $L_p$ spaces, since $\Mc^{p}_p (\rn)= L_p(\rn)$. In an analogous way, one can define the spaces $\Mc^\infty_{p}(\rn)$, $p\in(0, \infty)$, but using the Lebesgue differentiation theorem, one can easily prove  that
	$\Mc^\infty_p(\rn) = L_\infty(\rn)$. So we usually restrict ourselves to $u<\infty$.
  \end{remark}

We replace $L_p (\rn)$ with $p<\infty$ in Definition~\ref{D2.1} by the above spaces $\La^{\vr}_p (\rn)$ whereas all other notation
have the same meaning as there.

\begin{definition}    \label{D2.5}
Let $n\in \nat$, $s\in \real$, $0<p<\infty$, $0<q \le \infty$ and $-n \le \vr <0$. Then $\La_{\vr} \Bs (\rn)$ is the collection of all $f \in S' (\rn)$ such that
\begin{\eq}   \label{2.21}
\|f\, | \La_{\vr} \Bs (\rn) \| = \Big( \sum^\infty_{j=0} 2^{jsq} \big\| (\vp_j \wh{f} \big)^\vee \, | \La^{\vr}_p (\rn) \big\|^q \Big)^{1/q}
\end{\eq}
is finite $($with the usual modification if $q=\infty)$ and $\La_{\vr} \Fs (\rn)$ is the collection of all $f\in S'(\rn)$ such that
\begin{\eq}   \label{2.22}
\|f \, | \La_{\vr} \Fs (\rn) \| = \Big\| \Big( \sum^\infty_{j=0} 2^{jsq} \big| \big( \vp_j \wh{f} \big)^\vee (\cdot) \big|^q \Big)^{1/q} \, | \Lambda^{\vr}_p (\rn) \Big\|
\end{\eq}
is finite $($with the usual modification if $q= \infty)$.
\end{definition}

\begin{remark}   \label{R2.6}
By \eqref{2.18} one has
\begin{\eq}    \label{2.23}
\La_{-n} \As (\rn) = \As (\rn), \qquad A \in \{ B,F \}.
\end{\eq}
Using \eqref{2.20} and the standard definitions in the literature  it comes out that
\begin{\eq}   \label{2.24}
\Lambda_{\vr} \Bs (\rn) = \Nc^s_{u,p,q} (\rn), \qquad u \vr + np =0,
\end{\eq}
and
\begin{\eq}   \label{2.25}
\La_{\vr} \Fs (\rn) = \Ec^s_{u,p,q} (\rn), \qquad u \vr + np =0.
\end{\eq}
These spaces attracted some attention in the last decades. In particular they are quasi--Banach spaces which are independent of $\vp =
\{\vp_j \}^\infty_{j=0}$ (equivalent quasi--norms).

\new{Besov-Morrey spaces $\Nc^s_{u,p,q} (\rn)$ were introduced by Kozono and Yamazaki in
	\cite{KY}. They studied semi-linear heat equations and Navier-Stokes
	equations with initial data belonging to  Besov-Morrey spaces.  The
	investigations were continued by Mazzucato \cite{mazzu-2}, where one can find the
	atomic decomposition of some spaces. The Triebel-Lizorkin-Morrey spaces
        $\Ec^s_{u,p,q} (\rn)$	were later introduced by  Tang and Xu \cite{TX}. The ideas were further developed by Sawano and Tanaka \cite{SawTan,SawTan-2,Saw08,Saw09}. The most systematic and general approach to the spaces of this type  can  be found in the monograph \cite{YSY10} or in the  survey papers by Sickel \cite{Sic12,Sic13}, which we also recommend for further up-to-date references on this subject. We refer to the recent monographs \cite{FHS-MS-1,FHS-MS-2} for applications.

        It turned out that many of the results from the classical situation have their  counterparts for the spaces $\mathcal{A}^s_{u,p,q}(\rn)$, e.g. in view of elementary embeddings. However, there also exist some differences.
Sawano proved in \cite{Saw08} that, for $s\in\real$ and $0<p< u<\infty$,
\begin{equation}\label{elem}
{\cal N}^s_{u,p,\min\{p,q\}}(\rn)\, \hookrightarrow \, {\cal E}^s_{u,p,q}(\rn)\, \hookrightarrow \,{\cal N}^s_{u,p,\infty}(\rn),
\end{equation}
where in the latter embedding $\infty$ cannot be improved -- unlike
in the classical case of $u=p$. 
On the other hand, Mazzucato has shown in \cite[Proposition~4.1]{mazzu-2} that
\[
\mathcal{E}^0_{u,p,2}(\rn)=\mathcal{M}^u_p(\rn),\quad 1<p\leq u<\infty.
\]
}
\end{remark}

Next we introduce a second type of Morrey smoothness spaces. Let $\vp = \{\vp_j \}^\infty_{j=0}$ and $Q_{J,M}$ be as in
\eqref{2.3}--\eqref{2.6}.

\begin{definition}    \label{D2.7}
Let $n\in \nat$, $s\in \real$ and $0<p <\infty$, $0<q \le \infty$. Let $-n \le \vr <\infty$. Then $\La^{\vr} \Bs (\rn)$ is the 
collection of all $f\in S'(\rn)$ such that
\begin{\eq}   \label{2.26}
\| f \, | \La^{\vr} \Bs (\rn) \| = \sup_{J\in \ganz, M\in \zn} 2^{\frac{J}{p} (n+ \vr)} \Big( \sum_{j \ge J^+} 2^{jsq} \big\|
\big( \vp_j \wh{f}\, \big)^\vee \, | L_p (Q_{J,M} ) \big\|^q \Big)^{1/q}
\end{\eq}
is finite and $\La^{\vr} \Fs (\rn)$ is the collection of all $f\in S'(\rn)$ such that
\begin{\eq}   \label{2.27}
\begin{aligned}
&\| f \, | \La^{\vr} \Fs (\rn) \| \\
& = \sup_{J\in \ganz, M\in \zn} 2^{\frac{J}{p} (n + \vr)}\Big\| \Big( \sum_{j \ge J^+} 2^{jsq}
\big| \big( \vp_j \wh{f}\, \big)^\vee (\cdot) \big|^q \Big)^{1/q} \, | L_p (Q_{J,M}) \Big\|
\end{aligned}
\end{\eq}
is finite $($usual modification if $q=\infty)$.
\end{definition}

\begin{remark}  \label{R2.8}
By measure--theoretical arguments one has
\begin{\eq}   \label{2.28}
\La^{-n} \As (\rn) = \As (\rn), \qquad A\in \{B,F \},
\end{\eq}
for all admitted $s,p,q$. Otherwise these spaces are reformulations of corresponding Morrey smoothness spaces $A^{s,\tau}_{p,q} (\rn)$
which attracted a lot of attention,
\begin{\eq}   \label{2.29}
\La^{\vr} \As (\rn) = A^{s, \tau}_{p,q} (\rn), \qquad s\in \real, \quad 0<p<\infty, \quad 0<q \le \infty
\end{\eq}
where
\begin{\eq}   \label{2.30}
\tau = \frac{1}{p} \Big( 1 + \frac{\vr}{n} \Big), \qquad -n \le \vr <\infty.
\end{\eq}
\new{
  The so-called Besov-type spaces spaces $B^{s, \tau}_{p,q}$ and Triebel-Lizorkin-type spaces $F^{s, \tau}_{p,q}$, commonly denoted by $A^{s, \tau}_{p,q}$ nowadays with $A\in\{B,F\}$, were introduced in \cite{YSY10} and proved therein to be quasi-Banach spaces. In the Banach case  the scale of  Besov type spaces had already been introduced and investigated in \cite{baraka-1,baraka-2,baraka-3}, and also by Yuan and Yang \cite{YY,YY-2}. We refer again to the monograph  \cite{YSY10} and the fine survey papers \cite{Sic12,Sic13} for further background information and references. In \cite{LYYSU} an even more general approach was studied.
}

These spaces coincide for all admitted
parameters with the hybrid spaces
\begin{\eq}   \label{2.31}
L^r \As (\rn)  = \La^{\vr} \As (\rn), \qquad -n \le \vr <\infty, \quad 0 <p<\infty, \quad r= \frac{\vr}{p}, 
\end{\eq}
$0<q \le \infty$, according to \cite[Definition 1.6, p.\,6]{T20} and the references given there to \cite[Definition 3.36, 
  pp.\,68--69]{T14} and \cite{T13}. \new{We  refer to \cite{YSY13} for some discussion of the different approaches.} There and in \cite{T13,T20} one finds also detailed (historical) references. In contrast to Definition~\ref{D2.5} we
admitted in Definition~\ref{D2.7} also $\vr \ge 0$ (in good agreement with $\tau \ge 0$ in \eqref{2.30}). But Proposition~\ref{P2.9} below makes clear why we concentrate later on in both Definitions~\ref{D2.5} and \ref{D2.7} on $-n <\vr <0$.
\end{remark}

\new{
\begin{remark}\label{N-Bt-spaces}
	We briefly compare the two scales $\mathcal{A}^s_{u,p,q}(\rn)$ and $ A^{s, \tau}_{p,q} (\rn)$. It is known  that 
	\begin{equation}
	{\cal N}^{s}_{u,p,q}(\rn) \hookrightarrow  {B}_{p,q}^{s,\tau}(\rn) \qquad \text{with}\qquad \tau={1}/{p}- {1}/{u}, 
	\label{N-BT-emb}
	\end{equation}
    cf. \cite[Corollary~3.3]{YSY10}.
	Moreover, this embedding is proper if $\tau>0$ and $q<\infty$. If $\tau=0$ or $q=\infty$, then both spaces coincide with each other, that is, $	\mathcal{N}^{s}_{u,p,\infty}(\rn)  =  B^{s,\frac{1}{p}- \frac{1}{u}}_{p,\infty}(\rn)$. As for the $F$-spaces, if $0\le \tau <{1}/{p}$, then
	\begin{equation}\label{fte}
	F^{s,\tau}_{p,q}(\rn)\, = \, \mathcal{E}^s_{u,p,q}(\rn)\quad\text{with }\quad \tau =
	{1}/{p}-{1}/{u}\, ,\quad 0 < p\le u < \infty\, ;
	\end{equation}
	cf. \cite[Corollary~3.3]{YSY10}, \cite[Theorem 1.1(ii), p.\,74]{SYY10}, and \cite[Theorem 6.35, p.\,794]{Saw18}.
  Moreover, if $p\in(0,\infty)$ and $q\in (0,\infty)$, then
	\begin{equation}\label{ftbt}
	F^{s,\, \frac{1}{p} }_{p\, ,\,q}(\rn) \, = \, F^{s}_{\infty,\,q}(\rn)\, = \, B^{s,\, \frac1q }_{q\, ,\,q}(\rn) \, ;
	\end{equation}
	cf. \cite[Propositions~3.4,~3.5]{Sic12} and \cite[Remark~10]{Sic13}.
\end{remark}

\begin{remark}\label{bmo-def}
Recall that the space $\bmo(\rn)$ is covered by the above scales: Consider the local (non-homogeneous) space of functions of bounded mean oscillation, $\bmo(\rn)$, consisting of all locally integrable
functions $\ f\in \Lloc(\rn) $ satisfying that
\begin{equation*}
 \left\| f \right\|_{\bmo}:=
\sup_{|Q|\leq 1}\; \frac{1}{|Q|} \int\limits_Q |f(x)-f_Q| \di x + \sup_{|Q|>
1}\; \frac{1}{|Q|} \int\limits_Q |f(x)| \di x<\infty,
\end{equation*}
where $ Q $ appearing in the above definition runs over all cubes in $\rn$, and $ f_Q $ denotes the mean value of $ f $ with
respect to $ Q$, namely, $ f_Q := \frac{1}{|Q|} \;\int_Q f(x)\di x$,
cf. \cite[2.2.2(viii)]{T83}. The space $\bmo(\rn)$ coincides with $F^{0}_{\infty, 2}(\rn)$,  cf. \cite[Thm.~2.5.8/2]{T83}. 
Hence the above result \eqref{ftbt} implies, in particular,
$\bmo(\rn)= F^{0}_{\infty,2}(\rn)= F^{0, 1/p}_{p, 2}(\rn)= {B^{0, 1/2}_{2, 2}(\rn)}$, $ 0<p<\infty$.
\end{remark}
}


Recall that 
$\Cc^\sigma (\rn) = B^\sigma_{\infty, \infty} (\rn)$, $\sigma \in \real$, are the above H\"{o}lder--Zygmund spaces in \eqref{2.16}.
A continuous embedding, indicated as usual by $\hra$, is called strict if the two spaces involved do not coincide.

\begin{proposition}   \label{P2.9}
Let $s\in \real$, $0<p<\infty$, $0<q\le \infty$ and $A\in \{B,F \}$. 
\bli
\item
Then
\begin{\eq}   \label{2.32}
\La_{-n} \As (\rn) = \La^{-n} \As (\rn) = \As (\rn).
\end{\eq}
\item 
  Let $\vr >0$. Then 
\begin{\eq}   \label{2.33}
\La^{\vr} \As (\rn) = \Cc^{s+ \frac{\vr}{p}} (\rn).
\end{\eq}
\item Let in addition $0<q_1 <q_2 <\infty$. Then
\begin{\eq}    \label{2.34}
B^s_{\infty,q_1} (\rn) \hra \La^0 B^s_{p,q_1} (\rn) \hra \La^0 B^s_{p,q_2} (\rn) \hra \La^0 B^s_{p,\infty} (\rn) = \Cc^s (\rn).
\end{\eq}
\eli
All embeddings are strict. Furthermore,
\begin{\eq}   \label{2.35}
\La^0 \Fs (\rn) = F^s_{\infty,q} (\rn)
\end{\eq}
and
\begin{\eq}   \label{2.36}
\La^0 F^s_{p,p} (\rn) = \La^0 B^s_{p,p} (\rn) = F^s_{\infty,p} (\rn).
\end{\eq}
\end{proposition}

\begin{proof}
Part (i) summarises \eqref{2.23}, \eqref{2.28}. Part (ii) is covered by \cite[Proposition 3.54, p.\,92]{T14} and \eqref{2.31}. The
first strict embedding in \eqref{2.34}, the last coincidence  in \eqref{2.34} and also \eqref{2.35} follow from  \eqref{2.29} with \eqref{2.30} together with \eqref{ftbt}, and may also be found in \cite[Proposition 1.18, pp.\,12--13]{T20}. The related proofs rely on wavelet arguments. This can also be used  to justify that the remaining (rather
obvious) embeddings in \eqref{2.34} are also strict. Finally  one obtains \eqref{2.36} from \eqref{2.26}, \eqref{2.27} with $q=p$
and \eqref{2.13}.
\end{proof}

\new{
\begin{remark}
  Note that \eqref{2.33}, using the coincidence 
  \eqref{2.29} with \eqref{2.30}, was already obtained in \cite{YY-4} as 
  $A^{s,\tau}_{p,q}(\rn) = B^{s+n(\tau-\frac1p)}_{\infty,\infty}(\rn)$ whenever  $\tau>\frac1p$ or $\tau=\frac1p$ and $q=\infty$. In general, for arbitrary $\tau\geq 0$, then $A^{s,\tau}_{p,q}(\rn) \hookrightarrow B^{s+n(\tau-\frac1p)}_{\infty,\infty}(\rn)$, cf. \cite[Proposition~2.6]{YSY10}.
\end{remark}
}


\begin{remark}   \label{R2.10}
It makes sense to extend the definition of the spaces $\La^{\vr} \Bs (\rn)$ to $p=\infty$. But it follows from \eqref{2.26} that the
related spaces coincide with $B^s_{\infty,q} (\rn)$, $s\in \real$, $0<q \le \infty$. The situation is somewhat different as far as
the corresponding hybrid spaces
\begin{\eq}   \label{2.37}
L^r B^s_{\infty,q} (\rn), \qquad r \ge 0, \quad s\in \real, \quad 0<q \le \infty,
\end{\eq}
are concerned. For this purpose one has first to replace the factor $2^{\frac{J}{p} (n + \vr)}$ in \eqref{2.26} by $2^{J(\frac{n}{p}
+r)}$ and to choose  afterwards $p=\infty$. This gives the factor $2^{Jr}$. If $r=0$, then one has again $B^s_{\infty,q} (\rn)$. If
$r>0$, then it  follows from \cite[Proposition 3.54, p.\,92]{T14} that the corresponding spaces coincide with $\Cc^{s+r} (\rn)$.
\end{remark}

\subsection{The $\vr$--clans}    \label{S2.2}
\new{Based on the above definitions and coincidences it is obvious that the 
classical spaces 
\begin{\eq}   \label{2.38}
\As (\rn) \quad \text{with} \quad A \in \{B, F \}, \quad s\in \real, \quad \text{and} \quad 0<p,q \le \infty
\end{\eq}
(including $F^s_{\infty,q} (\rn)$) have to be complemented by 
\begin{\eq}   \label{2.39}
\La^{\vr} \As (\rn) \quad \text{and} \quad\La_{\vr} \As (\rn) \quad \text{with} \quad A\in \{B,F \},
\end{\eq}
\begin{\eq}   \label{2.40}
s\in \real, \quad 0<p<\infty, \quad 0<q \le \infty, \quad -n <\vr <0,
\end{\eq}
on the one hand, and
\begin{\eq}   \label{2.41}
\La^0 \Bs (\rn) \quad \text{with} \quad s\in \real \quad \text{and} \quad 0<p,q <\infty,
\end{\eq}
on the other hand, as far as the most prominent scales of Morrey smoothness spaces are concerned. This can be seen from Definitions~\ref{D2.5} and \ref{D2.7} together with \eqref{2.24}, \eqref{2.25} for the spaces of type $\MA(\rn)$, from \eqref{2.29} with \eqref{2.30} for the spaces of type $\At(\rn)$, and from Remark~\ref{R2.10} for the hybrid spaces. According to Proposition~\ref{P2.9} we regain the classical spaces \eqref{2.38} from $\La^{\vr} \As (\rn)$  and $\La_{\vr} \As (\rn)$ for $\vr=-n$ and $\vr=0$, respectively.}  
For the somewhat peculiar spaces in \eqref{2.41} one has \eqref{2.34}, \eqref{2.36} indicating that they are closely related to the spaces $A^s_{\infty,q} (\rn)$. We return to them occasionally, but concentrate  otherwise on the spaces in 
\eqref{2.39}, \eqref{2.40}. One may ask how these spaces are related to each other. This attracted some attention in the {literature}
and will be discussed later on in some details. Above all we rely on the coincidence
\begin{\eq}   \label{2.42}
\La_{\vr} \Fs (\rn) = \La^{\vr} \Fs (\rn), \quad s\in \real, \ 0<p<\infty, \ 0<q\le \infty, \ -n \le \vr <0,
\end{\eq}
extending \eqref{2.32} to the related $F$--spaces with $-n <\vr <0$. This follows from \eqref{2.25} and \eqref{2.29}, \eqref{2.30},
inserted in the well--known coincidence \eqref{fte}.
 In other words, one has only one $\La_{\vr}F =\La^{\vr}F$ scale in
\eqref{2.39}, \eqref{2.40}, but two $\La_{\vr} B$ and $\La^{\vr} B$ scales.

\begin{definition}   \label{D2.11}
Let $n\in \nat$.
\bli 
\item The $n$--clan consists of the spaces according to  \eqref{2.38}.
\item
For $-n < \vr <0$ the $\vr$--clan $\rhoAs(\rn)$ consists of the three families
\begin{\eq}   \label{2.44}
\La_{\vr} \Bs(\rn), \quad \La^{\vr} \Bs (\rn) \quad \text{and} \quad \La_{\vr} \Fs (\rn) = \La^{\vr} \Fs (\rn)
\end{\eq}
with
\begin{\eq}   \label{2.45} 
s\in \real, \quad   0<p<\infty, \quad 0<q \le \infty.
\end{\eq}
\item The $0$--clan  consists of the spaces according to \eqref{2.41}.
\eli
\end{definition}

For convenience, we have sketched the situation for the new spaces defined above in Remark~\ref{R-schema} below.


\begin{remark}    \label{R2.12}
As already said we are mainly interested in the $\vr$--clans in part (ii) of the above definition. The theory of the classical spaces
in \eqref{2.38}, now called the $n$--clan, has been developed over decades. This is well reflected in many books, including 
\cite{T83}, \cite{T92}, \cite{T06} and \cite{T20}. In particular, \cite{T20} may be considered as a collection of final answers about
distinguished problems for these spaces, including $F^s_{\infty,q} (\rn)$. This can be taken as a guide asking the same questions for
the larger community covered by the above definition. It is the main aim of this paper to make clear that the above classification
of the Morrey smoothness spaces in common use produces transparent answers having the same impressive beauty as their ancestors in
the $n$--clan. First examples had already been described in the Introduction, \eqref{1.7} compared with \eqref{1.6}, \eqref{1.10}
compared with \eqref{1.9}, but also \eqref{1.12} compared with \eqref{1.11}. We formulate these expectations as follows.
\end{remark}

\new{
\begin{slope}\label{slope_rules}
  Let $n\in \nat$ and $-n <\vr <0$.
  \bli
\item
  An extension of an adequate assertion for the spaces $\As (\rn)$ of the $n$--clan to the spaces \eqref{2.44}, \eqref{2.45} of the
$\vr$--clan  $\rhoAs (\rn)$
is subject to the {\em Slope--$1$--Rule} if it depends on $\frac{1}{p}\min(|\vr|,1)$ $($instead of $\frac{1}{p})$.
\item
An extension of an adequate assertion for the spaces $\As (\rn)$ of the $n$--clan to the spaces \eqref{2.44}, \eqref{2.45} of the
$\vr$--clan  $\rhoAs (\rn)$
is subject to the {\em Slope--$n$--Rule} if it depends on $\frac{|\vr|}{p}$ $($instead of $\frac{n}{p})$.
\eli
\end{slope}
}

\ignore{\begin{itemize}
\item[] {\bfseries Slope Rules}. {\em Let $n\in \nat$ and $-n <\vr <0$.
An extension of an adequate assertion for the spaces $\As (\rn)$ of the $n$--clan to the spaces \eqref{2.44}, \eqref{2.45} of the
$\vr$--clan  $\rhoAs (\rn)$
is subject to the Slope--$1$--Rule if it depends on $\frac{1}{p}\min(|\vr|,1)$} ({\em instead of $\frac{1}{p}$}) {\em 
and subject to the Slope--$n$--Rule if it depends on $\frac{|\vr|}{p}$} ({\em instead of $\frac{n}{p}$}).
\end{itemize}
}

\begin{remark}   \label{R2.13}
These slope rules are illustrated by the examples in the Introduction. We call $\vr$ the slope parameter although  not $\vr$ itself 
but $|\vr|$ is the slope of the corresponding  lines in the $\big( \frac{1}{p},s \big)$--diagram, \new{as it can be seen exemplarily in Figure~\ref{fig-1a} on p.~\pageref{fig-1a} below}. This suggest to replace $-n <\vr <0$
by $0<\vr <n$. But we stick at the usual habit that larger values of a fixed parameter produce smaller spaces. This applies to the
smoothness $s$, but also to the integrability $p$ for related spaces on bounded domains, \new{but not to the fine index $q$}. It is also in good agreement  with the
parameters $\tau$ in \eqref{2.29}, \eqref{2.30}, $r$ in \eqref{2.31} and $\vr >0$ in \eqref{2.33}.
\end{remark}

\subsection{Relations and coincidences}    \label{S2.3}
The already indicated sharp embeddings in \eqref{1.7} and \eqref{1.10} (as far as breaking lines in the 
$\big(\frac{1}{p},s\big)$--diagrams \new{in Figure~\ref{fig-1a} on p.~\pageref{fig-1a} and Figure~\ref{fig-1b} on p.~\pageref{fig-1b}} are concerned) suggest that, in general, spaces belonging to different $\vr$--clans do not 
coincide. At least it seems to be reasonable to discuss relations and coincidences within a fixed $\vr$--clan. Again one can take the
$n$--clan, consisting of the spaces $\As (\rn)$ in \eqref{2.38} as a guide. We repeat the final answer as it may be found in \cite[Section~2.3.9, pp. 61-62]{T83} (under the additional assumption $p_1,p_2<\infty$ in the $F$-case), and in \cite[Theorem 2.10, p.\,28]{T20} in the above complete form.

\begin{proposition}   \label{P2.14}
Let $n\in \nat$.
Let $0<p_1, p_2, q_1, q_2 \le \infty$ and $s_1 \in \real$, $s_2 \in \real$. Then
\begin{\eq}   \label{2.46}
B^{s_1}_{p_1, q_1} (\rn) = B^{s_2}_{p_2, q_2} (\rn) \quad \text{if, and only if, $s_1 = s_2$, $p_1 = p_2$, $q_1 = q_2$},
\end{\eq}
\begin{\eq}  \label{2.47}
F^{s_1}_{p_1, q_1} (\rn) = F^{s_2}_{p_2, q_2} (\rn) \quad \text{if, and only if, $s_1 = s_2$, $p_1 = p_2$, $q_1 = q_2$},
\end{\eq}
and
\begin{\eq}   \label{2.48}
F^{s_1}_{p_1, q_1} (\rn) = B^{s_2}_{p_2, q_2} (\rn) \quad \text{if, and only if, $s_1 = s_2$, $p_1 = p_2 =q_1 = q_2$}.
\end{\eq}
\end{proposition}

\begin{remark}   \label{R2.15}
By {\cite[Theorem 2.9, p.\,26]{T20}} one has for  $s\in \real$ and $0<p,q,u,v \le \infty$ that
\begin{\eq}  \label{2.49}
B^s_{p,u} (\rn) \hra \Fs (\rn) \hra B^s_{p,v} (\rn)
\end{\eq}
if, and only if,
\begin{\eq}   \label{2.50}
0<u \le \min (p,q) \qquad \text{and} \qquad \max (p,q) \le v \le \infty.
\end{\eq}
If $p<\infty$, then this is just the well-known result in \cite[Theorem~3.1.1]{SiT95}.
\end{remark}

One may ask for counterparts within the fixed $\vr$--clan $\rhoAs (\rn)$. 
Instead of the two families $\Bs (\rn)$ and $\Fs (\rn)$ for the $n$--clan one
has now to cope with the three families in \eqref{2.44}. By \eqref{2.21} and \eqref{2.26} the corresponding $B$--spaces are 
quasi--normed by
\begin{\eq}   \label{2.51}
\| f \, | \La_{\vr} \Bs (\rn) \| = \Big( \sum^\infty_{j=0} 2^{jsq} \sup_{J \in \ganz, M \in \zn} 2^{\frac{J}{p} (n+\vr)q}
\big\| \big( \vp_j \wh{f} \big)^\vee | L_p (Q_{J,M} ) \big\|^q \Big)^{1/q}
\end{\eq}
and
\begin{\eq}   \label{2.52}
\| f \, | \La^{\vr} \Bs (\rn) \| = \sup_{J\in \ganz, M\in \zn} 2^{\frac{J}{p} (n+ \vr)} \Big( \sum_{j \ge J^+} 2^{jsq} \big\|
\big( \vp_j \wh{f}\, \big)^\vee \, | L_p (Q_{J,M} ) \big\|^q \Big)^{1/q}
\end{\eq}
(usual modification if $q=\infty$). These explicit versions will be of some use, now and later on. We collect  a few more or less
known assertions and indicate how some of them can be proved based on properties which will be justified later on. Recall that the
continuous embedding $\hra$ is called strict if the two spaces in question do not coincide.

\begin{theorem}   \label{T2.16}
Let $n\in \nat$, $-n <\vr <0$ and \new{$s, s_1, s_2 \in\real$}. Let $0<p,p_1, p_2 <\infty$ and $0<q, q_1, q_2 \le \infty$. 
\bli
\item
The spaces $\La_{\vr} B^{s_1}_{p_1,q_1} (\rn)$ and $\La_{\vr} B^{s_2}_{p_2,q_2} (\rn)$ coincide if, and only if, $p_1 = p_2$,
$q_1 = q_2$, $s_1 = s_2$. Similarly the spaces $\La_{\vr} F^{s_1}_{p_1,q_1} (\rn)$ and $\La_{\vr} F^{s_2}_{p_2,q_2} (\rn)$ coincide if, and only if, $p_1 = p_2$, $q_1 = q_2$, $s_1 = s_2$. 
\item The spaces $\La_{\vr} B^{s_1}_{p_1,q_1} (\rn)$ and $\La_{\vr} F^{s_2}_{p_2,q_2} (\rn)$ do not coincide. Furthermore,
\begin{\eq}   \label{2.53}
\La_{\vr} B^s_{p, \min(p,q)} (\rn) \hra \La_{\vr} \Fs (\rn) \hra \La_{\vr} B^s_{p,\infty} (\rn).
\end{\eq}
\item
The spaces $\La^{\vr} F^{s_1}_{p_1, q_1} (\rn)$ and $\La^{\vr} F^{s_2}_{p_2, q_2} (\rn)$ coincide if, and only if, $p_1 = p_2$, $q_1 =q_2$, $s_1 = s_2$. Furthermore,
\begin{\eq}   \label{2.54}
\La^{\vr} B^s_{p, \min(p,q)} (\rn) \hra \La^{\vr} \Fs (\rn) \hra \La^{\vr} B^s_{p, \max(p,q)} (\rn)
\end{\eq}
and
\begin{\eq}   \label{2.55}
\La^{\vr} B^s_{p,p} (\rn) = \La^{\vr} F^s_{p,p} (\rn).
\end{\eq}
\item If $q<\infty$, then
\begin{\eq}   \label{2.56}
\La_{\vr} \Bs (\rn) \hra \La^{\vr} \Bs (\rn)
\end{\eq}
is a strict embedding. Furthermore,
\begin{\eq}   \label{2.57}
\La_{\vr} B^s_{p,\infty} (\rn) = \La^{\vr} B^s_{p,\infty} (\rn)
\end{\eq}
and
\begin{\eq}   \label{2.58}
\La_{\vr} \Fs (\rn) = \La^{\vr} \Fs (\rn).
\end{\eq}
\eli
\end{theorem}

\begin{proof}
{\em Step 1.} The parts (i) and (ii)  are covered by \cite[Proposition 1.3, Theorem 1.7, pp.\,96, 98]{Saw08}. The property 
\eqref{2.58} has already been mentioned in \eqref{2.42}. There one finds also related references. Then it follows from part (i) that
the spaces $\La^{\vr} F^{s_1}_{p_1, q_1} (\rn)$ and $\La^{\vr} F^{s_2}_{p_2, q_2} (\rn)$ coincide only trivially. The well--known
embeddings \eqref{2.54} and also \eqref{2.55} can be obtained by the same arguments as for the classical spaces $\As (\rn)$, based on
the Definitions~\ref{D2.1}, \ref{D2.7} and also \eqref{2.49}, \eqref{2.50}. \new{We refer to \cite[Proposition~2.1]{YSY10}.}
\cm
{\em Step 2.} The embedding \eqref{2.56} follows from \eqref{2.51}, \eqref{2.52}. According to \cite[Proposition 3.1, p.\,121]{Sic12} one can replace $j \ge J^+$
in \eqref{2.52} by $j \in \no$ (equivalent quasi--norms). Then \eqref{2.51} and \eqref{2.52} with $q=\infty$
coincide. This proves \eqref{2.57}. It remains to prove that the embedding \eqref{2.56} is strict if $q< \infty$. This follows 
from Theorem~\ref{T5.18} below if, in addition, $s = \frac{|\vr|}{p} -n$, and can be extended to all $s\in \real$
by lifting according to Theorem~\ref{T3.8} below.
\end{proof}


\begin{remark}    \label{R2.38}
The assertions of the above theorem are more or less known. But they are somewhat scattered in the literature and not related to the
$\vr$--clans. Furthermore we wanted to show that Theorem~\ref{T5.18} can be used to justify that the embedding \eqref{2.56}
is strict. As far as the coincidences of the spaces $\La_{\vr} \As (\rn)$ are concerned we refer the reader again to 
\cite{Saw08} for a more general version. In particular there is no need  to fix $\vr$ from the very beginning. The homogeneous
counterpart of part (iv) may be found in \cite[Theorem 1.1, p.\,74]{SYY10}. The justification of \eqref{2.57} for the inhomogeneous
version requires that one can replace the natural condition $j \ge J^+$ in \eqref{2.52}
by $j \in \no$. This has been used in Step 2 with a reference to \cite{Sic12}.
\new{We would like to point to the difference between \eqref{2.53} and \eqref{2.54} concerning the last embedding: while in case of spaces $\La^\vr\As(\rn)$ the well known behaviour as described in Remark~\ref{R2.15} is  preserved in view of the fine index $q$, this is no longer true for spaces $\La_\vr\As(\rn)$: here $q=\infty$ cannot be improved -- unlike
in case of $\vr=-n$, cf.~\cite{Saw08}. }
\end{remark} 

\new{
  \begin{remark}\label{R-schema}
In view of the coincidences and embeddings within the $B$- and $F$-scales according to Theorem~\ref{T2.16} we may illustrate the situation for the new spaces $\rhoAs(\rn)$ with $-n<\vr<0$ defined in Definition~\ref{D2.11} as follows: 
\begin{align*}
  \begin{array}{c} \rhoAs(\rn) \\ \overbrace{\hspace*{25em}}\\
      \La_{\vr} \As(\rn) \ \begin{array}{c}\nearrow \\ \searrow \end{array}
\begin{array}{ccc}
 \La_{\vr} \Bs(\rn) & \hookrightarrow &  \La^{\vr} \Bs (\rn)   \\[2ex]
 \La_{\vr} \Fs(\rn) & = &  \La^{\vr} \Fs (\rn) \end{array} \begin{array}{c}\nwarrow \\ \swarrow \end{array} \ 
 \La^{\vr} \As(\rn). 
  \end{array}
  \end{align*}
\end{remark}
}%
\smallskip~

By Proposition~\ref{P2.14} there are no other coincidences within the $n$--clan, the classical spaces  in \eqref{2.38}, than the
obvious ones. This suggests to ask for coincidences within a fixed $\vr$--clan as introduced in Definition~\ref{D2.11}(ii) or within
all spaces covered by Definition~\ref{D2.11} for fixed $n\in \nat$. We fix the expected ideal outcome that within a fixed $\vr$--clan
there are no further coincidences than those ones listed in Theorem~\ref{T2.16}.

\begin{conjecture}    \label{C2.18}
Let $n\in \nat$, $-n <\vr <0$ and \new{$s, s_1, s_2 \in\real$}. Let $0<p, p_1, p_2 <\infty$ and $0<q, q_1, q_2 \le \infty$.
\bli
\item
  Part {\em (i)} of Theorem~\ref{T2.16} remains valid if one replaces $\La_{\vr}$ by $\La^{\vr}$.
\item There are no further coincidences for the spaces belonging to the same $\vr$--clan than \eqref{2.55}, \eqref{2.57} and
\eqref{2.58}.
\eli
\end{conjecture}

\begin{remark}   \label{R2.19}
We did not deal with coincidences, diversities and strict embeddings within a fixed $\vr$--clan than those ones stated in Theorem
\ref{T2.16}. The above conjecture may be considered as a request to pay more attention to questions of this type. One could try to 
use wavelet expansions as described in Section~\ref{S3.1} for all spaces covered by Definition~\ref{D2.7} to prove or disprove the
above conjecture or parts of it. Then $\vr <0$ has to play a decisive r\^ole. If $\vr \ge 0$, then Proposition~\ref{P2.9} shows that 
there are many coincidences. For example it follows from \eqref{2.33} that
\begin{\eq}  \label{2.59}
\La^{\vr} F^{s_1}_{p_1,q_1} (\rn) = \La^{\vr} F^{s_2}_{p_2,q_2} (\rn) \quad \text{if} \quad \vr >0 \quad \text{and} \quad s_1 + 
\frac{\vr}{p_1} = s_2 + \frac{\vr}{p_2}.
\end{\eq}
One could also discuss the above conjecture in the limelight of distinguished properties of the above spaces. For example, if $-1 <
\vr <0$, then it follows from Theorem~\ref{T4.12} below that spaces belonging to the same $\vr$--clan cannot coincide if their
differential dimensions $s - \frac{|\vr|}{p}$ are different. It is an even more challenging request to extend Proposition~\ref{P2.14}
from the $n$--clan to all clans covered  by Definition~\ref{D2.11}. Then part (iii) of Proposition~\ref{P2.9} shows that there are
further coincidences.
\end{remark}

\section{Tools}   \label{S3}
\subsection{Wavelets}    \label{S3.1}
Some proofs below rely on wavelet arguments. We collect what we need restricting us to the bare minimum. We follow \cite[Section 1.2,
  pp.\,7--12]{T20} where one finds further (more comprehensive) assertions, explanations and references. \new{Let us, in particular, refer to the standard monographs \cite{Dau92,Mal98,Mey92,Woj97}.}

As usual, $C^{u} (\real)$ with $u\in
\nat$ collects all bounded complex-valued continuous functions on $\real$ having continuous bounded derivatives up to order $u$ inclusively. Let
\begin{\eq}   \label{3.1}
\psi_F \in C^{u} (\real), \qquad \psi_M \in C^{u} (\real), \qquad u \in \nat,
\end{\eq}
be real compactly supported Daubechies wavelets with
\begin{\eq}   \label{3.2}
\int_{\real} \psi_M (t) \, t^v \, \di t =0 \qquad \text{for all $v\in \no$ with $v<u$.}
\end{\eq}
Recall that $\psi_F$ is called the {\em scaling function} (father wavelet) and $\psi_M$ the associated {\em wavelet} 
(mother wavelet). We extend these wavelets from $\real$ to $\rn$ by the usual multiresolution procedure. Let $n\in \nat$, and let
\begin{\eq}   \label{3.3}
G = (G_1, \ldots, G_n) \in G^0 = \{F,M \}^n
\end{\eq}
which means that $G_r$ is either $F$ or $M$. Furthermore, let
\begin{\eq}   \label{3.4}
G= (G_1, \ldots, G_n) \in G^* = G^j = \{F, M \}^{n*}, \qquad j \in \nat,
\end{\eq}
which means that $G_r$ is either $F$ or $M$, where $*$ indicates that at least one of the components of $G$ must be an $M$. Hence $G^0$ has $2^n$ elements, whereas $G^j$ with $j\in \nat$ and $G^*$ have $2^n -1$ elements. Let
\begin{\eq}   \label{3.5}
\psi^j_{G,m} (x) = \prod^n_{l=1} \psi_{G_l} \big(2^j x_l -m_l \big), \qquad G\in G^j, \quad m \in \zn, \quad x\in \rn,
\end{\eq}
where (now) $j \in \no$. We always assume that $\psi_F$ and $\psi_M$ in \eqref{3.1} have $L_2$--norm 1. Then 
\begin{\eq}   \label{3.6}
\big\{ 2^{jn/2} \psi^j_{G,m}: \ j \in \no, \ G\in G^j, \ m \in \zn \big\}
\end{\eq}
is an  orthonormal basis in $L_2 (\rn)$ (for any $u\in \nat$) and
\begin{\eq}   \label{3.7}
f = \sum^\infty_{j=0} \sum_{G \in G^j} \sum_{m \in \zn} \lambda^{j,G}_m \, \psi^j_{G,m}
\end{\eq}
with
\begin{\eq}   \label{3.8}
\lambda^{j,G}_m = \lambda^{j,G}_m (f) = 2^{jn} \int_{\rn} f(x) \, \psi^j_{G,m} (x) \, \di x = 2^{jn} \big(f, \psi^j_{G,m} \big)
\end{\eq}
is the corresponding expansion. Recall that $f\in S'(\rn)$ is an element of $\Bs (\rn)$ where $s\in \real$ and $0<p,q \le \infty$
if, and only if, it can be expanded by \eqref{3.7}, \eqref{3.8} (with sufficiently large $u$ in \eqref{3.1}, \eqref{3.2}) such that
\begin{\eq}  \label{3.9}
\|f \, | \Bs (\rn) \| \sim
\Big( \sum^\infty_{j=0} 2^{j(s- \frac{n}{p})q} \sum_{G \in G^j} \Big( \sum_{m \in \zn} |\lambda^{j,G}_m (f)|^p \Big)^{q/p} \Big)^{1/q}
\end{\eq}
(usual modification if $\max(p,q) =\infty$) is finite (equivalent quasi--norms). Explanations and references may be found in 
\cite[Section 1.2.1, pp.\,7--10]{T20} and \cite[Section 3.2.3, pp.\,51--54]{T14}. The  related counterparts for the spaces $\Fs (\rn)$
do not play a substantial r\^ole here. This wavelet expansion can be extended to the $B$--spaces of the $\vr$--clan in \eqref{2.44} as
follows. 

\begin{proposition}   \label{P3.1}
Let $n\in \nat$, $-n < \vr <0$, $s\in \real$, $0<p<\infty$ and $0<q \le \infty$. Then $f\in S'(\rn)$ is an element of $\La^{\vr} \Bs
(\rn)$ if, and only if, it can be expanded by \eqref{3.7}, \eqref{3.8} $($with sufficiently large $u$ in \eqref{3.1}, \eqref{3.2}$)$
such that
\begin{\eq}    \label{3.10}
\begin{aligned}
&\| f \, | \La^{\vr} \Bs (\rn) \| \\
&\sim \sup_{J\in \ganz, M\in \zn} 2^{\frac{J}{p} (n+ \vr)} \bigg( \sum_{j \ge J^+} 2^{j(s- \frac{n}{p})q} 
\Big( \sum_{\substack{m:Q_{j,m}
\subset Q_{J,M}, \\ G\in G^j}} \big| \lambda^{j,G}_m (f) \big|^p \Big)^{q/p} \bigg)^{1/q}
\end{aligned}
\end{\eq}
$($usual modification if $q=\infty)$ is finite $($equivalent quasi--norms$)$. Similarly, $f\in S'(\rn)$ is an element of $\La_{\vr}
\Bs (\rn)$ if, and only if, it can be expanded by \eqref{3.7}, \eqref{3.8} $($with sufficiently large  $u$ in \eqref{3.1}, 
\eqref{3.2}$)$ such that 
\begin{\eq}    \label{3.11}
\begin{aligned}
&\| f \, | \La_{\vr} \Bs (\rn) \| \\
&\sim \bigg( \sum^\infty_{j=0} 2^{j(s- \frac{n}{p})q} \sup_{J\in \ganz, M \in \zn} 2^{\frac{J}{p} (n+\vr)q}
\Big( \sum_{\substack{m:Q_{j,m}
\subset Q_{J,M}, \\ G\in G^j}} \big| \lambda^{j,G}_m (f) \big|^p \Big)^{q/p} \bigg)^{1/q}
\end{aligned}
\end{\eq}
$($usual modification if $q=\infty)$ is finite $($equivalent quasi--norms$)$.
\end{proposition}

\begin{remark}   \label{R3.2}
The assertion about the spaces $\La^{\vr} \Bs (\rn)$ is covered by \cite[Proposition 1.16, p.\,11]{T20} with a reference
to \cite[Theorem 3.26, p.\,64]{T14} and \eqref{2.31}. It remains valid with $\vr =0$ for the spaces $\La^0 \Bs (\rn)$ in \eqref{2.41}.
Its counterpart for the spaces $\La_{\vr} \Bs (\rn)$ goes back to \cite[Theorem 1.5, p.\,97]{Saw08}, appropriately reformulated, \new{see also \cite{MR-1}}. 
There are related  $F$--counterparts. They play only a marginal r\^ole in this paper. If needed we give precise references. 
  The above formulations are sufficient for our purpose. But they are a little bit
sloppy from a technical point of view. More precisely, one justifies first that the expansion \eqref{3.7} converges unconditionally 
in $S'(\rn)$ if the coefficients $\{\lambda^{j,G}_m \}$ belong to the sequence spaces in \eqref{3.9}, \eqref{3.10}, \eqref{3.11} with
$\lambda^{j,G}_m$ in place of $\lambda^{j,G}_m (f)$. Afterwards one ensures that $f$ belongs to the related spaces with the uniquely
determined coefficients $\lambda^{j,G}_m = \lambda^{j,G}_m (f)$ in \eqref{3.8}. One may consult {\cite[Theorem 3.26, p.\,64]{T14}}  and the 
(technical) explanations given there. We return to this point later on in Section~\ref{S6.1} in connection with Haar wavelets.
\end{remark}

\subsection{Interpolation}   \label{S3.2}

{If one wishes to study interpolation of the Morrey spaces in \eqref{2.20}, or for the Morrey smoothness spaces in \eqref{2.24}, \eqref{2.25} and \eqref{2.29}, \eqref{2.30}, then one finds interesting and deep results in the literature, we refer to \cite{KY,Maz03,SawTan-2} and recommend \cite[Section~2]{Sic13} for a more detailed discussion. Such interpolation formulas fit well in the above concept with $\vr$ as the decisive parameter.} But this is not our topic {here in its full generality}. We concentrate on very few assertions which will be of some use later on. Let $(A_0, 
A_1)_{\theta,q}$ with $0<\theta <1$ and $0<q \le \infty$ be the classical real interpolation for interpolation couples $\{A_0, A_1 \}$
of quasi--Banach spaces, we refer to the standard literature \cite{BL,T78}.  First we deal with the real interpolation of the three families belonging to a
fixed $\vr$--clan according to Definition~\ref{D2.11}(ii). 

\begin{proposition}   \label{P3.3}
Let  $n\in \nat$ and $-n < \vr <0$. Let $0<p<\infty$ and $0<q, q_1, q_2 \le \infty$. Let \new{ $0<\theta<1$ and}
\begin{\eq}   \label{3.12}
-\infty <s_1 <s< s_2 <\infty, \qquad s= (1-\theta)s_1 + \theta s_2.
\end{\eq}
Then
\begin{\eq}    \label{3.13}
\La_{\vr} \Bs (\rn) = \big( \La_{\vr} B^{s_1}_{p, q_1} (\rn), \La_{\vr} B^{s_2}_{p, q_2} (\rn) \big)_{\theta,q}.
\end{\eq}
\end{proposition}

\begin{remark}   \label{R3.4}
This assertion goes back to \cite[Theorem 2.2, p.\,91]{Sic13}. As mentioned there one can replace $\La_{\vr} B^{s_1}_{p, q_1} (\rn)$
and $\La_{\vr} B^{s_2}_{p, q_2} (\rn)$ in \eqref{3.13} independently by $\La_{\vr} F^{s_1}_{p, q_1} (\rn)$ and/or $\La_{\vr} F^{s_2}_{p, q_2} (\rn)$. This follows immediately from \eqref{3.13} and \eqref{2.53}. Using in addition \eqref{2.58} one obtains
\begin{\eq}   \label{3.14}
\La_{\vr} \Bs (\rn) = \big( \La^{\vr} F^{s_1}_{p,q_1} (\rn), \La^{\vr} F^{s_2}_{p, q_2} (\rn) \big)_{\theta,q}.
\end{\eq}
This relation will be of some service for us later on. It shows that properties already proved for the spaces $\La^{\vr} \As (\rn)$
can be transferred to the third family $\La_{\vr} \Bs (\rn)$ of the $\vr$--clan (if interpolation can be applied).
\end{remark}

Quite obviously, \eqref{3.13}, \eqref{3.14} extends the well--known real interpolation for the classical spaces in \eqref{2.38} from
the $n$--clan to the $\vr$--clan. One may ask what happens with other interpolation methods, in particular diverse types of complex
interpolation. This attracted some attention in the literature, both for the Morrey spaces in \eqref{2.20} but also for their smooth
generalisations in \eqref{2.24}, \eqref{2.25} and \eqref{2.29}, \eqref{2.30}. Satisfactory results can only be obtained within a fixed
$\vr$--clan. As far as the complex interpolation of the Morrey spaces $\La^{\vr}_p (\rn)$ are concerned one may consult \cite{Lem13},
\cite{Lem14}, \new{see also \cite{HaSa-2,HaSa-3,MaSa} for some more recent contributions}. The complex interpolation especially of $\La_{\vr} \Fs (\rn)$ and their restrictions $\La_{\vr} \Fs (\Om)$ to bounded
Lipschitz domains $\Om$ has been studied in \cite{HNaS,HNoS,YYZ} and, quite 
recently, in \cite{ZHS20,Z21}. There one finds also further discussions and references. In any 
case it seems to be a rather tricky undertaking. Fortunately enough there is a very efficient interpolation method in the context of
the above spaces which will be of some service. We give a description.

The so--called $\pm$--method goes back to \cite{GuP77}. A description and further references  may be found in \cite[Section 2.2,
p.\,1842]{YSY15}. It applies to arbitrary interpolation couples $\{A_1, A_2 \}$ of quasi--Banach spaces and produces interpolation
spaces $\langle A_1, A_2, \theta \rangle$, $0<\theta <1$, having the desired interpolation property. We concentrate on the $\vr$--clan
according to Definition~\ref{D2.11}(ii) consisting of the three families according to \eqref{2.44}, \eqref{2.45}.

\begin{theorem}   \label{T3.5}
Let $n\in \nat$, $-n <\vr <0$, $s_1 \in \real$, $s_2 \in \real$. Let $0<p_1, p_2 <\infty$ and $0<q_1, q_2 \le \infty$. Let $0<\theta <1$ and
\begin{\eq}   \label{3.15}
s= (1- \theta)s_1 + \theta s_2, \qquad \frac{1}{p} = \frac{1-\theta}{p_1} + \frac{\theta}{p_2}, \qquad 
\frac{1}{q} = \frac{1-\theta}{q_1} + \frac{\theta}{q_2}.
\end{\eq}
Then
\begin{\eq}   \label{3.16}
\big\langle \La_{\vr} B^{s_1}_{p_1, q_1} (\rn), \La_{\vr} B^{s_2}_{p_2, q_2} (\rn), \theta \big\rangle = \La_{\vr} \Bs (\rn),
\end{\eq}
\begin{\eq}   \label{3.17}
\big\langle \La^{\vr} B^{s_1}_{p_1, q_1} (\rn), \La^{\vr} B^{s_2}_{p_2, q_2} (\rn), \theta \big\rangle = \La^{\vr} \Bs (\rn),
\end{\eq}
and
\begin{\eq}   \label{3.18}
\big\langle \La_{\vr} F^{s_1}_{p_1, q_1} (\rn), \La_{\vr} F^{s_2}_{p_2, q_2} (\rn), \theta \big\rangle = \La_{\vr} \Fs (\rn).
\end{\eq}
\end{theorem}

\begin{proof}
This remarkable assertion is covered by \cite[Theorem 2.12, p.\,1843]{YSY15} reformulated according to \eqref{2.24}, \eqref{2.25} and
\eqref{2.29}, \eqref{2.30}.
\end{proof}

\begin{remark}   \label{R3.6}
According to \cite{YSY15} one can extend
the above assertions to the $n$--clan in Definition~\ref{D2.11}(i) consisting of the classical
spaces in \eqref{2.38} with the following outcome.
Let $n\in \nat$, $s_1 \in \real$, $s_2 \in \real$ and $0<p_1, p_2, q_1, q_2 \le \infty$. Let $0<\theta <1$ and $s,p,q$ be as in 
\eqref{3.15}. Then
\begin{\eq}   \label{3.19}
\big\langle  B^{s_1}_{p_1, q_1} (\rn),  B^{s_2}_{p_2, q_2} (\rn), \theta \big\rangle = \Bs (\rn)
\end{\eq}
and
\begin{\eq}   \label{3.20}
\big\langle  F^{s_1}_{p_1, q_1} (\rn),  F^{s_2}_{p_2, q_2} (\rn), \theta \big\rangle =  \Fs (\rn).
\end{\eq}
The remarkable incorporation of $p_1 = \infty$ and/or $p_2 = \infty$ in \eqref{3.20} goes back to \cite[Theorem 8.5, pp.\,98, 
134]{FrJ90}. Otherwise \eqref{3.16}--\eqref{3.18} compared with \eqref{3.19}, \eqref{3.20} fit in the
Slope--$n$--Rule as formulated in \ignore{Section~\ref{S2.2}} \new{Slope Rules~\ref{slope_rules}(ii)}. On the other hand nothing seems to be known if one wishes to interpolate spaces
belonging to different $\vr$--clans.
\end{remark} 

\begin{remark}    \label{R3.7}
It has been shown in \cite[Corollary 2.14, p.\,1844]{YSY15} that also the basic spaces underlying the above smoothness spaces 
interpolate as expected. Let $-n \le \vr <0$,
\begin{\eq}   \label{3.21}
0<p_1 \le p_2 <\infty, \qquad 0<\theta <1, \qquad \frac{1}{p} = \frac{1-\theta}{p_1} + \frac{\theta}{p_2}.
\end{\eq}
Then
\begin{\eq}   \label{3.22}
\big\langle \La^{\vr}_{p_1} (\rn), \La^{\vr}_{p_2} (\rn), \theta \big\rangle = \La^{\vr}_p (\rn)
\end{\eq}
extends
\begin{\eq}   \label{3.23}
\big\langle L_{p_1} (\rn), L_{p_2} (\rn), \theta \big\rangle = L_p (\rn)
\end{\eq}
from the Lebesgue spaces $L_p (\rn)$ 
to the Morrey spaces $\La^{\vr}_p (\rn)$ as introduced in Definitions~\ref{D2.3}.
\end{remark}

\subsection{Lifts}   \label{S3.3}
It is well known that the classical lift operator 
\begin{\eq}   \label{3.24}
I_\sigma: \quad f \mapsto \big( \langle \xi \rangle^{-\sigma} \wh{f} \,\big)^\vee \quad \text{with} \quad \langle \xi \rangle = (1 + 
|\xi|^2 )^{1/2}, \quad \xi \in \rn, \quad \sigma \in \real,
\end{\eq}
maps the space $\As (\rn)$ isomorphically  onto $A^{s+\sigma}_{p,q} (\rn)$,
\begin{\eq}   \label{3.25}
\| I_\sigma f \, | A^{s+\sigma}_{p,q} (\rn) \| \sim \| f \, | \As (\rn) \|, \qquad f \in \As (\rn).
\end{\eq}
This goes back to \cite[Theorem 2.3.8, pp.\,58--59]{T83} with $p<\infty$ for the $F$--spaces and has been extended in \cite[Theorem
1.22, p.\,16]{T20} to all members $\As (\rn)$ in \eqref{2.38} of the $n$--clan in Definition~\ref{D2.11}(i) (now including 
$F^s_{\infty,q} (\rn)$). We fix the  more or less obvious counterpart for spaces belonging to a $\vr$--clan according to Definition
\ref{D2.11}(ii).

\begin{theorem}   \label{T3.8}
Let $n\in \nat$, $-n <\vr <0$ and
\begin{\eq}   \label{3.26}
s \in \real, \qquad 0<p<\infty, \qquad 0<q \le \infty.
\end{\eq}
Let $\sigma \in \real$. Then $I_\sigma$ maps $\La^{\vr} \As (\rn)$ isomorphically onto $\La^{\vr} A^{s+\sigma}_{p,q} (\rn)$, \new{ and $\La_{\vr} \As (\rn)$ isomorphically onto $\La_{\vr} A^{s+\sigma}_{p,q} (\rn)$, $A \in
\{B,F \}$,}
\begin{\eq}   \label{3.27}
I_\sigma \La^{\vr} \As (\rn) = \La^{\vr} A^{s+\sigma}_{p,q} (\rn), \qquad A \in \{B,F \},
\end{\eq}
\begin{\eq}   \label{3.28}
\new{I_\sigma \La_{\vr} \As (\rn) = \La_{\vr} A^{s+\sigma}_{p,q} (\rn), \qquad A \in \{B,F \}.}
\end{\eq}
\end{theorem}

\begin{proof}
The assertion \eqref{3.27} follows from \cite[Theorem 3.72, p.\,102]{T14} and \eqref{2.31}. The real interpolation \eqref{3.14} shows
that this property can be extended to the spaces \new{$\La_{\vr} \Bs(\rn)$ in \eqref{3.28}, while the result for $\La_{\vr} \Fs(\rn)$ is a consequence of the coincidence \eqref{2.42} and \eqref{3.27}}.
\end{proof}

  
\begin{remark}  \label{R3.9}
  \new{
    In view of Definition~\ref{D2.11} we can summarise \eqref{3.27}, \eqref{3.28} as
    \[
    I_\sigma \left(\rhoAs (\rn)\right) = \rhoAx{s+\sigma}{p}{q} (\rn), \qquad A \in \{B,F \}.
\]
The result \eqref{3.28} for spaces $\La_{\vr}\Bs(\rn)$ can be found in \cite{TX} already, while \eqref{3.27} is covered by \cite[Proposition~5.1]{YSY10}, see also \cite[Corollary~3.2]{Sic12}. Again, using the identity \eqref{2.42}, this finally covers the case $\La_{\vr}\Fs(\rn)$ in \eqref{3.28}.  }
  
  The related mapping property
\begin{\eq}   \label{3.29}
I_\sigma \La^0 \Bs (\rn) = \La^0 B^{s+\sigma}_{p,q} (\rn)
\end{\eq}
of the members \eqref{2.41} of the $0$--clan in Definition~\ref{D2.11}(iii) is also covered by \cite[Theorem 3.72, p.\,102]{T14} and
\eqref{2.31} with $r= \vr =0$.
\end{remark}

\subsection{Derivatives}    \label{S3.4}
Derivatives are denoted as usual as $\Dd^\alpha = \pa^{\alpha_1}_1 \cdots \pa^{\alpha_n}_n$, $\pa^{\alpha_j}_j =
\pa^{\alpha_j}/ \pa x^{\alpha_j}_j$, $\alpha = (\alpha_1, \ldots , \alpha_n)$, $\alpha_j \in \no$, $|\alpha| =
\sum^n_{j=1} \alpha_j$. The classical assertion 
\begin{\eq}    \label{3.30}
\| f \, | \As (\rn) \| \sim \sup_{0 \le |\alpha| \le m} \| \Dd^\alpha f \, | A^{s-m}_{p,q} (\rn) \|
\end{\eq}
(equivalent  quasi--norms) for all spaces of the $n$-clan in \eqref{2.38} (including $F^s_{\infty,q} (\rn)$)
may be found in  {\cite[Theorem 1.24, p.\,17]{T20}}. We are interested in an extension of this assertion to all
families of the $\vr$--clan in Definition~\ref{D2.11}(ii). We agree that \eqref{3.30} (and its counterparts in
the following theorem) includes the assertion that $f \in \As (\rn)$ if, and only if, $\Dd^\alpha f \in
A^{s-m}_{p,q} (\rn)$ for all $\alpha$ with $0 \le |\alpha| \le m$. 

\begin{theorem}   \label{T3.10}
Let $n\in \nat$, $-n <\vr <0$ and
\begin{\eq}   \label{3.30a}
s\in \real, \qquad 0<p<\infty, \qquad 0<q \le \infty.
\end{\eq}
Let $m\in \nat$. Then
\begin{\eq}   \label{3.31}
\| f \, | \La^{\vr} \As (\rn) \| \sim \sup_{0 \le |\alpha| \le m} \| \Dd^\alpha f \, | \La^{\vr} A^{s-m}_{p,q} (\rn) \|, \quad A\in \{B,F \},
\end{\eq}
$($equivalent quasi--norms$)$ and
\begin{\eq}   \label{3.32}
\new{\| f \, | \La_{\vr} \As (\rn) \| \sim \sup_{0 \le |\alpha| \le m} \| \Dd^\alpha f \, | \La_{\vr} A^{s-m}_{p,q} (\rn) \|, \quad A\in \{B,F \},}
\end{\eq}
$($equivalent quasi--norms$)$.
\end{theorem}

\begin{proof}
The equivalence \eqref{3.31}, including the above agreement, follows from \cite[Corollary 3.66, p.\,98]{T14}
and again \eqref{2.31}. Then the real interpolation formula \eqref{3.14} shows that also
\begin{\eq}   \label{3.33}
\sup_{0 \le |\alpha| \le m} \| \Dd^\alpha f \, | \La_{\vr} B^{s-m}_{p,q} (\rn) \| \le c\,\| f \, | \La_{\vr} \Bs (\rn) \|, \quad f\in \La_{\vr} \Bs (\rn),
\end{\eq}
$m\in \nat$. The reverse inequality can be based on the lifting \eqref{3.28},
\begin{\eq} \label{3.34}
\big\| (1- \Delta)^{\sigma/2} f \, | \La_{\vr} B^{s-\sigma}_{p,q} (\rn) \big\| \sim \| f \, | \La_{\vr} \Bs (\rn) \|, \qquad \sigma \in \real.
\end{\eq}
If $m\in \nat$ is even, then the converse of \eqref{3.33} follows from \eqref{3.34} with $\sigma =m$. Let $m\in
\nat$ be odd. \new{Let  $f\in \La_{\vr} \Bs (\rn)$}. We assume that for each $\ve >0$ there is an $f_{\ve} \in \La_{\vr} \Bs (\rn)$ such that
\begin{\eq}   \label{3.35}
  \sup_{0 \le |\alpha| \le m} \| \Dd^\alpha f_{\ve} \, | \La_{\vr} B^{s-m}_{p,q} (\rn) \| \le \ve\,\| f_{\ve} \, | \La_{\vr} \Bs (\rn) \|.
\end{\eq}
Application of what one already knows shows that there is a corresponding inequality with $m+1$ in place of
$m$. But this is a contradiction. This proves  \eqref{3.32} for all $m\in \nat$ \new{for spaces $\La_\vr\Bs(\rn)$}, \new{while the result for $\La_\vr\Fs(\rn)$ follows from \eqref{2.42} and \eqref{3.31} again}. 
\end{proof}

\begin{remark}   \label{R3.11}
  \new{We refer to \cite{TX}, for first results of type \eqref{3.31} in case of spaces $\Ft(\rn)=\La^{\vr} \Fs(\rn)$, recall \eqref{2.29}, \eqref{2.30} and for spaces of type $\mathcal{N}^s_{u,p,q}(\rn)=\La_\vr\Bs(\rn)$, recall \eqref{2.24}; see also \cite{SawTan}. Again we may summarise \eqref{3.31} and \eqref{3.32} by
    \[
\| f \, | \rhoAs (\rn) \| \sim \sup_{0 \le |\alpha| \le m} \| \Dd^\alpha f \, | \rhoAx{s-m}{p}{q} (\rn) \|, \quad A\in \{B,F \},
    \]
in the sense of $($equivalent quasi--norms$)$, recall Definition~\ref{D2.11}.
  }
  
Based on \cite[Corollary 3.66, p.\,98]{T14} one can extend \eqref{3.31} to the $0$--clan in Definition 
\ref{D2.11}(iii),
\begin{\eq}   \label{3.36}
\| f \, | \La^0 \Bs (\rn) \| \sim \sup_{0 \le |\alpha| \le m} \| \Dd^\alpha f \, | \La^0 B^{s-m}_{p,q} (\rn) \|,
\quad f\in \La^0 \Bs (\rn).
\end{\eq}
\end{remark}

\begin{remark}   \label{R3.11a}
We need Theorem~\ref{T3.10} as a tool in our later considerations. On the other hand, derivatives and 
differences played a decisive r\^ole in the theory of function spaces from the very beginning. One may expect
that characterisations  of some spaces $\As (\rn)$ in \eqref{2.38}, the $n$--clan, in terms of the differences
\begin{\eq}   \label{3.36a}
\Delta^1_h f (x) = f(x+h) - f(x), \qquad \big( \Delta^{k+1}_h f \big)(x) = \Delta^1_h \big( \Delta^k_h f \big) (x),
\end{\eq}
$x\in \rn$, $h\in \rn$, $k\in \nat$, and their ball means
\begin{\eq}   \label{3.36b}
d^m_{t,v} f(x) = \Big( t^{-n} \int_{\{h: |h| \le t\}} \big| \Delta^m_h f(x) \big|^v \, \di h \Big)^{1/v}, \quad x\in \rn, \quad t>0,
\end{\eq}
$0<v \le \infty$, have suitable counterparts for the spaces in \eqref{2.44}, \eqref{2.45} of the $\vr$--clan,
$-n< \vr <0$, subject to the (modified) Slope--$n$-Rule, see \new{Slope Rules~\ref{slope_rules}(ii)}. The related theory for the 
spaces $\As (\rn)$ with $p<\infty$ for $F$--spaces in \cite[Theorem 3.5.3, p.\,194]{T92} has been complemented
quite recently by corresponding assertions for $F^s_{\infty,q} (\rn)$. Discussions and detailed references 
may be found in {\cite[Section 4.3.3, pp.\,134--135]{T20}}. First substantial results about characterisations of
some spaces in \eqref{2.44}, \eqref{2.45} in terms of \eqref{3.36a}, \eqref{3.36b} may be found in \cite{HoS20} and \cite{Hov20b}. \new{In \cite{WYY} further equivalent characterisations of the spaces $\La^\vr\As(\rn)$ via derivatives can be found.}
\end{remark}

\subsection{Fatou property}   \label{S3.5}
Let $A(\rn)$ be a quasi--normed space in $S'(\rn)$ with $A(\rn) \hra S'(\rn)$ (continuous embedding). Then $A(\rn)$ is said to have
the  Fatou property if there is a positive constant $c$ such that from
\begin{\eq}   \label{3.37}
\sup_{j \in \nat} \| g_j \, | A(\rn) \| <\infty \qquad \text{and} \quad \text{$g_j \to g$ in $S'(\rn)$}
\end{\eq}
it follows that $g\in A(\rn)$ and
\begin{\eq}   \label{3.38}
\| g \, | A(\rn) \| \le c \, \sup_{j\in \nat} \| g_j \, | A(\rn) \|.
\end{\eq}
We took over the above formulation from \cite[Section 1.3.4, pp.\,18--19]{T20}. But the Fatou property for
this type of function spaces has some history which may be found there. The surprisingly simple argument
ensuring that all spaces in \eqref{2.38} of the $n$--clan have the Fatou property applies equally to all spaces
in Definition~\ref{D2.11}.

\begin{theorem}  \label{T3.12}
Let $-n <\vr <0$. Then all spaces \eqref{2.44}, \eqref{2.45} of the $\vr$--clan according to  Definition~\ref{D2.11}{\em(ii)} have the Fatou property.
\end{theorem}

\begin{proof}
Let $\psi \in S(\rn)$. Then
\begin{\eq}   \label{3.39}
(\psi \wh{f} )^\vee (x) = c \, \big(f, \psi^\vee (x - \cdot) \big), \qquad f \in S'(\rn), \quad x\in \rn.
\end{\eq}
This reduces the Fatou property for the Fourier--analytically defined spaces in \eqref{2.21}, \eqref{2.22}
based on \eqref{2.17} and \eqref{2.26}, \eqref{2.27}
to the classical measure--theoretical Fatou property for $L_p$--spaces.
\end{proof}

\begin{remark}    \label{R3.13}
This argument applies also to the spaces of the $n$--clan in \eqref{2.38}, as stated in \cite[Theorem 1.25,
p.\,18]{T20}, and to the spaces of the $0$--clan in \eqref{2.41}. Recall  that $S(\rn)$ is dense in the 
classical spaces $\As (\rn)$ if $\max(p,q) <\infty$. Then quite often properties can be proved first for
$S(\rn)$ and extended by completion to $\As (\rn)$. This argument cannot be applied to spaces belonging to
the $\vr$--clan, $-n<\vr<0$, or $0$--clan. These spaces are not separable and $S(\rn)$ is not dense in them.
Then the Fatou property is an effective substitute.
\end{remark}

\section{Key problems}    \label{S4}
\subsection{Motivations}    \label{S4.1}
The spaces $\Bs (\rn)$ and $\Fs (\rn)$ with $0<p,q \le \infty$ ($p<\infty$ for $F$--spaces) and $s \in \real$ according to the parts
(i) and (ii) of Definition~\ref{D2.1}  were introduced from the mid 1960s to the mid 1970s. Related (historical) references may be
found in {\cite[Section 4.3.4, pp.\,135--136]{T20}}. The incorporation of the spaces $F^s_{\infty,q} (\rn)$ in part (iii) of Definition
\ref{D2.1} came somewhat later and had been discussed in detail in \cite[pp.\,3--5]{T20}. At the end of the 1970s the question arose
whether these spaces are worth to be studied, especially their extensions from $p\ge 1$ to $p<1$. As a good criterion at this time (and
maybe up to now) one may ask whether the restriction $\As (\Om)$ of $\As (\rn)$, $A \in \{B,F \}$, to bounded smooth domains $\Om$
in $\rn$ can be used to study elliptic boundary value problems in \Om. This requires that these spaces have the following 
distinguished properties, called key problems in \cite{T92} (but treated already in \cite{T83}):
\begin{enumerate}
\item pointwise (smooth) multipliers in $\As (\rn)$,
\item diffeomorphisms of $\As (\rn)$,
\item extensions of corresponding spaces $\As (\real^n_+)$ to $\As (\rn)$,
\item traces of $\As (\rn)$ on \Rn, $2\le n \in \nat$.
\end{enumerate}
Here
\begin{\eq}   \label{4.1}
\real^n_+ = \big\{ x= (x_1, \ldots, x_n )\in \rn: \ x_n >0 \big\}, \qquad n\in \nat.
\end{\eq}
Final assertions and (historical) references may be found in {\cite{T20}}. Beginning with the early 1990s the classical spaces $\As
(\rn)$ have been extended by the two type of Morrey smoothness spaces \eqref{1.2}, \eqref{1.3}. They have been studied in detail in
numerous papers. An examination of the obtained assertions (some of them are quite recent) suggests not to look at these two types of
Morrey smoothness spaces separately, but to unite and reorganise them into the clans as described in Definition~\ref{D2.11} based on
\eqref{2.24}, \eqref{2.25} and \eqref{2.29}, \eqref{2.30}. The emerging fourth parameters $\vr$ is now on equal footing with the
classical parameters $s$ (smoothness), $p$ (integrability) and $q$ (summability) where $|\vr|$ is quite often the slope of 
distinguished (broken) lines in the $\big( \frac{1}{p}, s \big)$--diagram. In addition, the three families in \eqref{2.44} of a fixed
$\vr$--clan are symbiotically related to each other, where the real interpolation \eqref{3.14} is a good and useful example. In what
follows we reformulate many already existing assertions in terms of the $\vr$--clans and complement them by some new properties which
illuminate our approach. It is quite natural to deal first with the above key problems. It comes out that the first three of them can 
be treated rather quickly using already existing properties. But related assertions about traces in the literature are not really
satisfactory. We deal with them in detail in the context of our approach. A new phenomenon comes out (here and at other occasions).
The behaviour of the low slope spaces, that means $0<|\vr| <1$, is more or less the same in all dimensions $n$. In particular, there
is no breaking point at $p=1$ as usual for the spaces $\As (\rn)$, $2 \le n \in \nat$, and also for high slope spaces with $1<|\vr|
<n$. This is in good agreement with the Slope--1--Rule as formulated in \new{Slope Rules~\ref{slope_rules}(i)}\ignore{at the end of Section~\ref{S2.2}}.

\subsection{Pointwise multipliers}    \label{S4.2}
Pointwise multipliers for the classical spaces $\As (\rn)$ in \eqref{2.38}, the $n$--clan according to Definition~\ref{D2.11}, have
been studied with great intensity over decades. Rather final results have been obtained (now including $F^s_{\infty,q} (\rn)$). One
may consult {\cite[Section 2.4]{T20}} where one finds also detailed references. It might be a challenging task to develop a comparable
theory for the spaces within the $\vr$--clans as introduced in Definition~\ref{D2.11}. Our aim here is rather modest. We deal with 
smooth pointwise multipliers as requested in the first of the four key problems listed above in Section~\ref{S4.1}.

Let $A(\rn)$ be a quasi--Banach space in $\rn$ with the continuous embedding
\begin{\eq}   \label{4.2}
S(\rn) \hra A(\rn) \hra S'(\rn).
\end{\eq}
Then $g\in L_\infty (\rn)$ is said to be a pointwise multiplier for $A(\rn)$ if
\begin{\eq}   \label{4.3}
f \mapsto gf \qquad \text{generates a bounded map in $A(\rn)$}.
\end{\eq}
Of course one has to say what this multiplication means in this generality. References for related detailed discussions in connection
with the spaces $\As (\rn)$ may be found {\cite[Section 2.4.1, pp.\,40--43]{T20}}. We  reduce smooth pointwise multipliers for the 
spaces covered by Definition~\ref{D2.11} to related assertions for the spaces $\As (\rn)$ (as already done in \cite[Section 3.6.4, 
pp.\,99--101]{T14}) and interpolation. In particular there is no need for further technical explanations. Recall that $C^k (\rn)$
with $k\in \no$ is the collection of all complex--valued continuous functions in $\rn$ having classical continuous derivatives up to
order $k$ inclusively with
\begin{\eq}   \label{4.4}
\| g \, | C^k (\rn) \| = \sup_{|\alpha| \le k, x\in \rn} | \Dd^\alpha g(x) | < \infty.
\end{\eq}
For our purpose it is sufficient to know that for any classical space $\As (\rn)$ in \eqref{2.38} there is a natural number $k\in 
\nat$ and a constant $c>0$ such  that
\begin{\eq}  \label{4.5}
\| g f \, | \As (\rn) \| \le c \, \|g \, | C^k (\rn) \| \cdot \|f \, | \As (\rn) \|
\end{\eq}
for all $g\in C^k (\rn)$ and $f\in \As (\rn)$.

\begin{theorem}   \label{T4.1}
For any space in \eqref{2.44} with $-n < \vr <0$, $s\in \real$, $0<p<\infty$, $0<q \le \infty$, there is a natural number $k\in \nat$
and a constant $c>0$ such that
\begin{\eq}  \label{4.6}
\| g f \, | \La^{\vr}\As (\rn) \| \le c \, \|g \, | C^k (\rn) \| \cdot \|f \, | \La^{\vr}\As (\rn) \|
\end{\eq}
for all $g\in C^k (\rn)$ and $f\in \La^{\vr}\As (\rn)$, $A \in \{B,F \}$, and
\begin{\eq}  \label{4.7}
\new{\| g f \, | \La_{\vr}\As (\rn) \| \le c \, \|g \, | C^k (\rn) \| \cdot \|f \, | \La_{\vr}\As (\rn) \|}
\end{\eq}
for all $g\in C^k (\rn)$ and \new{$f\in \La_{\vr}\As (\rn)$, $A \in \{B,F \}$}.
\end{theorem}

\begin{proof}
  Let $Q_{J,M}$ be as in \eqref{2.6} and $2 Q_{J,M}$ be as explained there. Let $\La^{\vr} \As (\rn)$ be again the spaces as introduced
in Definition~\ref{D2.7} (now with $-n <\vr <0$). Let $s <m \in \no$. Then it follows from \cite[Corollary 3.66, p.\,98]{T14} and
\eqref{2.31} that $f \in \La^{\vr} \As (\rn)$ if, and only if,
\begin{\eq}    \label{4.8}
\sup_{\substack{J\in \ganz, M \in \zn, \\0\le |\alpha| \le m}} 2^{\frac{J}{p} (n + \vr)} \| \Dd^\alpha f \, | A^{s-m}_{p,q} (2 Q_{J,M})
\|
\end{\eq}
is finite (equivalent quasi--norms). Here $\As (\Om)$ is the usual restriction of $\As (\rn)$ to the domain (= open set) $\Om$ in \rn.
In particular \eqref{4.5} remains valid with the same $c$ and $g$ if one replaces there $\As (\rn)$ by $\As (\Om)$ (independently of
\Om). Then \eqref{4.6} follows from \eqref{4.8} and \eqref{4.5} (with a different $k\in \nat$). \new{Now, in view of the coincidence \eqref{2.42}, it is sufficient to prove \eqref{4.7} in case of $A=B$. This can} be obtained by the real interpolation \eqref{3.14}.
\end{proof}

\begin{remark}   \label{R4.2}
We refer the reader to \cite[Theorem 3.69, Remarks 3.70, 3.71, p.\,101]{T14}. There one finds \eqref{4.6} based on the same arguments
as above, further explanations and references. In particular, \eqref{4.6} can be extended to the spaces in \eqref{2.41} of the
$0$--clan in Definition~\ref{D2.11}(iii). The simple proof of \eqref{4.7} based  on \eqref{4.6} underlines again the close relationship
of the three families in \eqref{2.44} of a {\vr}--clan among each other.
\end{remark}

\subsection{Diffeomorphisms}   \label{S4.3}
A continuous  one--to--one map of $\rn$, $n\in \nat$,  onto itself,
\begin{align}  \label{4.9}
y &= \psi (x) = \big(\psi_1 (x), \ldots, \psi_n (x) \big), && \text{$x\in \rn$}, &&  \\  \label{4.10}
x &= \psi^{-1} (y) = \big( \psi^{-1}_1 (y), \ldots, \psi^{-1}_n (y) \big), && \text{$y\in \rn$}, &&
\end{align}
is called a diffeomorphism if all components $\psi_j (x)$ and $\psi^{-1}_j (y)$ are real $C^\infty$ functions on $\rn$ and
for $j=1, \ldots,n$,
\begin{\eq}   \label{4.11}
\sup_{x\in \rn} \big( | \Dd^\alpha \psi_j (x) | + | \Dd^\alpha \psi^{-1}_j (x)| \big) < \infty \quad  \text{for all $\alpha \in \nat^n_0$
with $|\alpha| >0$}.
\end{\eq}
We used again standard notation. In particular, $\nat^n_0$ collects all $\alpha = (\alpha_1, \ldots, \alpha_n)$, $\alpha_j \in \no$,
and $|\alpha| = \sum^n_{j=1} \alpha_j$. 
Then $\vp \mapsto \vp \circ \psi$, given by $(\vp \circ \psi )(x) = \vp \big( \psi(x) \big)$, is an one--to--one map of $S(\rn)$ onto
itself. This can be extended by standard arguments to an one--to--one map of $S'(\rn)$ onto itself. Some related details may be found 
in {\cite[Section 2.3, pp.\,39--40]{T20}}.

\begin{proposition}   \label{P4.3}
Let $\psi$ be the above diffeomorphic map of $\rn$ onto itself. Then
\begin{\eq}   \label{4.12}
D_\psi: \quad \As (\rn) \hra \As (\rn), \qquad D_\psi f = f \circ \psi
\end{\eq}
is an isomorphic map for all spaces in \eqref{2.38}, the $n$--clan.
\end{proposition}

\begin{remark}   \label{R4.4}
This diffeomorphism is a tricky problem and attracted some attention over decades. Related references may be found in 
{\cite[p.\,39]{T20}}. Finally emerged in \cite[Section 1.3.8, pp.\,66--67]{T19} a streamlined proof on less than one page for all spaces
with $p<\infty$ for $F$--spaces. The incorporation of the spaces $F^s_{\infty,q} (\rn)$ in \cite[Theorem 2.25, p.\,39]{T20} had been
based on \eqref{2.35} and \eqref{4.8} with $\vr =0$ and $A^{s-m}_{p,q} (\rn) = F^{s-m}_{p,q} (\rn)$, $p<\infty$, $s <m \in \no$. But
this reduction works for all spaces $\La^{\vr} \As (\rn)$ and had already been used in \cite[Theorem 3.69, p.\,101]{T14} in connection
with the above problem. We formulate the outcome.
\end{remark}

\begin{theorem}   \label{T4.5}
Let $\psi$ be the above diffeomorphic map of $\rn$ onto itself. For $-n < \vr <0$ let \new{$\La^{\vr} \As (\rn)$ and
$\La_{\vr} \As (\rn)$, $A\in \{B,F \}$, be} the three families of the $\vr$--clan $\rhoAs (\rn)$ according to Definition~\ref{D2.11}{\em (ii)}. Then
\begin{\eq}   \label{4.13}
D_\psi : \quad \La^{\vr} \As (\rn) \hra \La^{\vr} \As (\rn), \qquad D_\psi f = f\circ \psi,
\end{\eq}
and
\begin{\eq}   \label{4.14}
\new{D_\psi : \quad \La_{\vr} \As (\rn) \hra \La_{\vr} \As (\rn), \qquad D_\psi f = f\circ \psi,}
\end{\eq}
are isomorphic maps.
\end{theorem}

\begin{proof}
One has for some $d>0$ the controlled distortion
\begin{\eq}    \label{4.15}
\psi (Q_{J,M}) \subset d \, Q_{J, M'} \qquad \text{for all} \quad J\in \ganz, \quad M\in \zn,
\end{\eq}
and $M' =M' (J,M) \in \zn$. Then application of \eqref{4.5} and Proposition~\ref{P4.3} to \eqref{4.8} with $f \circ \psi$ in place of
$f$ show that $D_\psi$ in \eqref{4.13} maps $\La^{\vr} \As (\rn)$ into itself. This argument can also be applied to $\psi^{-1}$. This
ensures that $D_\psi$ is a map onto. The incorporation of $D_\psi$ in \eqref{4.14} is again a matter of the real interpolation 
\eqref{3.14}, \new{where again we may restrict ourselves to the case $A=B$ in view of \eqref{2.42}}.
\end{proof}

\begin{remark}   \label{R4.6}
It follows by the same argument that
\begin{\eq}   \label{4.16}
D_\psi : \quad \La^0 \Bs (\rn) \hra \La^0 \Bs (\rn), \qquad D_\psi f = f\circ \psi,
\end{\eq}
is an isomorphic map for the spaces in \eqref{2.41} of the $0$--clan according to Definition~\ref{D2.11}(iii).
\end{remark}

\subsection{Extensions}    \label{S4.4}
Next we deal with the third of the four key problems described in Section~\ref{S4.1}. Let again
\begin{\eq}   \label{4.17}
\rnp = \{ x= (x_1, \ldots, x_n) \in \rn: \ x_n >0 \}, \qquad n\in \nat.
\end{\eq}
As usual, $D'(\rnp)$ denotes the set of all distributions in $\rnp$. Furthermore, $g|\rnp \in D'(\rnp)$ stands for
the restriction of $g\in S' (\rn)$ to $\rnp$.

\begin{definition}   \label{D4.7}
Let $A(\rn)$ be a space covered by Definition~\ref{D2.11}. Then
\begin{\eq}   \label{4.18}
A(\rnp) = \big\{ f \in D' (\rnp): \ \text{$f= g|\rnp$ for some $g\in A(\rn)$} \big\},
\end{\eq}
\begin{\eq}   \label{4.19}
\| f \, | A(\rnp) \| = \inf \| g\,| A(\rn) \|,
\end{\eq}
where the infimum  is taken over all $g\in A(\rn)$ with $g|\rnp =f$.
\end{definition}

\begin{remark}   \label{T4.8}
It follows from standard arguments that $A(\rnp)$ is a quasi--Banach space (and a Banach space for $p \ge 1$, $q \ge 1$), continuously
embedded in $D' (\rnp)$ or the restriction of $S'(\rn)$ to \rnp.
\end{remark}

The restriction operator $\re$,
\begin{\eq}   \label{4.20}
\re f = f|\rnp: \quad S'(\rn) \hra D'(\rnp)
\end{\eq}
generates a linear and bounded map from $A(\rn)$ onto $A (\rnp)$. One asks for a linear and bounded extension operator $\ext$,
\begin{\eq}   \label{4.21}
\ext: \quad A(\rnp) \hra A(\rn)
\end{\eq}
such that
\begin{\eq}   \label{4.22}
\re \circ \ext = \id \qquad \text{identity in $A(\rnp)$}.                            
\end{\eq}
Of interest are common extension operators for substantial parts of the spaces in Definition~\ref{D2.11}. In the case of the 
$\vr$--clan with $-n <\vr <0$ there are for any $0<\ve <1$ common extension operators for all related spaces restricted by
\begin{\eq}   \label{4.23}
\ve <p<\infty, \quad \ve <q \le \infty, \quad |s| < \ve^{-1}, \quad -n < \vr < -\ve,
\end{\eq}
where the $n$--clan, consisting of the classical spaces in \eqref{2.38}, and the $0$--clan can be incorporated. This ensures that
interpolation can be applied. This will be of some use for us. The extension as described above for the classical spaces $\As (\rn)$
in \eqref{2.38} was a central topic in the theory of these function spaces from the very beginning. In {\cite[Section 2.5, 
pp.\,57--62]{T20}} one finds final assertions, including in particular $F^s_{\infty,q} (\rn)$, and related references which will not
be repeated here. But the incorporation of $F^s_{\infty,q} (\rn)$ into the already existing theory for the spaces $\As (\rn)$ (with
$p<\infty$ for $F$--spaces) is based on \eqref{2.35},
\begin{\eq}   \label{4.24}
\La^0 \Fs (\rn) = F^s_{\infty,q} (\rn), \qquad s\in \real, \quad 0<p<\infty, \quad 0<q \le \infty,
\end{\eq}
and rather peculiar lifting properties for the spaces $\As (\rn)$ and their restrictions to \rnp. A detailed description may be found
in {\cite[Proposition 2.68, p.\,59]{T20}}. The method itself goes back to \cite[Section 2.10.3, pp.\,231--233]{T78}. It has been 
extended in \cite{Saw10} and \cite{MNS19} to all spaces $\La^{\vr} \As (\rn)$ covered by Definition~\ref{D2.7}. One may also consult
\cite[Section 6.4.1, pp.\,168--172]{YSY10}. We fix the outcome.

\begin{theorem}    \label{T4.9}
Let
\begin{\eq}   \label{4.25}
\La_{\vr} \Bs(\rnp), \quad \La^{\vr} \Bs (\rnp) \quad \text{and} \quad \La_{\vr} \Fs (\rnp) = \La^{\vr} \Fs (\rnp)
\end{\eq}
be the restrictions of the spaces \eqref{2.44}, \eqref{2.45} on $\rnp$ of the $\vr$--clan $\rhoAs (\rn)$ 
with $-n <\vr <0$. Then for any $0 <\ve <1 $
and $p,q,s, \vr$ as in \eqref{4.23} there is a common extension operator from these spaces to their $\rn$--counterparts.
\end{theorem}

\begin{proof}
This follows from \cite[Theorems 3.7, 3.15, pp.\,327, 336]{MNS19}, reformulated according to \eqref{2.24}, \eqref{2.25} and 
\eqref{2.29}, \eqref{2.30}.
\end{proof}

\begin{remark}   \label{R4.10}
In \cite[Section 1.11.8, pp.\,69--72]{T06} it had been explained how restrictions and extensions can be used to shift interpolation
formulas from $\rn$ to bounded Lipschitz domains \Om. The arguments work equally well for $\rnp$ in place of \Om. In particular, 
Proposition~\ref{P3.3}, the useful interpolation formula \eqref{3.14} and Theorem~\ref{T3.5} remain valid if one replaces there
$\rn$ by \rnp.
\end{remark}

\begin{remark}   \label{R4.11}
By the given references one can extend Theorem~\ref{T4.9} to the spaces of the $0$--clan in Definition~\ref{D2.11}(iii) and their
restrictions to \rnp.
\end{remark}

\subsection{Traces}    \label{S4.5}
We come to the fourth and last key problem as listed in Section~\ref{S4.1}. Let $n\in \nat$ and $ n \ge 2$. Let $x =(x', x_n)$, $x'
\in \Rn$, $x_n \in \real$. The question arises for which spaces covered by Definition~\ref{D2.11} the trace $\tr$,
\begin{\eq}   \label{4.25a}
\tr f = f(x', 0), \quad x' \in \Rn, \quad f\in \La^{\vr} \As (\rn) \ \text{or} \  f\in \La_{\vr} \As (\rn)
\end{\eq}
generates a linear and bounded operator into (or better onto) some related spaces on \Rn. For the classical spaces in \eqref{2.38}, the
$n$--clan, one has rather final answers. One may consult {\cite[Section 2.2, pp.\,29--38]{T20}} and the references {therein}, covering also
some technical explanations which will not be repeated here. It was the main aim of \cite[Theorem 2.13, p.\,32]{T20} to complement
the  already existing assertions for the spaces $\As (\rn)$ with $p<\infty$ for the $F$--spaces by
\begin{\eq}   \label{4.26}
\tr: \quad F^s_{\infty,q} (\rn) \hra \Cc^s (\Rn), \qquad 0<q \le \infty, \quad s>0,
\end{\eq}
and
\begin{\eq}   \label{4.27}
\ext: \quad \Cc^s (\Rn) \hra F^s_{\infty,q} (\rn), \qquad 0<q \le \infty, \quad s>0,
\end{\eq}
with
\begin{\eq}   \label{4.28}
\tr \circ \ext = \id, \qquad \text{identity in $\Cc^s (\Rn)$}.
\end{\eq}
It comes out that the underlying arguments based on \eqref{4.24} can also be used to study traces of the spaces $\La^{\vr} \As (\rn)$.
This will be done below. But first we discuss some assertions about traces available in the literature. 

Let $n \ge 2$, $0<p<\infty$, $0<q \le \infty$, $-n \le \vr <-1$ and 
\begin{\eq}   \label{4.29}
s- \frac{1}{p} > \sigma^{n-1}_p = (n-1) \Big( \max \big( \frac{1}{p}, 1 \big) - 1 \Big).
\end{\eq}
Then $\tr$ is a linear and bounded operator such that
\begin{\eq}   \label{4.30}
\tr \La^{\vr} \Bs (\rn) = \La^{\vr+1} B^{s- \frac{1}{p}}_{p,q} (\Rn)
\end{\eq}
and
\begin{\eq}   \label{4.31}
\tr \La^{\vr} \Fs (\rn) = \La^{\vr+1} B^{s- \frac{1}{p}}_{p,p} (\Rn).
\end{\eq}
These assertions are covered by \cite[Theorem 3.10.8, p.\,144]{Sic12} based on \cite[Theorem 1.3, p.\,75]{SYY10} and \cite[Theorem 6.8, p.\,163]{YSY10} reformulated according to \eqref{2.29}, \eqref{2.30}. The related counterpart
\begin{\eq}   \label{4.32}
\tr \new{\La_{\vr} \Bs (\rn)} = \La_{\vr+1} B^{s- \frac{1}{p}}_{p,q} (\Rn)
\quad\text{and}\quad \tr \La_{\vr} \Fs (\rn) = \La^{\vr+1} B^{s- \frac{1}{p}}_{p,p} (\Rn)
\end{\eq}
under the same restrictions for the parameters $p,q, \vr$ and $s$ as above goes back to \cite[Theorem 3.8, p.\,329]{MNS19}. \new{Again the case $A=F$ in \eqref{4.32} follows from \eqref{4.31} due to the coincidence \eqref{2.42}, whereas Theorem~\ref{T2.16}(iv) shows that one cannot replace there $\La^{\vr+1}$ by $\La_{\vr+1}$ as one might expect.} It follows
from \eqref{2.32} that one recovers the classical assertions if one chooses $\vr = -n$ in \eqref{4.30}--\eqref{4.32}. The 
Slope--$n$--Rule, see \new{Slope Rules~\ref{slope_rules}(ii)}, suggests that $s - \frac{|\vr|}{p}$ is the so--called differential dimension for the
spaces on the left--hand sides in \eqref{4.30}--\eqref{4.32}. Similarly for the spaces on the related right--hand sides. It follows
from
\begin{\eq}   \label{4.33}
s - \frac{|\vr|}{p} = s- \frac{1}{p} + \frac{\vr +1}{p} = s- \frac{1}{p} - \frac{|\vr + 1|}{p}
\end{\eq}
that they are the same both for the original spaces in $\rn$ and the related target spaces in \Rn. But the above assertions for the
spaces \new{$\rhoAs(\rn)$} in \eqref{2.44} of the $\vr$--clan with $-n < \vr <0$ are less complete than their classical counterparts in \cite[Theorem 2.13, p.\,32]{T20}. There arise several questions. The condition \eqref{4.29} is natural for the classical spaces $\As (\rn)$, but not for
the spaces in \eqref{2.44} subject to the \new{Slope Rules~\ref{slope_rules}}. The above assertions exclude spaces with $-1 \le \vr <0$,
deserving special attention as will be clear later on. The (preferably atomic) arguments in the papers mentioned above do not produce
linear and bounded extension operators in generalisation of \eqref{4.27}, \eqref{4.28}. But this is crucial both for the theory itself
and also for applications, for example to boundary value problems for elliptic differential operators in, say, smooth bounded domains
in \rn. We try to answer at least some of these questions and to raise the trace theory for the spaces in \eqref{2.44} to the same
level as the other key problems treated above.

Let again $\Cc^s (\rn)$, $s\in \real$, be the H\"{o}lder--Zygmund spaces as introduced in \eqref{2.16}. Let $\La^{\vr} \As (\rn)$, $\La_{\vr} \As (\rn)$, 
$A \in \{B,F\}$, be the spaces in \eqref{2.44} of the related $\vr$--clan with $-n < \vr <0$. From \cite[Proposition 3.54, 
p.\,92]{T14} and \eqref{2.31} it follows that
\begin{\eq}   \label{4.34}
\La^{\vr} \As (\rn) \hra \Cc^{s- \frac{|\vr|}{p}} (\rn).
\end{\eq}
If $s > \frac{|\vr|}{p}$, then $\Cc^{s- \frac{|\vr|}{p}} (\rn) \hra C(\rn)$, where $C(\rn) = C^0 (\rn)$ is the usual space of all
bounded continuous functions in $\rn$ normed according to \eqref{4.4}. Then the trace on $\Rn$ makes sense pointwise. Combined with
the (almost but not totally obvious) restriction
\begin{\eq}  \label{4.35}
\tr \Cc^\sigma (\rn) = \Cc^\sigma (\Rn), \qquad \sigma >0,
\end{\eq}
\cite[Theorem 2.13, p.\,32]{T20}, one obtains
\begin{\eq}   \label{4.36}
\tr: \quad \La^{\vr} \As (\rn) \hra \Cc^{s- \frac{|\vr|}{p}} (\Rn), \quad 0<p<\infty, \ 0<q \le \infty, \ s> |\vr|/p.
\end{\eq}
We wish to show that the space on the right--hand side is the related trace space and that there is a common linear and bounded
extension operator which does not only apply to a fixed space $\La^{\vr_0} A^{s_0}_{p_0,q_0} (\rn)$ but also to spaces $\La^{\vr}
\As (\rn)$ with \new{parameters $(s,p,q,\vr)$ close to $(s_0, p_0, q_0, \vr_0)$}\ignore{neighbouring $(s,p,q,\vr)$ compared with $(s_0, p_0, q_0, \vr_0)$}. We will not stress this point, but it justifies
interpolation ensuring below that the extension operator applies also to related spaces $\La_{\vr} \Bs (\rn)$. 

\begin{theorem}   \label{T4.12}
Let $n \ge 2$, $-1 <\vr <0$, 
\begin{\eq}   \label{4.37}
0<p<\infty, \qquad 0<q \le \infty \qquad \text{and} \qquad s > |\vr|/p.
\end{\eq}
Then
\begin{\eq}   \label{4.38}
\tr: \quad \La^{\vr} \As (\rn) \hra \Cc^{s- \frac{|\vr|}{p}} (\Rn ),\quad \new{A\in\{B,F\},}
\end{\eq}
\begin{\eq}   \label{4.39}
\tr: \quad \La_{\vr} \Bs (\rn) \hra B^{s- \frac{|\vr|}{p}}_{\infty, q} (\Rn),
\end{\eq}
\new{
and
  \begin{\eq}   \label{4.39a}
\tr: \quad \La_{\vr} \Fs (\rn) \hra \Cc^{s- \frac{|\vr|}{p}} (\Rn).
\end{\eq}}%
Furthermore there are linear and  bounded extension operators $\ext$ with
\begin{\eq}   \label{4.40}
\tr \circ \ext = \id \quad \text{identity in} \quad \Cc^{s- \frac{|\vr|}{p}} (\Rn) \ \text{and} \ B^{s- \frac{|\vr|}{p}}_{\infty,q} 
(\Rn)
\end{\eq}
such that
\begin{\eq}   \label{4.41}
\ext: \quad \Cc^{s- \frac{|\vr|}{p}} (\Rn) \hra \La^{\vr} \As (\rn),
\end{\eq}
\begin{\eq}   \label{4.42}
\ext: \quad B^{s- \frac{|\vr|}{p}}_{\infty, q} (\Rn) \hra \La_{\vr} \Bs (\rn),
\end{\eq}
\new{and
\begin{\eq}   \label{4.42a}
\ext: \quad \Cc^{s- \frac{|\vr|}{p}} (\Rn) \hra \La_{\vr} \Fs (\rn).
\end{\eq}}
\end{theorem}

\begin{proof}
  {\em Step 1.} By \eqref{4.36} one has \eqref{4.38}. Then \eqref{3.14} and and its classical counterpart prove \eqref{4.39}
 by real interpolation. 
\cm
{\em Step 2.} We prove \eqref{4.41} with $A=B$. For this purpose we rely on the same wavelet arguments as in the proof of
\cite[Theorem 2.13, pp.\,32--34]{T20}. We expand $f\in \Cc^\sigma (\Rn)$, $\sigma = s - \frac{|\vr|}{p} >0$ in $\Rn$ similarly as in
\eqref{3.7}--\eqref{3.9},
\begin{\eq}   \label{4.43}
f = \sum_{\substack{j \in \no, G\in G^j, \\ m\in \Zn}} \lambda^{j,G}_m \, \psi^j_{G,m} (x'), \quad \psi^j_{G,m} (x') = 
\prod^{n-1}_{l=1} \psi_{G_l} \big( 2^j x_l - m_l \big),
\end{\eq}
where $x' = (x_1, \ldots, x_{n-1})$,
\begin{\eq}   \label{4.44}
\lambda^{j,G}_m = \lambda^{j,G}_m (f) = 2^{j(n-1)} \big( f, \psi^j_{G,m} \big), \qquad j \in \no, \quad m\in \Zn,
\end{\eq}
with $G^0 = \{F,M \}^{n-1}$, $G^j = \{F,M \}^{n-1, *}$ for $j\in \nat$ and
\begin{\eq}   \label{4.45}
\| f \, | \Cc^\sigma (\Rn) \| \sim \sup_{\substack{j \in \no, G\in G^j, \\ m\in \Zn}} 2^{j\sigma} \big| \lambda^{j,G}_m (f) \big|.
\end{\eq}
We use now a simplified version of the wavelet--friendly extension as constructed in \cite[Section 5.1.3, pp.\,139--147]{T08}. Let
\begin{\eq}   \label{4.46}
\chi \in D(\real) = C^\infty_0 (\real), \quad \supp \chi \subset (-1,1), \quad \chi (t) =1 \ \text{if} \ |t| \le 1/2,
\end{\eq}
and $x= (x', x_n)$, $x' \in \Rn$, $x_n \in \real$. Then
\begin{\eq}   \label{4.47}
g = \ext f = \sum_{\substack{j \in \no, G\in G^j, \\ m\in \Zn}} \lambda^{j,G}_m (f) \, \psi^j_{G,m} (x') \, \chi \big(2^j x_n \big)
\end{\eq}
with
\begin{\eq}    \label{4.48}
g(x',0) = f(x'), \qquad x' \in \Rn,
\end{\eq}
is an atomic expansion for the spaces $\La^{\vr} \Bs (\rn)$ as described in \cite[Theorem 3.33, p.\,67]{T14} (including all requested
moment conditions, if required). In particular, according to \eqref{3.10} one has to estimate
\begin{\eq}    \label{4.49}
\begin{aligned}
&\| g \, | \La^{\vr} \Bs (\rn) \| \\
&\le c \, \sup_{J\in \ganz, M\in \zn} 2^{\frac{J}{p} (n+ \vr)} \bigg( \sum_{j \ge J^+} 2^{j(s- \frac{n}{p})q} 
\Big( \sum_{\substack{m: Q_{j,m}
\subset Q_{J,M}, \\ G\in G^j}} \big| \lambda^{j,G}_m (f) \big|^p \Big)^{q/p} \bigg)^{1/q}.
\end{aligned}
\end{\eq}
We may assume $q<\infty$. By \eqref{4.43} and \eqref{4.45} with $\sigma = s - \frac{|\vr|}{p}$ only $\sim 2^{(n-1)(j-J)}$ terms
contribute to the last sum. Then
\begin{\eq}   \label{4.50}
\begin{aligned}
& \| g \, | \La^{\vr} \Bs (\rn) \|^q \\
&\le c \, \|f \, | \Cc^{s- \frac{|\vr|}{p}} (\Rn) \|^q \, \sup_{J\in \ganz, M \in \zn} 2^{\frac{Jq}{p} (n+\vr)} \sum_{j \ge J}
2^{j(s- \frac{n}{p})q - j(s- \frac{|\vr|}{p})q} \, 2^{(n-1) \frac{q}{p} (j- J)} \\
& \le c' \, \|f \, | \Cc^{s- \frac{|\vr|}{p}} (\Rn) \|^q \, \sup_{J\in \ganz, M\in \zn} 2^{\frac{Jq}{p} (1+\vr)} \sum_{j \ge J}
2^{-j \frac{q}{p} ( 1 - |\vr|)} \\
& \le c'' \|f \, | \Cc^{s- \frac{|\vr|}{p}} (\Rn) \|^q
\end{aligned}
\end{\eq}
where we used $0<|\vr| <1$. This proves \eqref{4.41} with $A = B$.
\cm
{\em Step 3.} The corresponding assertion for the spaces in \eqref{4.41} with $A=F$ follows from \eqref{2.54}. The real interpolation
\eqref{3.14} and its classical counterpart prove \eqref{4.42}. 
\new{Clearly, \eqref{4.39a} and \eqref{4.42a} are consequences of \eqref{4.38} and \eqref{4.41} together with the coincidence \eqref{2.42}.}
\end{proof}

The above arguments can also be used to justify \eqref{4.30}--\eqref{4.32} and to show that $\ext$ as constructed in \eqref{4.47} is also
an extension operator for the related trace spaces. As there it applies  not only to a fixed space $\La^{\vr_0} A^{s_0}_{p_0, q_0}
(\rn)$ or $\La_{\vr_0} B^{s_0}_{p_0, q_0} (\rn)$ but also to corresponding spaces with neighbouring $(s,p,q,\vr)$. We formulate the
outcome and indicate the respective modifications. We concentrate again on the three families in \eqref{2.44} of the related 
$\vr$--clan. Let $\sigma^{n-1}_p$ be as in \eqref{4.29}.

\begin{theorem}  \label{T4.13}
Let $n \ge 2$, $-n < \vr <-1$,
\begin{\eq}   \label{4.51}
0<p<\infty, \quad 0<q \le \infty \quad \text{and} \quad s - \frac{1}{p} > \sigma^{n-1}_p.
\end{\eq}
Then
\begin{\eq}   \label{4.52}
\tr:\quad \La^{\vr} \Bs (\rn) \hra \La^{\vr +1} B^{s- \frac{1}{p}}_{p,q} (\Rn),
\end{\eq}
\begin{\eq}   \label{4.53}
\tr:\quad \La^{\vr} \Fs (\rn) \hra \La^{\vr +1} B^{s- \frac{1}{p}}_{p,p} (\Rn),
\end{\eq}
\begin{\eq}   \label{4.54}
\tr:\quad \La_{\vr} \Bs (\rn) \hra \La_{\vr +1} B^{s- \frac{1}{p}}_{p,q} (\Rn).
\end{\eq}
\new{and
\begin{\eq}   \label{4.54a}
\tr:\quad \La_{\vr} \Fs (\rn) \hra \La^{\vr +1} B^{s- \frac{1}{p}}_{p,p} (\Rn).
\end{\eq}}
Furthermore there are linear and bounded extension operators $\ext$ with
\begin{\eq}   \label{4.55}
\tr \circ \ext = \id \quad \text{identity in} \quad \La^{\vr+1} B^{s- \frac{1}{p}}_{p,q} (\Rn) \ \text{and} \ \La_{\vr +1} B^{s- 
\frac{1}{p}}_{p,q} (\Rn)
\end{\eq} 
such that
\begin{\eq}  \label{4.56}
\ext: \quad \La^{\vr +1} B^{s- \frac{1}{p}}_{p,q} (\Rn) \hra \La^{\vr} \Bs (\rn),
\end{\eq}
\begin{\eq}  \label{4.57}
\ext: \quad \La^{\vr +1} B^{s- \frac{1}{p}}_{p,p} (\Rn) \hra \La^{\vr} \Fs (\rn),
\end{\eq}
\begin{\eq}  \label{4.58}
\ext: \quad \La_{\vr +1} B^{s- \frac{1}{p}}_{p,q} (\Rn) \hra \La_{\vr} \Bs (\rn),
\end{\eq}
\new{and
\begin{\eq}  \label{4.58a}
\ext: \quad \La^{\vr +1} B^{s- \frac{1}{p}}_{p,p} (\Rn) \hra \La_{\vr} \Fs (\rn),
\end{\eq}  }
\end{theorem}

\begin{proof}
{\em Step 1.} We prove \eqref{4.52} and expand $f\in \La^{\vr} \Bs (\rn)$ as in \eqref{3.7}, \eqref{3.8} with the equivalent 
quasi--norm \eqref{3.10}.
Then $f(x',0)$ looks like \eqref{4.43}. But one must be aware that
there are terms with $G_l =F$ for all $l= 1,\ldots,n-1$. 
Then one has no moment conditions and the outcome must be considered as an
expansion in $\Rn$ by atoms without moment conditions. The prototypes of the remaining terms on the right--hand side in \eqref{3.10}
are now $M = (M',0)$, $M' \in \Zn$ and related counterparts of $Q_{j,m} \subset Q_{J,M}$. Using
\begin{\eq}   \label{4.59}
n+ \vr = n-1 + (\vr +1) \qquad \text{and} \qquad s - \frac{n}{p} = s- \frac{1}{p} - \frac{n-1}{p}
\end{\eq}
then $-n +1 \le \vr +1 <0$ and $s - \frac{1}{p} > \sigma^{n-1}_p$ ensure that $f(x',0)$ is an expansion  in $\Rn$ by atoms without
moment conditions. Now it follows from \cite[Theorem 3.33, p.\,67]{T14} that $ f(x',0) \in \La^{\vr+1} B^{s- \frac{1}{p}}_{p,q} (\Rn)$
and
\begin{\eq}   \label{4.60}
\| f(x',0) \, | \La^{\vr+1} B^{s- \frac{1}{p}}_{p,q} (\Rn) \| \le c \, \| f \, | \La^{\vr} \Bs (\rn) \|.
\end{\eq}
This proves \eqref{4.52}. As far as \eqref{4.53} is concerned we rely on the wavelet expansion for $f\in \La^{\vr} \Fs (\rn) =
L^r \Fs (\rn)$ with $r = \vr/p$
in Theorem \cite[Theorem 3.26, p.\,64]{T14}. Then the sequence space $\La^{\vr} b^s_{p,q} (\rn)$ on the right--hand side of
\eqref{3.10} must be replaced by its counterpart $\La^{\vr} f^s_{p,q} (\rn) = L^{\vr/p} f^s_{p,q} (\rn)$ in \cite[Definition 3.24,
p.\,63]{T14}. But for the restriction of these wavelet expansions to terms contributing to $f(x',0)$, $x' \in \Rn$, one is in the same
situation as in \cite[p.\,145, (5.98)]{T08} and the references given there. For these terms the corresponding quasi--norm $\La^{\vr}
f^s_{p,q} (\rn)$ is equivalent to $\La^{\vr} b^s_{p,p} (\rn)$, which means the right--hand side of \eqref{3.10} with $q=p$. But then
\eqref{4.53} follows from the above arguments applied to $\La^{\vr} B^s_{p,p} (\rn)$. The interpolation 
\eqref{3.14} extends  the above assertion to \eqref{4.54}.
\cm
{\em Step 2.} Let $\ext$ be the same extension operator as in \eqref{4.46}--\eqref{4.48} (common for neighbouring spaces) with $f\in
\La^{\vr+1} B^{s- \frac{1}{p}}_{p,q} (\Rn)$. Then \eqref{4.56} follows from \eqref{4.59} applied to \eqref{3.10} with $g$ in place
of $f$ and its counterpart on $\Rn$ for $f$. Recall that the atoms in \eqref{4.47} have all moment conditions one needs. Then 
\eqref{4.57} follows again from \eqref{3.10} with $q=p <\infty$ and the comments at the end of the preceding Step 1. Finally 
\eqref{4.58} is a matter of interpolation in the same way as above.
\new{Again \eqref{4.54a} and \eqref{4.58a} result from \eqref{4.53}, \eqref{4.57}, and the coincidences \eqref{2.42}.}
\end{proof}

\begin{remark}   \label{R4.14}
The proof of the atomic expansions for the spaces $\La^{\vr}\As (\rn) = L^{\vr/p} \As (\rn)$ 
in \cite[Theorem 3.33, pp.\,67/68]{T14} based on \cite[Theorem
1.37, pp.\,28--31]{T13} relies on the related technicalities for the classical spaces $\As (\rn)$. This explains that one still
requires $s> \sigma^n_p$ with $t=n$ in
\begin{\eq}   \label{4.61}
\sigma^t_p = t \Big( \max \big(\frac{1}{p}, 1 \big) - 1 \Big), \qquad t \ge 0, \quad 0<p<\infty,
\end{\eq}
to ensure atomic expansions in $\La^{\vr} \Bs (\rn)$ without moment conditions. One may ask whether this
condition can be replaced by $s > \sigma^{|\vr|}_p$  as suggested by the Slope-$n$-Rule, see \new{Slope Rules~\ref{slope_rules}(ii)}. It would be desirable
to confirm this replacement. If so, then $s -
\frac{1}{p} > \sigma^{n-1}_p$ in \eqref{4.51} originating from the right--hand sides of \eqref{4.52}, \eqref{4.53}, could be 
replaced by $s - \frac{1}{p} > \sigma^{|\vr| -1}_p$ where $-n < \vr < -1$. Combined with Theorem~\ref{T4.12} one obtains the 
reasonable restriction
\begin{\eq}   \label{4.62}
s > \frac{1}{p} \min \big( |\vr|,1 \big) + \sigma^{\max (|\vr|,1) -1}_p, \qquad 0<p<\infty, \quad -n <\vr <0,
\end{\eq}
for traces of the spaces in \eqref{2.44} of the $\vr$--clan in \rn. If confirmed it would be an outstanding example of the 
Slope--1--Rule, see \new{Slope Rules~\ref{slope_rules}(i)}, in the interpretation that $|\vr|=1$ is a breaking point. Conditions of type \eqref{4.62} seem
to be quite natural. They appear also in other occasions, for example in Theorem~\ref{T5.29} below.
\end{remark}

The Theorems~\ref{T4.12} and \ref{T4.13} apply to all $\vr$--clans, $-n < \vr <0$, 
according to Definition~\ref{D2.11}(ii) with exception of $\vr = -1
$. It is quite clear that $|\vr|=1$ is a breaking point for  trace spaces. Theorem~\ref{T4.13} cannot be extended to $\La_{-1} \Bs 
(\rn)$ because possible trace spaces $\La_0 B^{s- \frac{1}{p}}_{p,q} (\Rn)$ are not covered by Definition~\ref{D2.5}. The situation
is better for the spaces $\La^{\vr} \As (\rn)$ as introduced in Definition~\ref{D2.7} for all $\vr$ with $-n \le \vr <\infty$. Furthermore one can extend the arguments in the proof of Theorem~\ref{T4.13} to corresponding spaces  with $\vr =-1$. We formulate
the outcome.

\begin{corollary}   \label{C4.15}
Let $n \ge 2$,
\begin{\eq}   \label{4.63}
0<p<\infty, \quad 0<q \le \infty \quad \text{and} \quad s- \frac{1}{p} > \sigma^{n-1}_p.
\end{\eq}
Then
\begin{\eq}  \label{4.64}
\tr: \quad \Lambda^{-1} \Bs (\rn) \hra \Lambda^0 B^{s- \frac{1}{p}}_{p,q} (\Rn)
\end{\eq}
and
\begin{\eq}  \label{4.65}
\tr: \quad \Lambda^{-1} \Fs (\rn) \hra \Lambda^0 B^{s- \frac{1}{p}}_{p,p} (\Rn).
\end{\eq}
Furthermore there are linear and bounded extension operators $\ext$ with
\begin{\eq}   \label{4.66}
\tr \circ \ext = \id \qquad \text{identity in} \quad \Lambda^0 B^{s- \frac{1}{p}}_{p,q} (\Rn)
\end{\eq}
such that
\begin{\eq}   \label{4.67}
\ext: \quad \Lambda^0 B^{s- \frac{1}{p}}_{p,q} (\Rn) \hra \La^{-1} \Bs (\rn)
\end{\eq}
and
\begin{\eq}   \label{4.68}
\ext: \quad  \Lambda^0 B^{s- \frac{1}{p}}_{p,p} (\Rn) \hra \La^{-1} \Fs (\rn).
\end{\eq}
\end{corollary}

\begin{proof}
 As mentioned in Remark~\ref{R3.2}, Proposition~\ref{P3.1} remains valid for the spaces $\La^0 \Bs (\rn)$. Then the
arguments in the proof of Theorem~\ref{T4.13} apply also to the above spaces, including the references about atomic expansions and the
indicated technicalities as far as the $F$--spaces are concerned.
\end{proof}

\begin{remark}   \label{R4.16}
Some consequences of the above corollary might be of interest. It follows from \eqref{2.34} and \eqref{2.36} that
\begin{\eq}   \label{4.69}
\tr \La^{-1} B^s_{p,\infty} (\rn) = \Cc^{s- \frac{1}{p}} (\Rn), \qquad 0<p<\infty, \quad s > \frac{1}{p}
\end{\eq}
and
\begin{\eq}   \label{4.70}
\tr \Lambda^{-1} \Fs (\rn) = F^{s- \frac{1}{p}}_{\infty,p} (\Rn), \quad 0<p<\infty, \quad 0<q\le \infty, \quad s - \frac{1}{p} 
>\sigma^{n-1}_p.
\end{\eq}
In \eqref{4.69} one does not need the sharper restriction $s - \frac{1}{p} > \sigma^{n-1}_p$ because no moment conditions for
atomic expansions in $\Cc^\sigma (\Rn)$ with $\sigma >0$ are requested. Furthermore, \eqref{4.50} appropriately modified can also be
applied to $|\vr| =1$ and $q=\infty$. In other words, \eqref{4.69} complements Theorem~\ref{T4.12}.
\end{remark}

\section{Essential features}    \label{S5}
\subsection{Embeddings in $L_\infty (\rn)$ and $C(\rn)$}     \label{S5.1}
First we deal with embeddings of the spaces covered by Definition~\ref{D2.11} in $L_\infty (\rn)$ and $C(\rn)$ as target spaces. We
always assume $f(x)=0$ if $x \in \rn$ is not a Lebesgue point of the locally Lebesgue--integrable function function $f(x)$ in \rn. Then
\begin{\eq}  \label{5.1}
\| f \, | L_\infty (\rn) \| = \sup_{x\in \rn} |f(x)|
\end{\eq}
is the norm in the usual space $L_\infty (\rn)$ of (essentially) bounded complex--valued functions in \rn. Let $C(\rn)$ be the space of
all continuous bounded complex--valued functions in $\rn$ normed by \eqref{5.1}. Both $L_\infty (\rn)$ and $C(\rn)$ are considered as
subspaces of $S'(\rn)$. As above $\hra$ indicates continuous embedding. We concentrate on the $\vr$--clan with $-n <\vr <0$
consisting of the three
families \eqref{2.44}, \eqref{2.45}. But it comes out that there are some peculiarities compared with the related embeddings of the
classical spaces \eqref{2.38}, the $n$--clan. This may justify that we first recall the corresponding assertion.

\begin{proposition}   \label{P5.1}
Let $\As (\rn)$ with $A \in \{B,F \}$, $s\in \real$, $0<p\le \infty$ and $0<q\le \infty$ be the spaces according to Definition
\ref{D2.1}, the $n$--clan \eqref{2.38}. Then
\begin{\eq}   \label{5.2}
\Bs (\rn) \hra L_\infty (\rn)
\end{\eq}
if, and only if,
\begin{\eq} \label{5.3}
\begin{cases}
\text{either} &s> \frac{n}{p}, \ 0<q\le \infty, \\
\text{or}     & s = \frac{n}{p}, \ 0<q \le 1,
\end{cases}
\end{\eq}
and
\begin{\eq}   \label{5.4}
\Fs (\rn) \hra L_\infty (\rn)
\end{\eq}
if, and only if,
\begin{\eq} \label{5.5}
\begin{cases}
\text{either} &s > \frac{n}{p}, \ 0<q \le \infty, \\
\text{or}     & s= \frac{n}{p}, \ 0<p \le 1, \ 0<q\le \infty.
\end{cases}
\end{\eq}
Furthermore one can replace $L_\infty (\rn)$ in \eqref{5.2} and \eqref{5.4} by $C(\rn)$.
\end{proposition}

\begin{remark}    \label{R5.2}
This coincides with {\cite[Theorem 2.3, pp.\,22--23]{T20}} where we complemented already known assertions for the spaces $\As (\rn)$,
$A \in \{B,F \}$, with $p<\infty$ for the $F$--spaces by $F^s_{\infty,q} (\rn)$. There one finds also related references, like \cite[Theorem~3.3.1]{SiT95}.
\end{remark}

For the spaces in \eqref{2.44}, \eqref{2.45} of the $\vr$--clan according to Definition~\ref{D2.11}(ii) with $-n < \vr <0$ one has the
following counterpart.

\begin{theorem}   \label{T5.3}
Let
\begin{\eq}  \label{5.6}
s \in \real, \quad 0<p<\infty, \quad 0<q \le \infty \quad \text{and} \quad -n < \vr <0.
\end{\eq}
Then
\begin{\eq}  \label{5.7}
\La^{\vr} \As (\rn) \hra L_\infty(\rn) \quad \text{if, and only if,} \quad s > \frac{|\vr|}{p},
\end{\eq}
$A \in \{B,F \}$, and
\begin{\eq}   \label{5.8}
\La_{\vr} \Bs (\rn) \hra L_\infty(\rn) \quad \text{if, and only if,} \quad
\begin{cases}
s> \frac{|\vr|}{p}, & 0<q \le \infty, \quad\text{or}\\
s= \frac{|\vr|}{p}, & 0< q \le 1,
\end{cases}
\end{\eq}
\begin{\eq}  \label{5.8a}
\new{\La_{\vr} \Fs (\rn) \hra L_\infty(\rn) \quad \text{if, and only if,} \quad s > \frac{|\vr|}{p}.}
\end{\eq}
Furthermore, one can replace $L_\infty (\rn)$ in \eqref{5.7}--\eqref{5.8a} by $C(\rn)$.
\end{theorem}

\begin{remark}   \label{R5.4}
The embedding \eqref{5.7} is covered by \cite[Proposition 5.4, p.\,334]{YHSY15} whereas \eqref{5.8} goes back to \cite[Proposition
  5.5, p.\,140]{HaS13}, \new{and \eqref{5.8a} to \cite[Proposition~3.8]{HaS14}, all} reformulated according to \eqref{2.24} and \eqref{2.29}, \eqref{2.30}. They may also be found in 
\cite[Proposition 3.1, p.\,225]{HMS16}. These embeddings, compared with their classical counterparts in Proposition~\ref{P5.1}, are
examples of the Slope--$n$--Rule, see \new{Slope Rules~\ref{slope_rules}(ii)}. But there are also some peculiarities as far as limiting
embeddings with $s = \frac{n}{p}$ and $s= \frac{|\vr|}{p}$ are concerned. Further related discussions  especially about limiting
embeddings may be found in \cite{HaS14}, \cite{YHMSY15} and \cite{HMS20}.
\new{One may extend \eqref{5.7} to $\vr=0$ in the following way. The $0$--clan consisting of the spaces $\La^0\Bs(\rn)$ in \eqref{2.41}, and also the spaces $\La^0\Fs(\rn)$ given by \eqref{2.35}, they can be incorporated: using the coincidences  \eqref{2.29} with \eqref{2.30}, as well as \eqref{ftbt} and \eqref{2.35}, then \cite[Proposition~4.1]{YHMSY15} yields
  \[
  \La^0 \As(\rn) \hra L_\infty(\rn) \quad\text{if, any only if,}\quad s>0,
  \]
where again $A\in\{B,F\}$ and $L_\infty(\rn) $ can be replaced by $C(\rn)$.} %
  The limiting embedding in \eqref{5.8} compared with \eqref{5.7} illuminates
also the strict embedding \eqref{2.56}.
\end{remark}

\begin{minipage}{\textwidth}
\hfill  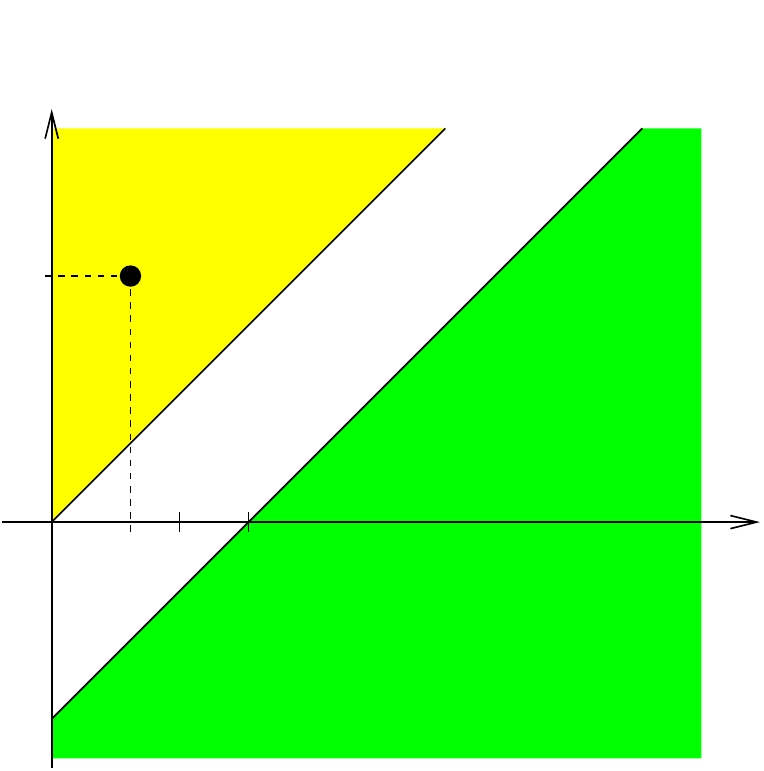\hfill~\\  
~\hspace*{\fill}\unterbild{fig-1a}
\end{minipage}

\subsection{Embeddings in the space of locally integrable functions}   \label{S5.2}
Let again $A(\rn)$ be a quasi--Banach space with
\begin{\eq}   \label{5.9}
S(\rn) \hra A(\rn) \hra S'(\rn).
\end{\eq}
Recall that $L^{\loc}_1 (\rn)$ collects all Lebesgue--measurable (complex--valued) functions in $\rn$ which are integrable on any
bounded domain in \rn. Then
\begin{\eq}   \label{5.10}
A(\rn) \subset L^{\loc}_1 (\rn)
\end{\eq}
means that any $f\in A(\rn)$ can be represented as a locally integrable function such that
\begin{\eq}   \label{5.11}
f(\vp) = \int_{\rn} f(x) \, \vp(x) \, \di x, \qquad \vp \in D(\rn) = C^\infty_0 (\rn).
\end{\eq}
Let as before
\begin{\eq}   \label{5.12}
\sigma^t_p = t \Big( \max \big( \frac{1}{p}, 1 \big) -1 \Big), \qquad  t \ge 0, \quad 0<p \le \infty.
\end{\eq}
First we recall the final conditions ensuring \eqref{5.10} for the $n$--clan, consisting of the classical spaces in \eqref{2.38}.

\begin{proposition}   \label{P5.5}
Let $\As (\rn)$ with $A \in \{B,F \}$, $s\in \real$, $0<p \le \infty$, $0<q \le \infty$, be the spaces according Definition 
\ref{D2.1}, the $n$--clan \eqref{2.38}. Then
\begin{\eq}   \label{5.13}
\Bs (\rn) \subset L^{\loc}_1 (\rn)
\end{\eq}
if, and only if,
\begin{\eq}   \label{5.14}
\begin{cases}
\text{either} &0<p \le \infty, \ s> \sigma^{n}_p, \ 0<q\le \infty, \\
\text{or}     &0<p \le 1, \ s= \sigma^{n}_p, \ 0<q \le 1, \\
\text{or}     &1<p \le \infty, \ s=0, \ 0<q \le \min(p,2), 
\end{cases}
\end{\eq}
and
\begin{\eq}    \label{5.15}
\Fs (\rn) \subset L^{\loc}_1 (\rn)
\end{\eq}
if, and only if,
\begin{\eq}   \label{5.16}
\begin{cases}
\text{either} &0<p<1, \ s \ge \sigma^{n}_p, \ 0<q\le \infty, \\
\text{or}     &1\le p \le \infty, \ s>0, \ 0<q \le \infty, \\
\text{or}     &1\le p \le \infty, \ s=0, \ 0<q \le 2.
\end{cases}
\end{\eq}
\end{proposition}

\begin{remark}   \label{R5.6}
This coincides with {\cite[Theorem 2.4, pp.\,23--24]{T20}} where already known assertions for the spaces $\As (\rn)$ with $p<\infty$ for
$F$--spaces, cf. \cite[Theorem~3.3.2]{SiT95}, had been complemented by $F^s_{\infty,q} (\rn)$. There one finds further references.
\end{remark}

The Slope--$n$--Rule, recall \new{Slope Rules~\ref{slope_rules}(ii)}, suggests that one has to replace the breaking line $s= \sigma^n_p$ in the 
$\big(\frac{1}{p}, s \big)$--diagram for the spaces $\As (\rn)$ by $s = \sigma^{|\vr|}_p$ for the spaces \eqref{2.44}, \eqref{2.45}
of the $\vr$--clan, $-n <\vr <0$. This is the case. But there is so far no final assertion what happens for the limiting spaces with
$s= \sigma^{|\vr|}_p$. This may justify that we concentrate on the non--limiting spaces with $s \not= \sigma^{|\vr|}_p$.

\begin{theorem}   \label{T5.7}
Let
\begin{\eq}   \label{5.17}
s \in \real, \quad 0<p<\infty, \quad 0<q \le \infty \quad \text{and} \quad -n<\vr <0.
\end{\eq}
Then
\begin{\eq}   \label{5.18}
\rhoAs(\rn) \subset L^{\loc}_1 (\rn)  \quad \text{if} \quad 
s>\sigma^{|\vr|}_p
\end{\eq}
and
\begin{\eq}   \label{5.19}
\rhoAs(\rn) \not\subset L^{\loc}_1 (\rn)  \quad \text{if} \quad 
s<\sigma^{|\vr|}_p,
\end{\eq}
$A\in \{B,F \}$.
\end{theorem} 

\begin{remark}   \label{R5.8}
These assertions are covered by \cite[Theorems 3.3, 3.6, pp.\,228, 232]{HMS16} reformulated according to \eqref{2.24} and 
\eqref{2.29}, \eqref{2.30}. In \cite[Theorems 3.4, 3.8, pp.\,228, 233]{HMS16} there are detailed discussions what happens if $s=
\sigma^{|\vr|}_p$, at least partly comparable with $s= \sigma^n_p$ in Proposition~\ref{P5.5} but less final. \new{In \cite[Theorems~3.2, 3.4]{HMS20} the complete characterisation in case of $\rhoAs=\La_\vr\As$ could be obtained for the limiting case $s=\sigma^{|\vr|}_p$, it reads as
 \begin{\eq}\label{DDH-6} 
  \La_\vr B^{\sigma^{|\vr|}_p}_{p,q}(\rn) \subset \Lloc(\rn) \quad \text{if, and only if,}\quad 0<q\leq \min(\max(p,1),2),
\end{\eq}
and  
  \begin{\eq}\label{DDH-4}
  \La_\vr F^{\sigma^{|\vr|}_p}_{p,q}(\rn) \subset \Lloc(\rn) \quad \text{if, and only if,}\quad \begin{cases} \text{either}\quad p\geq 1 \quad \text{and}&  0<q\leq 2, \\ \text{or}\quad 0<p<1.\end{cases}
  \end{\eq}
  This is literally the extension of \eqref{5.13}, \eqref{5.14} and \eqref{5.15}, \eqref{5.16} when $s=\sigma^n_p$ in the classical case there. The coincidence \eqref{2.42} thus answers the question for $\La^\vr\Fs(\rn)$ in case of  $s=\sigma^{|\vr|}_p$, too, that is,
  \begin{\eq}\label{DDH-5}
  \La^\vr F^{\sigma^{|\vr|}_p}_{p,q}(\rn) \subset \Lloc(\rn) \quad \text{if, and only if,}\quad \begin{cases} \text{either}\quad p\geq 1 \quad \text{and}&  0<q\leq 2, \\ \text{or}\quad 0<p<1.\end{cases}
  \end{\eq}
  The only gap for the moment concerns the characterisation of $ \La^\vr B^{\sigma^{|\vr|}_p}_{p,q}(\rn) \subset \Lloc(\rn)$, i.e., the counterpart of \eqref{DDH-6}, we refer to \cite[Theorem 3.8(i)]{HMS16} for the details.
}
\end{remark}

\begin{minipage}{\textwidth}
\hfill  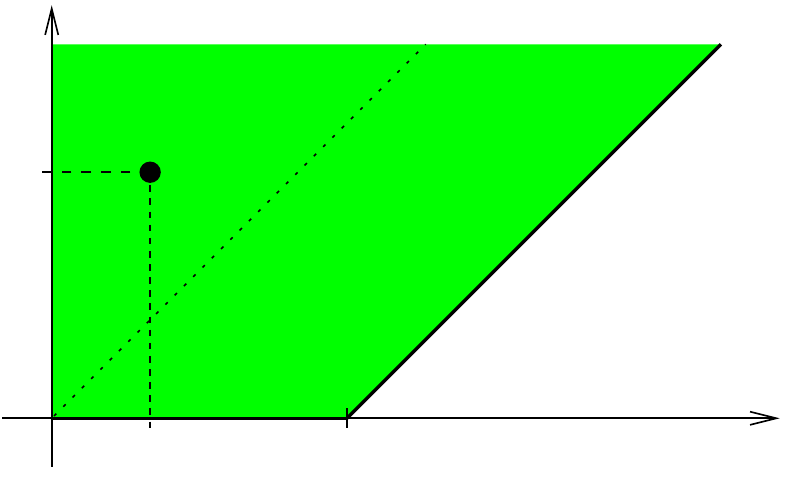\hfill~\\  
~\hspace*{\fill}\unterbild{fig-1b}
\end{minipage}

\subsection{Multiplication algebras}   \label{S5.3}
Recall that a quasi--Banach space $A(\rn)$ on $\rn$ with
\begin{\eq}   \label{5.20}
S(\rn) \hra A(\rn) \hra S' (\rn), \qquad A(\rn) \subset L^{\loc}_1 (\rn),
\end{\eq}
is said to be a multiplication algebra if $f_1 f_2 \in A(\rn)$ whenever $f_1 \in A(\rn)$, $f_2 \in A(\rn)$ and if there is a constant
$c>0$ such that
\begin{\eq}  \label{5.21}
\| f_1 f_2 \, | A(\rn) \| \le c \, \|f_1 \, | A(\rn) \| \cdot \| f_2 \, | A(\rn) \|
\end{\eq}
for all $f_1 \in A(\rn)$, $f_2 \in A(\rn)$. This coincides with {\cite[Definition 2.37, p.\,44]{T20}}. There one finds related
 (historical)
references, also about possible applications in the theory of distinguished non--linear PDEs, including the Navier--Stokes equations.

\begin{proposition}   \label{P5.9}
Let $A(\rn)$ be a space covered by Definition~\ref{D2.11} satisfying in addition \eqref{5.20}. If $A(\rn)$ is a multiplication algebra, 
then
\begin{\eq}   \label{5.22}
A(\rn) \hra L_\infty (\rn).
\end{\eq}
\end{proposition}

\begin{remark}  \label{R5.10}
The short proof of this well--known assertion in \cite[Proposition 2.41, p.\,90]{T13} applies to all spaces in question. In other 
words, the search for possible multiplication algebras is now shifted from Proposition~\ref{P5.5}, Theorem~\ref{T5.7} to Proposition
\ref{P5.1}, Theorem~\ref{T5.3}. First we recall the final assertion for the classical spaces $\As (\rn)$.
\end{remark}

\begin{proposition}   \label{P5.11}
Let $\As (\rn)$ be the classical spaces in \eqref{2.38}, the $n$--clan, satisfying in addition \eqref{5.20} with $A(\rn) = \As(\rn)$.
Then $\Bs (\rn)$ is a multiplication algebra if, and only if,
\begin{\eq}   \label{5.23}
\begin{cases}
\text{either} & \text{$s>n/p$ where $0<p,q \le \infty$,} \\
\text{or}     & \text{$s=n/p$ where $0<p<\infty$, $0<q\le 1$},
\end{cases}
\end{\eq}
and $\Fs (\rn)$ is a multiplication algebra if, and only if,
\begin{\eq}   \label{5.24}
\begin{cases}
\text{either} & \text{$s> n/p$ where $0<p,q \le \infty$}, \\
\text{or}     & \text{$s = n/p$ where $0<p\le 1$, $0<q\le \infty$.}
\end{cases}
\end{\eq}
\end{proposition}

\begin{remark}   \label{R5.12}
This coincides with \cite[Theorem 2.41, p.\,45]{T20} where the related assertion for the spaces $\As (\rn)$ with $p<\infty$ for
$F$--spaces has been extended to $F^s_{\infty,q} (\rn)$. In {\cite[Remark 2.42, p.\,46]{T20}} one finds also detailed references about
the substantial history of this problem, including \cite{Mar87} as one of the earliest contributions.
\end{remark}

We ask for a counterpart of Proposition~\ref{P5.11} for the spaces \eqref{2.44} of the $\vr$--clan with $-n < \vr <0$. Let again
\begin{\eq}   \label{5.25}
\sigma^{|\vr|}_p = |\vr| \Big( \max \big( \frac{1}{p},1 \big) - 1 \Big), \qquad 0<p \le \infty,
\end{\eq}
be as in \eqref{5.12} with $t = |\vr|$.

\begin{theorem}   \label{T5.13}
Let $A(\rn)$ be one of the spaces in \eqref{2.44}, \eqref{2.45} of the $\vr$--clan with $-n< \vr <0$, satisfying in addition
\eqref{5.20}.
\bli
\item Then $A(\rn) = \La^{\vr} \As (\rn)$, $A \in \{B,F \}$, is a multiplication algebra if, and only if, $s> \frac{|\vr|}{p}$.
\item If $A(\rn) = \La_{\vr} \Bs (\rn)$ is a multiplication algebra, then either $s > \frac{|\vr|}{p}$, $0<q \le \infty$, or
$s= \frac{|\vr|}{p}$, $0< q \le 1$. Conversely, if $s > \frac{|\vr|}{p}$, $0<q \le \infty$, then $\La_{\vr} \Bs (\rn)$ is a 
multiplication algebra.
\new{\item Then $A(\rn) = \La_{\vr} \Fs (\rn)$ is a multiplication algebra if, and only if, $s> \frac{|\vr|}{p}$.}
\eli
\end{theorem}

\begin{proof}
{\em Step 1.} If $A(\rn)$ is a multiplication algebra, then the indicated conditions for $s,p,q$ follow from \eqref{5.22} and Theorem
\ref{T5.3}.
\cm
{\em Step 2.} If $s > \frac{|\vr|}{p}$, then it follows from \eqref{2.31} and \cite[Theorem 3.60, p.\,95]{T14} with a reference to
\cite[Theorem 2.43, p.\,91]{T13} that the spaces $\La^{\vr} \As (\rn)$ are multiplication algebras. This proves part (i). The
corresponding proof of the (local version) of the spaces $\La^{\vr} \Bs (\rn)$ in \cite[pp.\,91--93]{T13} relies on (the local 
version) of the wavelet expansion according to Proposition~\ref{P3.1} with the right--hand side of \eqref{3.10} as the related
sequence space. But it works also for the right--hand side of \eqref{3.11} at least as long as $s > \frac{|\vr|}{p}$.  This proves 
part (ii). \new{Part (iii) is a consequence of \eqref{2.42} and (i).}
\end{proof}

\begin{corollary}   \label{C5.14}
A space $\La^{\vr} \As (\rn)$, $A \in \{B,F \}$, of the $\vr$--clan according to Definition~\ref{D2.11}{\em (ii)} with $-n <\vr <0$ is
a multiplication algebra if, and only if, it is continuously embedded in $C(\rn)$.
\end{corollary}

\begin{proof} 
This follows immediately from Theorem~\ref{T5.3} and part (i) of the above theorem.
\end{proof}

\begin{remark}   \label{R5.15}
The above-mentioned somewhat sketchy proof in \cite[pp.\,91--93]{T13} relies decisively on $s > \frac{|\vr|}{p}$. But it is not clear
whether the related arguments can be extended to the spaces $\La_{\vr} \Bs (\rn)$ with $s = \frac{|\vr|}{p}$ and $0<q \le 1$. If this
is the case (as we expect), then Corollary~\ref{C5.14} can be extended to all spaces in \eqref{2.44} of the $\vr$--clan with $-n < \vr
<0$. But this is not obvious as the situation for the $n$--clan shows. The conditions \eqref{5.3} and \eqref{5.23} differ for $p=
\infty$. This somewhat tricky point has its own history and was clarified eventually in \cite[Remark 3.3.2, p.\,114 and Corollary
4.3.2, Remark 4.3.5, p.\,120]{SiT95}. In any case, Theorem~\ref{T5.13}, compared with Proposition~\ref{P5.11} is a further example 
of the Slope--$n$--Rule, see \new{Slope Rules~\ref{slope_rules}(ii)}.
\end{remark}

\subsection{The $\delta$--distribution}    \label{S5.4}
As a preparation of our later arguments we clarify to  which spaces covered by Definition~\ref{D2.11} the $\delta$--distribution,
$\delta (\vp) = \vp (0)$, $\vp \in S(\rn)$, belongs. First we deal with the $n$--clan, consisting of the spaces in \eqref{2.38},
complementing related very classical assertions for these spaces with $p<\infty$ for the $F$--spaces by $F^s_{\infty,q} (\rn)$.

\begin{proposition}  \label{P5.16}
Let $n\in \nat$ and let $\As (\rn)$ be the spaces according to \eqref{2.38}. Then
\begin{\eq}   \label{5.26}
\delta \in \Bs (\rn) \quad \text{if, and only if,} \quad
\begin{cases}
0<p \le \infty, & s< n (\frac{1}{p} -1), \ 0<q \le \infty,\quad\text{or} \\
0<p \le \infty, & s= n (\frac{1}{p} -1), \ q= \infty, 
\end{cases}
\end{\eq}
and 
\begin{\eq}   \label{5.27}
\delta \in \Fs (\rn) \quad \text{if, and only if,} \quad
\begin{cases}
0<p < \infty, & s< n (\frac{1}{p} -1), \ 0<q \le \infty,\quad\text{or} \\
p = \infty, & s \le -n, \ 0<q \le \infty.
\end{cases}
\end{\eq}
\end{proposition}

\begin{proof} Let $\{ \vp_j \}^\infty_{j=0}$ be the usual dyadic resolution of unity according to \eqref{2.3}--\eqref{2.5} and let
$\vp_j (\xi) = \vp (2^{-j} \xi)$, $j \in \nat$, $\xi \in \rn$. Using $\wh{\delta} = c \not= 0$ it follows from
\begin{\eq}   \label{5.28}
\big( \vp_j \wh{\delta} \big)^\vee (x) = c \, 2^{jn} \vp^\vee (2^j x), \qquad j\in \nat, \quad x\in \rn,
\end{\eq}
that
\begin{\eq}   \label{5.29}
\| \big( \vp_j \wh{\delta} \big)^\vee | L_p (\rn) \| \sim 2^{jn (1 - \frac{1}{p})}, \qquad j \in \nat, \quad 0<p \le \infty.
\end{\eq}
Inserted  in \eqref{2.10} one obtains
\begin{\eq}    \label{5.30}
\big\| \delta \, | \Bs (\rn) \big\| \sim \Big( \sum^\infty_{j=0} 2^{jq (s+n- \frac{n}{p} )} \Big)^{1/q}.
\end{\eq}
This  proves \eqref{5.26}. Let $0<p_0 <p< p_1 \le \infty$, $0<q \le \infty$ and
\begin{\eq}   \label{5.31}
s_0 - \frac{n}{p_0} = s - \frac{n}{p} = s_1 - \frac{n}{p_1}.
\end{\eq}
Then \eqref{5.27} follows from \eqref{5.26} and the well--known sharp embeddings
\begin{\eq}   \label{5.32}
B^{s_0}_{p_0,u} (\rn) \hra \Fs (\rn) \hra B^{s_1}_{p_1,v} (\rn)
\end{\eq}
with $0<u \le p \le v \le \infty$, often called Franke-Jawerth embedding nowadays to honour the contributions by Franke \cite{Franke86}, concerning the left-hand embedding, and Jawerth \cite{Jaw77} concerning the right-hand embedding. The sharp result is can be found \cite{SiT95}. Let us also refer to the more recent, elegant new proof of Vyb{\'\i}ral \cite{Vyb08}. We recommend 
{\cite[Theorem 2.5, p.\,25]{T20}} for further discussion. Let again $\chi_Q$ be the 
characteristic function of the unit cube $Q =(0,1)^n$. One has according to \cite[Proposition 2.43, p.\,46]{T20} 
\begin{\eq}   \label{5.33}
\chi_Q \in F^0_{\infty,q} (\rn) \qquad \text{for all} \quad 0<q \le \infty.
\end{\eq}
Let $\pa^n = \pa_1 \cdots \pa_n$. Then \eqref{5.27} with $p=\infty$ follows from
\begin{\eq}   \label{5.34}
\pa^n \chi_Q = \delta + \ldots \in F^{-n}_{\infty, q} (\rn),
\end{\eq}
where $+ \ldots$ indicates $\delta$--distributions with the remaining corners of $Q$ as off--points. Here we used in addition the 
special case
\begin{\eq}   \label{5.35}
\| f \, | F^0_{\infty,q} (\rn) \| \sim \sup_{0 \le |\alpha| \le n} \| \Dd^\alpha f \, | F^{-n}_{\infty, q} (\rn) \|
\end{\eq}
of \eqref{3.30}, $F^s_{\infty,q} (\rn) \hra \Cc^s (\rn)$, \cite[Theorem 2.9, p.\,26]{T20},  and \eqref{5.26}.
\end{proof}

\begin{remark}   \label{R5.17}
  We incorporated $F^s_{\infty,q} (\rn)$ into the otherwise classical assertion for the spaces $\As (\rn)$.
  But we have no reference at
hand. We inserted the above proposition because it will illuminate again the difference between the spaces in
\eqref{2.38} of the $n$--clan and the spaces in \eqref{2.44}, \eqref{2.45} of the $\vr$--clans, $-n<\vr<0$, in limiting situations.
\end{remark}

Next we deal with the $\vr$--clans according to Definition~\ref{D2.11}(ii) consisting of the three families \eqref{2.44} with $-n <
\vr <0$, $0<p<\infty$, $0<q \le \infty$ and $s\in \real$.

\begin{theorem}   \label{T5.18}
Let $n\in \nat$, $-n <\vr <0$, $0<p<\infty$. Let $0<q \le \infty$ and $s\in \real$. Then
\begin{\eq}    \label{5.36}
\delta \in \La^{\vr} \As (\rn) \quad \text{if, and only if,} \quad s \le \frac{|\vr|}{p} -n, \ 0<q \le \infty,
\end{\eq}
where $A \in \{B, F \}$,
\begin{\eq}   \label{5.37}
\delta \in \La_{\vr} \Bs (\rn) \quad \text{if, and only if,} \quad
\begin{cases}
s< \frac{|\vr|}{p} -n, &0<q \le \infty, \quad\text{or}\\
s= \frac{|\vr|}{p} -n, &q = \infty,
\end{cases}
\end{\eq}
\new{and
\begin{\eq}    \label{5.37a}
  \delta \in \La_{\vr} \Fs (\rn) \quad \text{if, and only if,} \quad s \le \frac{|\vr|}{p} -n, \ 0<q \le \infty.
\end{\eq}}
\end{theorem}

\begin{proof}
{\em Step 1.} We prove \eqref{5.37}. This means by \eqref{2.51} that we have to check whether 
\begin{\eq}   \label{5.38}
\| \delta \, | \La_{\vr} \Bs (\rn) \| = \Big( \sum^\infty_{j=0} 2^{jsq} \sup_{J \in \ganz, M \in \zn} 2^{\frac{J}{p} (n+\vr)q}
\big\| \big( \vp_j \wh{\delta} \big)^\vee | L_p (Q_{J,M} ) \big\|^q \Big)^{1/q}
\end{\eq}
is finite or not (usual modification if $q=\infty$). By \eqref{5.28} one has
\begin{\eq}   \label{5.39}
\sup_{M\in \zn} \big\| \big( \vp_j \wh{\delta} \big)^\vee \, | L_p (Q_{J,M}) \big\| \sim
\begin{cases}
2^{jn(1 - \frac{1}{p})} &\text{if $j \ge J$}, \\
2^{jn} 2^{- \frac{Jn}{p}} & \text{if $j<J$}.
\end{cases}
\end{\eq}
Using $-n<\vr <0$ it follows that
\begin{\eq}   \label{5.40}
\| \delta \, | \La_{\vr} \Bs (\rn) \| \sim \Big( \sum^\infty_{j=0} 2^{jq (s+n+ \frac{\vr}{p})} \Big)^{1/q}
\end{\eq}
(again with the usual modification if $q=\infty$). This proves \eqref{5.37}.
\cm
{\em Step 2.} We prove \eqref{5.36} for $A=B$. This means by \eqref{2.52} that we have to check whether
\begin{\eq}   \label{5.41}
\| \delta \, | \La^{\vr} \Bs (\rn) \| = \sup_{J\in \ganz, M\in \zn} 2^{\frac{J}{p}(n + \vr)} \Big( \sum_{j \ge J^+} 2^{jsq}
\big\| \big( \vp_j \wh{\delta} \big)^\vee | L_p (Q_{J,M}) \big\|^q \Big)^{1/q}
\end{\eq}
is finite or not (usual modification if $q= \infty$). By \eqref{5.39} (with $M=0$) one has
\begin{\eq}    \label{5.42}
\| \delta \, | \La^{\vr} \Bs (\rn) \| \sim \sup_{J \in \ganz} 2^{\frac{J}{p}(n + \vr)} \Big( \sum_{j \ge J^+} 2^{jq(s+n -
\frac{n}{p})} \Big)^{1/q}.
\end{\eq}
Since $|\vr| <n$ it is sufficient to justify \eqref{5.36} for $s < \frac{n}{p} - n$. Using in addition $n+\vr >0$ it follows that
\begin{\eq}    \label{5.43}
\| \delta \, | \La^{\vr} \Bs (\rn) \| \sim \sup_{J\in \nat} 2^{\frac{J}{p}(n + \vr)} 2^{J(s+n- \frac{n}{p})}.
\end{\eq}
This proves \eqref{5.36} for $A=B$. One obtains the corresponding assertion for the spaces $\La^{\vr} \Fs (\rn)$ from \eqref{2.54}. \new{This finally covers \eqref{5.37a} in view of the coincidence \eqref{2.42}.}
\end{proof}

\begin{remark}   \label{R5.19}
The arguments of Step 2 of the above proof apply also to the $0$--clan according to Definition\ref{D2.11}(iii) with the outcome that
\begin{\eq}   \label{5.44}
\delta \in \La^0 \Bs (\rn) \quad \text{with $0<p,q<\infty$ if, and only if, $s \le -n$}.
\end{\eq}
\new{The parallel result holds for $\La^0 \Fs (\rn)$ due to  \eqref{ftbt} and \eqref{2.35}, as then \cite[Example~3.2]{HMS16} provides 
\[
\delta \in \La^0 \Fs (\rn) \quad \text{with $0<p,q<\infty$ if, and only if, $s \le -n$.}
\]
  }
\end{remark}

\begin{remark}  \label{R5.20}
The above theorem and also the comments in Remark~\ref{R5.19} are already covered by \cite[Example 3.2, pp.\,225--227]{HMS16},
appropriately reformulated. This applies also to the proof, relying as above, on the Fourier--analytical definition of these spaces. \new{However, comparing the corresponding diagrams in \cite{HMS16} with the above Figure~\ref{fig-1a} on p.~\pageref{fig-1a} which indicate the parameters $(\frac1p,s)$ such that $\delta\in\rhoAs(\rn)$, the advantage of the new setting is obvious.}
\end{remark}

\begin{remark}   \label{R5.21}
Theorem~\ref{T5.18} compared with Proposition~\ref{P5.16} is again an example of the Slope--$n$--Rule, recall \new{Slope Rules~\ref{slope_rules}(ii)}, where $\frac{n}{p}$ is replaced by $\frac{|\vr|}{p}$. But there are also some differences in limiting situations, this
means $s= \frac{n}{p} -n$ in Proposition~\ref{P5.16} and $s= \frac{|\vr|}{p} -n$, $0<|\vr| <n$ in Theorem~\ref{T5.18}. This applies
also to \eqref{5.44} compared with $B^{-n}_{\infty,q} (\rn)$. Combined with the lifts as described in \eqref{3.25}, \eqref{3.29} it
shows again that the first embedding in \eqref{2.34} is strict. Furthermore if $q<\infty$, then it follows from the above theorem 
combined with the lifts as described in Theorem~\ref{T3.8} that the spaces $\La^{\vr} \Bs (\rn)$ and $\La_{\vr} \Bs (\rn)$, $-n
<\vr <0$, do not coincide. This  had already been used in Step 2 of the proof of Theorem~\ref{T2.16} to justify that the embedding
\eqref{2.56} is strict.
\end{remark}

\subsection{The characteristic function $\chi_Q$}     \label{S5.5}
Expansions in term of Haar wavelets in suitable spaces $\As (\rn)$ attracted a lot of attention up to our time. A description 
including impressive recent results may be found in {\cite[Section 3.5, pp.\,98--103]{T20}}. It is quite natural to ask for 
corresponding expansions in appropriate spaces according to \eqref{2.44}, \eqref{2.45}, the $\vr$--clan, $-n <\vr <0$. These spaces
are not separable: The corresponding proof for the Morrey spaces $\La^{\vr}_p (\rn)$ with $-n<\vr <0$ according to Definition
\ref{D2.3} in \cite[Section 2.3.4, pp.\,23--24]{T14} can be extended to these spaces. Then the related Haar wavelets cannot be a
basis. But, of course, expansions within the dual pairing $\big( S(\rn), S'(\rn) \big)$ make sense, quite similar as the wavelet 
expansions in Section~\ref{S3.1}. First steps in this direction have been done quite recently in \cite{YSY20}, discussing the
question to which spaces the characteristic functions $\chi_Q$ of the unit cube $Q= (0,1)^n$ belongs and in which spaces $\chi_Q$
generates a linear and bounded functional. We deal with these topics  in the context of our preceding considerations.

First we recall to which classical spaces $\As (\rn)$ the characteristic function $\chi_Q$ of $Q =(0,1)^n$, $n\in \nat$, belongs.

\begin{proposition}    \label{P5.22}
Let $\As (\rn)$ be the spaces according to \eqref{2.38}, the $n$--clan. Then
\begin{\eq}   \label{5.45}
\chi_Q \in B^{1/p}_{p,q} (\rn) \qquad \text{if, and only if, $q=\infty$}
\end{\eq}
and
\begin{\eq}   \label{5.46}
\chi_Q \in F^{1/p}_{p,q} (\rn) \qquad \text{if, and only if, $p=\infty$}.
\end{\eq}
Furthermore, $\chi_Q \in \As (\rn)$ if, and only if, either $s<1/p$ or as in \eqref{5.45}, \eqref{5.46}.
\end{proposition}

\begin{remark}   \label{R5.23}
This coincides with \cite[Proposition 2.50, p.\,50]{T20}. The proof consists of two steps. First one deals with $n=1$ and secondly
one reduces higher dimensions to one dimension. This method works also for the other spaces covered by Definition~\ref{D2.11}.
\end{remark}

\begin{theorem}   \label{T5.24}
Let $\La^{\vr} \As (\rn)$ \new{and $\La_{\vr} \As (\rn)$ with $A \in \{B,F \}$} be the spaces according to \eqref{2.44}, \eqref{2.45} of
the $\vr$--clan in Definition~\ref{D2.11}{\em (ii)} with $-n <\vr <0$. Then
\begin{\eq}   \label{5.47}
\chi_Q \in \La^{\vr} \As (\rn)
\end{\eq}
if, and only if,
\begin{\eq}   \label{5.48}
-\infty <s \le \frac{1}{p} \min (|\vr|, 1 ) \qquad \text{with $q=\infty$ if $A=B$ and $s= \frac{1}{p}$}.
\end{\eq}
Furthermore,
\begin{\eq}   \label{5.49}
\chi_Q \in \La_{\vr} \Bs (\rn)
\end{\eq}
if, and only if,
\begin{\eq}   \label{5.50}
-\infty <s \le \frac{1}{p} \min (|\vr|, 1 ) \qquad \text{with $q=\infty$ if $s= \frac{\min(|\vr|,1)}{p}$},
\end{\eq}
and
\new{
  \begin{\eq}   \label{5.49a}
\chi_Q \in \La_{\vr} \Fs (\rn)
\end{\eq}
if, and only if,
\begin{\eq}   \label{5.50a}
-\infty <s \le \frac{1}{p} \min (|\vr|, 1 ).
\end{\eq}}
\end{theorem}

\noindent
One can sketch the parameter areas according to \eqref{5.48}, \eqref{5.50} as follows.\\

\noindent
\begin{minipage}{\textwidth}
~\hspace*{\fill}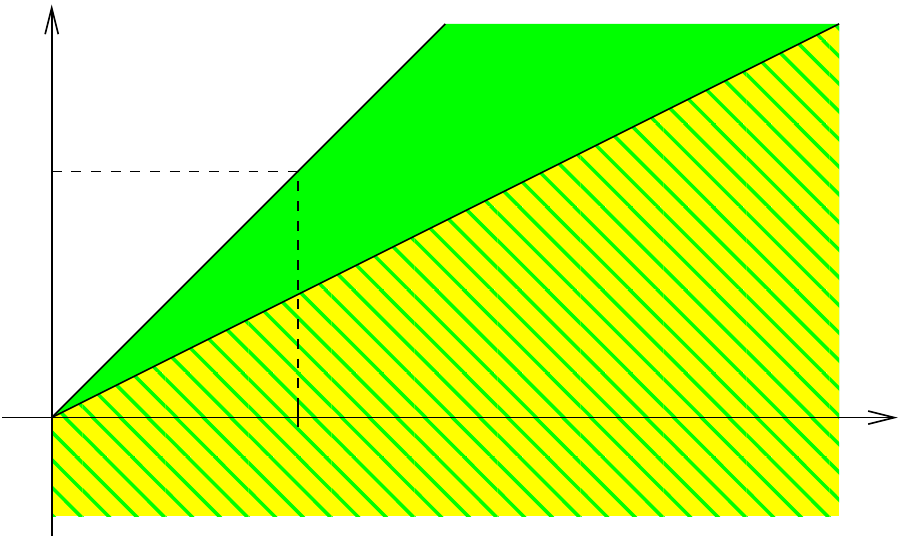\hfill~\\[1ex]    
  ~\hspace*{\fill} $\chi_Q \in \rhoAs(\rn)$ \quad for $0<|\vr|<1$ and $|\vr|\geq 1$ \hfill~\\[1ex]
~\hspace*{\fill}\unterbild{fig-2}
\end{minipage}

\begin{proof}
{\em Step 1.} Let $n=1$ and $-1 <\vr <0$. Let $\chi_I$ be the characteristic function of the unit interval $I = (0,1)$. Then 
$\chi'_I = c (\delta - \delta_1)$, $c\not= 0$, where $\delta_1$ is the shifted $\delta$--distribution with the off--point 1. Now
Theorem~\ref{T3.10} shows that
\begin{\eq}   \label{5.51}
\chi_I \in \La^{\vr} \As (\real) \quad \text{if, and only if,} \quad \delta \in \La^{\vr} A^{s-1}_{p,q} (\real)
\end{\eq}
and
\begin{\eq}   \label{5.52}
\chi_I \in \La_{\vr} \Bs (\real) \quad \text{if, and only if,} \quad \delta \in \La_{\vr} B^{s-1}_{p,q} (\real).
\end{\eq}
Application of Theorem~\ref{T5.18} proves the one--dimensional case of the theorem.
\cm
{\em Step 2.}
Let $2 \le n \in \nat$. We prove \eqref{5.47}, \eqref{5.48} for $A=B$. For this purpose we expand $f= \chi_Q$ according to 
\eqref{3.7}, \eqref{3.8}. Then we are in the same position as in \cite[p.\,51]{T20}, especially
\begin{\eq}   \label{5.53}
\lambda^{j,G}_m (\chi_Q) = \lambda^{j,M}_{m_n} (\chi_I), \quad m= (m_1, \ldots, m_n), \quad G= (F, \ldots, F,M),
\end{\eq}
is the prototype to be considered,
relying on the same reasoning as there. In \eqref{3.10} and \eqref{3.11} with $f = \chi_Q$
we may assume $J \in \no$ instead of $J \in \ganz$. This
follows from $\vr +n >0$. For fixed $m_n$ in \eqref{5.53}
one has now $\sim 2^{(j-J)(n-1)}$ terms in \eqref{5.53} with $m= (m',m_n)$ and
$Q_{j,m} \subset Q_{J,M}$. Then
\begin{\eq}   \label{5.54}
2^{\frac{J}{p} (n+\vr) + j (s - \frac{n}{p})} \Big( \sum_{m': Q_{j,m} \subset Q_{J,M}} |\lambda^{j,G}_m (\chi_Q)|^p \Big)^{1/p} \sim
2^{\frac{J}{p} (1+ \vr) + j(s- \frac{1}{p})} |\lambda^{j,M}_{m_n} (\chi_I) |.
\end{\eq}
If $-n <\vr \le -1$, then $J=0$ is the largest term in \eqref{5.54}. This reduces \eqref{3.10} with $f= \chi_Q$ to the 
one--dimensional case of \eqref{3.9} with $f= \chi_I$. In other words,
\begin{\eq}   \label{5.55}
\| \chi_Q \, | \La^{\vr} \Bs (\rn) \| \sim \| \chi_I \, | \Bs (\real)\|, \qquad -n <\vr \le -1.
\end{\eq}
Then \eqref{5.47} with $A=B$ and \eqref{5.48} with $|\vr| \ge 1$ follows from Proposition~\ref{P5.22}. 
If $-1 <\vr <0$, then one has by \eqref{5.54} and \eqref{3.10} that
\begin{\eq}   \label{5.56}
\| \chi_Q \, | \La^{\vr} \Bs (\rn) \| \sim \| \chi_I \, | \La^{\vr} \Bs (\real) \|, \qquad -1<\vr<0.
\end{\eq}
This reduces \eqref{5.47}, \eqref{5.48} to Step 1.
\cm
{\em Step 3.} The embedding \eqref{2.54} and Step 2 prove \eqref{5.47}, \eqref{5.48} with $A=F$ for all relevant spaces $\La^{\vr}
\Fs (\rn)$ with exception of $|\vr| \ge 1$, $s =\frac{1}{p}$ and $q=\infty$. In this case one relies on the wavelet expansion 
\eqref{3.7}, \eqref{3.8} for $\La^{\vr} F^{1/p}_{p,\infty} (\rn)$ with a related counterpart of the right--hand side of \eqref{3.10}
as described in \cite[Proposition 1.16, p.\,11--12]{T20} based on a reference to \cite[Theorem 3.26, p.\,64]{T14}. Using again 
\eqref{5.53} one obtains
\begin{\eq}   \label{5.57}
\| \chi_Q \, | \La^{\vr} F^{1/p}_{p,\infty} (\rn) \| \sim \| \chi_I \, | F^{1/p}_{p, \infty} (\real)\|, 
\quad 0<p<\infty, \quad -n<\vr \le -1,
\end{\eq}
as the counterpart of \eqref{5.55}. The one--dimensional case of \eqref{5.46} shows that $\chi_Q$ does not belong to $\La^{\vr}
F^{1/p}_{p, \infty} (\rn)$.
\cm
{\em Step 4.} The proof of \eqref{5.49}, \eqref{5.50} follows the same scheme as in Step 2. 
One relies now on \eqref{3.11} and \eqref{5.54}.
Instead of \eqref{5.55}, \eqref{5.56} one has now
\begin{\eq}   \label{5.58}
\| \chi_Q \, | \La_{\vr} \Bs (\rn) \| \sim \|\chi_I \, | \Bs (\real) \|, \qquad -n <\vr \le -1,
\end{\eq}
and
\begin{\eq}   \label{5.59}
\| \chi_Q \, | \La_{\vr} \Bs (\rn) \| \sim \| \chi_I \, | \La_{\vr} \Bs (\real) \|, \qquad -1< \vr <0.
\end{\eq}
Then \eqref{5.49}, \eqref{5.50} follows from \eqref{5.45} and Step 1. \new{Again \eqref{5.49a} with \eqref{5.50a} is a consequence of \eqref{5.47} and \eqref{5.48} (with $A=F$) based on the coincidence \eqref{2.42}.}
\end{proof}

\begin{remark}   \label{R5.25}
  Proposition~\ref{P5.22} shows that the above theorem is a typical example of the Slope--1--Rule, see \new{Slope Rules~\ref{slope_rules}(i)}.
\end{remark}

\begin{remark}    \label{R5.26}
The above arguments can also be applied to the spaces in \eqref{2.41} of the $0$--clan in Definition~\ref{D2.11}(iii) with the outcome
that
\begin{\eq}   \label{5.60}
\chi_Q \in \La^0 \Bs (\rn) \qquad \text{with $0<p,q<\infty$ if, and only if, $s\le 0$}.
\end{\eq}
This follows from the above Step 1 now using \eqref{5.44} and \eqref{5.54} with $\vr =0$, based on Proposition~\ref{P3.1} extended
to $\vr =0$ according to a related comment in Remark~\ref{R3.2}. Compared with \eqref{5.45},
\begin{\eq}   \label{5.61}
\chi_Q \in B^0_{\infty,q} (\rn) \qquad \text{if, and only if, $q=\infty$,}
\end{\eq}
it illuminates again the first strict embedding in \eqref{2.34}.
\end{remark}

The assertion \eqref{5.47}, \eqref{5.48} with $A=B$ and also \eqref{5.60} go back to \cite[Theorem 2.1]{YSY20} reformulated according
to \eqref{2.29}, \eqref{2.30}. The second main topic of this paper addresses the problem in which spaces $\chi_Q$ generates a linear
and bounded functional. This means,
\begin{\eq}   \label{5.61a}
\big| (f, \chi_Q )\big| \le c \, \|f \, | A(\rn) \|, \qquad f \in A(\rn),
\end{\eq}
where $A(\rn)$ is a space covered by Definition~\ref{D2.11}. We deal first with the one--dimensional case.

\begin{proposition}   \label{P5.27}
Let $\La^{\vr} \As (\real)$ and $\La_{\vr} \As (\real)$ with $A \in \{B,F \}$  be the spaces according to \eqref{2.44}, \eqref{2.45}
of the $\vr$--clan $\rhoAs (\real)$
in Definition~\ref{D2.11}{\em (ii)} with $n=1$ and $-1 <\vr <0$. Then
the characteristic function $\chi_I$ of the unit interval $I= (0,1)$ generates a linear and bounded functional in
\begin{\eq}   \label{5.62}
\La^{\vr} \As (\real) \quad \text{if, and only if,} \quad s > \frac{|\vr|}{p} -1,
\end{\eq}
in
\begin{\eq}    \label{5.63}
\La_{\vr} \Bs (\real) \quad \text{if, and only if,} \quad
\begin{cases}
s> \frac{|\vr|}{p} -1, & 0<q \le \infty,\quad\text{or} \\
s= \frac{|\vr|}{p} -1, & 0< q \le 1,
\end{cases}
\end{\eq}
\new{and in
\begin{\eq}   \label{5.63a}
\La_{\vr} \Fs (\real) \quad \text{if, and only if,} \quad s > \frac{|\vr|}{p} -1.
\end{\eq}
}
\end{proposition}

\begin{proof}
Recall that $\chi_I' = c (\delta - \delta_1)$, $c \not= 0$, where $\delta_1$ is the shifted $\delta$--distribution
with $1$ as off--point. Let $\vp$ be a smooth function on $\real$ with compact support near the origin $0$. Then
\begin{\eq}   \label{5.64}
(\chi_I, \vp') = - (\chi_I', \vp) = -c\,\vp (0) = - c \,(\delta, \vp).
\end{\eq}
This can be extended by completion to arbitrary continuous functions with compact support near the origin. Let $I_2 = (-2,2)$. By
Theorem~\ref{T3.10} one has
\begin{\eq}   \label{5.65}
\begin{aligned}
\|f \, | \La^{\vr} A^{s+1}_{p,q} (\real)\| & \sim \| f \, |\La^{\vr} \As (\real) \| + \| f' \, | \La^{\vr} \As (\real) \| \\
&\sim \| f' \, | \La^{\vr} \As (\real) \|, \quad f\in \La^{\vr} A^{s+1}_{p,q} (\real), \quad \supp f \subset I_2,
\end{aligned}
\end{\eq}
where the second equivalence can be obtained in the usual way based on the compact embedding of $\La^{\vr} A^{s+1}_{p,q} (I_2)$ into
$\La^{\vr} \As (I_2)$ being a special case of Theorem~\ref{T5.39} below. Similarly,
\begin{\eq}   \label{5.66}
\| f \, | \La_{\vr} B^{s+1}_{p,q} (\real) \| \sim \| f' \, | \La_{\vr} \Bs (\real) \|, \quad f \in \La_{\vr} B^{s+1}_{p,q} (\real), 
\quad \supp f \subset I_2.
\end{\eq}
Let $f\in \La^{\vr} A^{s+1}_{p,q} (\real)$, $\supp f \subset I_2$, with $s+1 > \frac{|\vr|}{p}$. Then it follows from \eqref{5.64},
\eqref{5.65} and Theorem~\ref{T5.3} that
\begin{\eq}   \label{5.66a}
\big| (\chi_I , f') \big| \le c \, \|f \, | C(\real) \| \le c' \|f\, | \La^{\vr} A^{s+1}_{p,q} (\real) \| \le c'' \| f' \, |
\La^{\vr} \As (\real) \|.
\end{\eq}
If $g$ is a smooth function with $\supp g \subset I_2$, then one has $g= f'$ in $I_2$
for a suitable smooth function with, say, $\supp f \subset
[-4, 4]$. Hence, one can replace $f'$ in \eqref{5.66a} by $g$. This can be extended by a Fatou argument to all $g\in \La^{\vr} \As
(\real)$ with $\supp g \subset I_2$. Then it follows that $\chi_I$ generates a linear and bounded functional in $\La^{\vr} \As
(\real)$ with $s > \frac{|\vr|}{p} -1$. Similarly for the spaces in \eqref{5.63} based again on Theorem~\ref{T5.3}. Conversely, if
$\chi_I$ generates a linear and bounded functional in some space $\La^{\vr} \As (\real)$, then it follows from \eqref{5.64} that
\begin{\eq}   \label{5.66b}
|f(0)| \le c \, \big| (\new{\chi_I}, f') \big| \le c \, \|f' \, | \La^{\vr} \As (\real) \| \le c \, \|f \, | \La^{\vr} A^{s+1}_{p,q}
(\real) \|
\end{\eq}
for compactly supported smooth functions. This requires
\begin{\eq}   \label{5.67}
\La^{\vr} A^{s+1}_{p,q} (\real) \hra C(\real).
\end{\eq}
Similarly for $\La_{\vr} B^{s+1}_{p,q} (\real)$. Then the only--if--parts in \eqref{5.62}, \eqref{5.63} follow again from Theorem
\ref{T5.3}.
\new{Assertion \eqref{5.63a} is covered by \eqref{5.62} with the coincidence \eqref{2.42}.}
\end{proof} 

\begin{remark}   \label{R5.28}
The counterpart of the above proposition  for the classical spaces in \eqref{2.38} with $n=1$ can be obtained if one uses Proposition~\ref{P5.1} with $n=1$ instead of Theorem~\ref{T5.3} with $n=1$. We do not formulate the more or less known outcome. But
it is quite clear that the above proposition is again an example of the Slope--1--Rule,  see \new{Slope Rules~\ref{slope_rules}(i)}.
\end{remark}

It is not clear how the above arguments can be extended from one dimension to higher dimensions. But the question in which spaces
$A(\rn)$ the characteristic function $\chi_Q$ of the unit cube $Q = (0,1)^n$ generates a linear and bounded functional according to
\eqref{5.61a} has been treated in detail in \cite{YSY20} with the following remarkable outcome. Let again
\begin{\eq}   \label{5.68}
\sigma^t_p = t \Big( \max \big( \frac{1}{p}, 1 \big) -1 \Big), \qquad 0<p \le \infty, \quad t \ge 0.
\end{\eq}

\begin{theorem}   \label{T5.29}
Let $n\in \nat$. Let $\La^{\vr} \As (\rn)$ \new{and $\La_{\vr} \As (\rn)$  with $A \in \{B,F \}$} be the spaces according to \eqref{2.44},
\eqref{2.45} of the $\vr$--clan $\rhoAs(\rn)$ in Definition~\ref{D2.11}{\em (ii)} with $-n < \vr <0$. 
Then $\chi_Q$  generates a linear and bounded functional in these spaces if, in addition,
\begin{\eq}   \label{5.69}
s > \frac{1}{p} \, \min \big( |\vr|, 1 \big) -1 + \sigma^{\max (|\vr|,1) -1}_p.
\end{\eq}
It does not generate a linear and bounded functional in these spaces if, in addition,
\begin{\eq}   \label{5.70}
s < \frac{1}{p} \, \min \big( |\vr|, 1 \big) -1 + \sigma^{\max (|\vr|,1) -1}_p.
\end{\eq}
\end{theorem}

\begin{proof}
Both assertions are covered by \cite[Theorems 2.3, 2.5]{YSY20} for the spaces $B^{s, \tau}_{p,q} (\rn) = \La^{\vr} \Bs (\rn)$ 
according to \eqref{2.29}, \eqref{2.30}. The remaining spaces \eqref{2.44}, \eqref{2.45}
can be incorporated by elementary embedding based on the monotonicity of these spaces with respect to $s$, Theorem~\ref{T2.16} \new{and the equivalence \eqref{2.42}}.
\end{proof}

\noindent
We indicate the parameter area given by \eqref{5.69} according for the different cases of $\vr$.\\[1ex]
\noindent
\begin{minipage}{\textwidth}
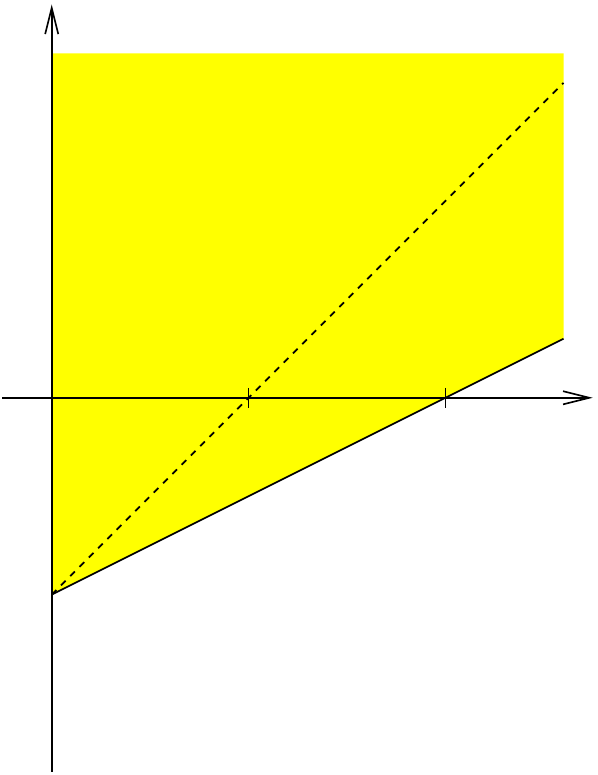\hfill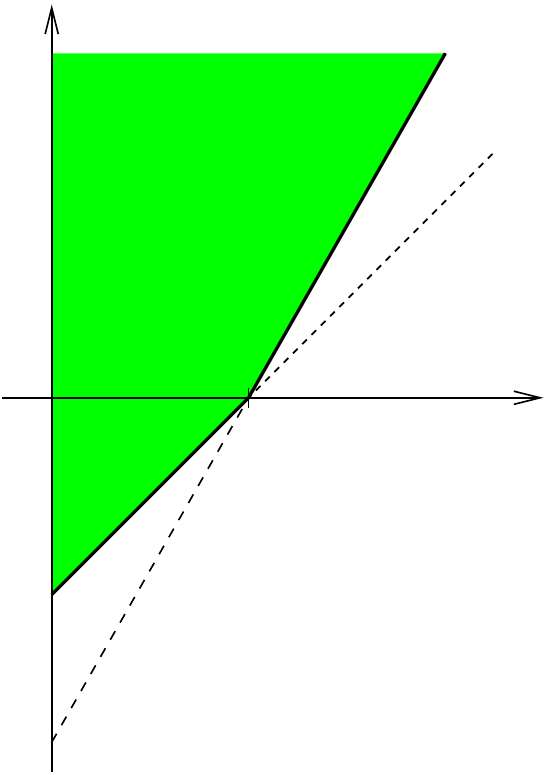\\[2ex]
~\hspace*{\fill} $0<|\vr|<1$ \hfill\hfill $1< |\vr|<n$\hspace*{\fill}~\\[1ex]
~\hspace*{\fill}\unterbild{fig-2'}
\end{minipage}

\begin{remark}   \label{R5.30}
In \cite{YSY20} one finds also a detailed and (almost) final discussion what happens on the breaking line
\begin{\eq}   \label{5.71}
s = \frac{1}{p} \, \min \big( |\vr|, 1 \big) -1 + \sigma^{\max (|\vr|,1) -1}_p
\end{\eq}
for the spaces $\La^{\vr} \Bs (\rn)$. The arguments can also be applied to the spaces in \eqref{2.38} of the $n$--clan in Definition
\ref{D2.11}. Replacing $|\vr|$ in \eqref{5.71} by $n\in \nat$ one obtains the well--known breaking line
\begin{\eq}   \label{5.72}
s = \frac{1}{p} - 1 + \sigma^{n-1}_p = \max \Big( \frac{1}{p} - 1, n \big( \frac{1}{p} -1 \big)\Big)
\end{\eq}
for the classical spaces $\As (\rn)$. For $-1 \le \vr <0$ (as it is always in one dimension)  one has the breaking line $s =
\frac{|\vr|}{p} -1$ for these low--slope spaces, independently of $n\in \nat$. In any case \eqref{5.71} is an example  of the 
Slope--1--Rule, recall \new{Slope Rules~\ref{slope_rules}(i)}, at least in the understanding that $|\vr| =1$ is a breaking point.
\end{remark}

\subsection{Truncation}     \label{S5.6}
A (complex--valued) quasi--Banach space $A(\rn)$ on $\rn$ with
\begin{\eq}   \label{5.73}
S(\rn) \hra A(\rn) \hra S' (\rn), \qquad A(\rn) \subset L^{\loc}_1 (\rn),
\end{\eq}
is said to have the truncation property if $|f| \in A(\rn)$ for any real $f\in A(\rn)$ and if there is a constant $c>0$ such that
\begin{\eq}   \label{5.74}
\big\| \, |f| \, | A(\rn) \big\| \le c \, \|f \, | A(\rn) \|, \qquad f\in A(\rn) \quad \text{real}.
\end{\eq}
This innocent looking but nevertheless rather tricky problem attracted a lot of attention for the classical spaces in \eqref{2.38}, the
$n$--clan according to Definition~\ref{D2.11}. At least for the spaces $\Bs (\rn)$ one has the following satisfactory assertion. Let
$\sigma^t_p$ be again as in \eqref{5.68}.

\begin{proposition}   \label{P5.31}
The spaces
\begin{\eq}    \label{5.75}
\Bs (\rn) \quad \text{with} \quad 0<p,q \le \infty \quad \text{and} \quad s \in \Big( \sigma^n_p, 1 + \frac{1}{p} \Big) 
\end{\eq}
have the truncation property and the spaces
\begin{\eq}    \label{5.76}
\Bs (\rn) \quad \text{with} \quad 0<p,q \le \infty \quad \text{and} \quad s \not\in \Big[ \sigma^n_p, 1 + \frac{1}{p} \Big] 
\end{\eq}
do not have the truncation property.
\end{proposition}

\begin{remark}   \label{R5.31a}
If $\sigma^n_p \ge 1 + \frac{1}{p}$, then \eqref{5.75} is empty and does not apply to any $s$. If $\sigma^n_p > 1 + \frac{1}{p}$, then
the condition in \eqref{5.76} means $s\in \real$. The above result is covered by \cite[Theorem 25.8, p.\,364]{T01} where one finds
a similar assertion for the spaces $\Fs (\rn)$, being less final if $q<1$. 
\end{remark}

One may ask for counterparts for the other spaces covered
by Definition~\ref{D2.11}. According to Theorem~\ref{T5.7} the natural restriction $s> \sigma^{|\vr|}_p$ ensures \eqref{5.73}. 
Furthermore, $1 + \frac{1}{p}$ in \eqref{5.75}, \eqref{5.76} is a typical candidate for the Slope--1--Rule in \new{Slope Rules~\ref{slope_rules}(i)}. This
suggests to ask for the following counterpart of the above proposition for the $B$--spaces in \eqref{2.44}. Let $\sigma^t_p$
be as in \eqref{5.68}.

\begin{conjecture}   \label{C5.32}
The spaces
\begin{\eq}   \label{5.77}
\La_{\vr} \Bs (\rn) \ \text{and} \   \La^{\vr} \Bs (\rn) \ \text{with $0<p<\infty$, $0<q \le \infty$,  $-n < \vr <0$}
\end{\eq}
have the truncation property if
\begin{\eq}   \label{5.78}
s \in \Big( \sigma^{|\vr|}_p, 1 + \frac{1}{p} \min \big( |\vr|,1 \big) \Big)
\end{\eq}
and they do not have the truncation property if
\begin{\eq}   \label{5.79}
s \not\in \Big[ \sigma^{|\vr|}_p, 1 + \frac{1}{p} \min \big( |\vr|,1 \big) \Big].
\end{\eq}
\end{conjecture}

\begin{remark}   \label{R5.33}
As above \eqref{5.78} is empty if $\sigma^{|\vr|}_p =a \ge b = 1+ \frac{1}{p} \min (|\vr|,1)$
and \eqref{5.79} means $s\in \real$ if $a>b$. If one replaces $|\vr|$ by $n$, then the above conditions  coincide with the related
ones in Proposition~\ref{P5.31}. If $0< |\vr| \le 1$, then \eqref{5.78}, \eqref{5.79} is a strip in the $\big( \frac{1}{p},s 
\big)$--diagram similarly as in Proposition~\ref{P5.31} with $n=1$.
First results about these problems for the spaces \eqref{2.24}
with $1\le p <\infty$, $1\le q \le \infty$ have been obtained recently in \cite[Theorem 2]{Hov20a}. There is a 
technical assumption for the spaces with $s \ge 1$ caused by the method but surely not by the topic. Neglecting this additional
condition one obtains by reformulation according to \eqref{2.24} that the spaces $\La_{\vr} \Bs (\rn)$ with $1\le p  <\infty$, $1\le q
\le \infty$ and $-n <\vr <0$ have the truncation property if, and only if,
\begin{\eq}   \label{5.80}
0<s< 1  + \frac{1}{p} \min \big( |\vr|, 1 \big).
\end{\eq}
This fits in the scheme of Conjecture~\ref{C5.32} (under the indicated additional technical assumption).
\end{remark}

\noindent
\begin{minipage}[b]{\textwidth}
~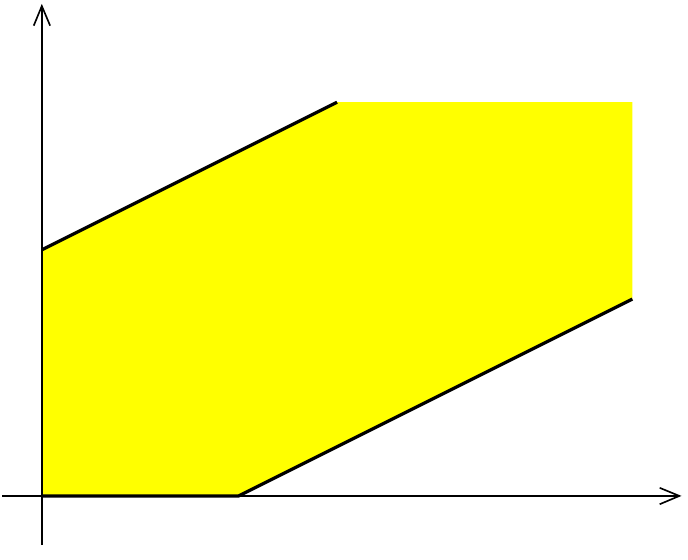\hfill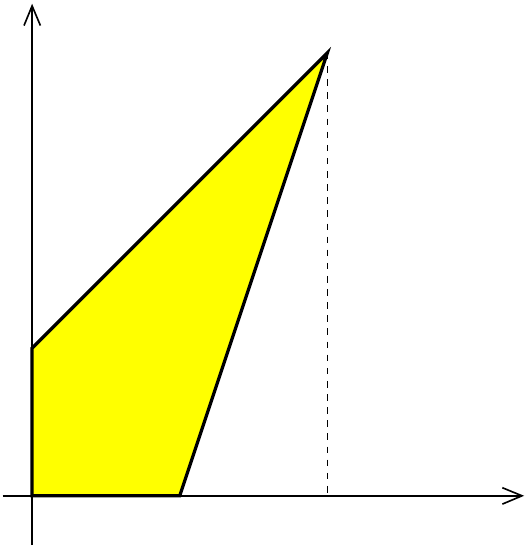\\[1ex]    
~\hspace*{\fill} $0<|\vr|\leq 1$ \hfill\hspace*{\fill}~$1<|\vr|<n$ \hfill~\\[1ex]
~\hspace*{\fill}\unterbild{fig-4}
\end{minipage}

\begin{remark}   \label{R5.34}
As mentioned above there are less final counterparts of \eqref{5.75}, \eqref{5.76} for the spaces $\Fs (\rn)$, especially if $p<1$
or $q<1$. But at least under the restriction $p \ge 1$, $q\ge 1$ one would expect an $F$--version of the above Conjecture~\ref{C5.32}.
This has been confirmed in \cite[Theorem 3]{Hov20a} (again under some additional restrictions) for the spaces on the right--hand side
of \eqref{2.25} and their reformulations in terms of $\La_{\vr} \Fs (\rn)$.
\end{remark}

\subsection{Haar wavelets}   \label{S6.1}
We recalled in Proposition~\ref{P3.1} the characterization of the spaces $\La^{\vr}\Bs (\rn)$ and $\La_{\vr} \Bs (\rn)$ with $n\in 
\nat$, $-n<\vr <0$, $s\in \real$, $0<p<\infty$ and $0<q \le \infty$ in terms of compactly supported Daubechies wavelets. We used these
assertions as tools. But they are also of self--contained interest, including related $F$--spaces and other types of wavelets, above
all Haar wavelets, which attracted a lot of attention. In {\cite[Section 3.5, pp.\,98--103]{T20}} one finds what is known nowadays about
Haar expansions for the classical spaces $\As (\rn)$, including detailed related references. One has now natural restrictions for the
parameters $s,p,q$ under which Haar expansions in $\As (\rn)$ can be expected. This had been used in \cite[Theorem 3.41, p.\,74]{T14}
to prove Haar expansions for the spaces $\La^{\vr} \As (\rn)$ with the same restrictions for $s,p,q$ as for the spaces $\As(\rn)$
complemented by $s<|\vr|/p$. But for fixed $\vr$ with $-n <\vr <0$ one cannot expect that the same conditions for $s,p,q$ as for the
spaces $\As (\rn)$ are still natural for $\La^{\vr} \As (\rn)$. It is just the main aim of \cite{YSY20}, underlying the above 
Theorems~\ref{T5.24}, \ref{T5.29}, to collect basic ingredients indicating how suitable natural substitutes for the spaces $\La^{\vr}
\Bs (\rn)$ may look like. But we are not aware of resulting satisfactory Haar expansions for $\La^{\vr} \Bs (\rn)$ in terms of natural
restrictions for $s,p,q$. Clipping together the above ingredients we formulate the expected outcome as a conjecture. But first we 
collect some basic notation and fix what is known for the spaces $\Bs (\rn)$. We follow {\cite[Section 3.5, pp.\,98--103]{T20}} which in
turn is based on the related references mentioned there.

Let $y\in \real$ and
\begin{\eq}   \label{6.1}
h_M (y) =
\begin{cases}
1 &\text{if $0<y<1/2$},\\
-1 &\text{if $1/2 \le y <1$},\\
0 &\text{if $y \not\in (0,1)$}.
\end{cases}
\end{\eq}
Let $h_F (y) = |h_M (y)| = \chi_I (y)$ be the characteristic function of the unit interval $I =(0,1)$. This is the counterpart \eqref{3.1}, \eqref{3.2}. We now use the same construction as in \eqref{3.3}--\eqref{3.8}. Let again $n\in \nat$, and let
\begin{\eq}   \label{6.2}
G = (G_1, \ldots, G_n) \in G^0 = \{F,M \}^n
\end{\eq}
which means that $G_r$ is either $F$ or $M$. Let
\begin{\eq}   \label{6.3}
G = (G_1, \ldots, G_n) \in G^* = G^j \in \{F,M \}^{n*}, \qquad j \in \nat,
\end{\eq}
which means that $G_r$ is either $F$ or $M$, where $*$ indicates that at least one of the components of $G$ must be an $M$. Let
\begin{\eq}   \label{6.4}
h^j_{G,m} (x) =  \prod^n_{l=1} h_{G_l} \big( 2^j x_l - m_l \big), \quad G\in G^j, \quad m \in \zn, \quad x\in \rn,
\end{\eq}
where (now) $j\in \no$. It is well known that 
\begin{\eq}   \label{6.5}
\big\{ 2^{jn/2} \, h^j_{G,m}: \ j\in \no, \ G\in G^j, \ m\in \zn \big\}
\end{\eq}
is an orthonormal basis in $L_2 (\rn)$ and
\begin{\eq}   \label{6.6}
f = \sum_{j\in \no} \sum_{G\in G^j} \sum_{m\in \zn} \lambda^{j,G}_m \, h^j_{G,m}
\end{\eq}
with 
\begin{\eq}   \label{6.7}
\lambda^{j,G}_m = \lambda^{j,G}_m (f) = 2^{jn} \int_{\rn} f(x) \, h^j_{G,m} (x) \, \di x = 2^{jn} \big(f, h^j_{G,m} \big)
\end{\eq}
is the corresponding expansion. In contrast to \eqref{3.7}--\eqref{3.9} and Proposition~\ref{P3.1} we are interested now in expansions
by Haar wavelets not as a tool but as a self--contained topic. This suggests to offer a more careful formulation as already indicated
in Remark~\ref{R3.2}. In particular, the quasi--Banach space $b^s_{p,q} (\rn)$ with $s\in \real$ and $0<p,q \le \infty$ collects all
sequences
\begin{\eq}   \label{6.8}
\lambda = \big\{ \lambda^{j,G}_m \in \comp: \ j\in \no, \ G\in G^j, \ m\in \zn \big\}
\end{\eq}
such that
\begin{\eq}  \label{6.9}
\|\lambda \, |b^s_{p,q}(\rn) \| =
\Big( \sum^\infty_{j=0} 2^{j(s- \frac{n}{p})q}  \Big( \sum_{m \in \zn, G\in G^j} |\lambda^{j,G}_m|^p \Big)^{q/p} \Big)^{1/q}
\end{\eq}
is finite (with the usual modification if $\max(p,q) =\infty$). 

\begin{proposition}   \label{P6.1}
Let $n\in \nat$,
\begin{\eq}   \label{6.10}
0<p,q \le \infty \quad \text{and} \quad \max \Big( n \big(\frac{1}{p}-1), \ \frac{1}{p} - 1 \Big) <s< \min \big( \frac{1}{p}, 1 \big)
\end{\eq}
as indicated in the Figure~\ref{fig-3} on p.~\pageref{fig-3} below. 
Let $f \in S'(\rn)$. Then $f\in \Bs (\rn)$ if, and only if, it can be represented by
\begin{\eq}   \label{6.11}
f = \sum_{\substack{j\in \no, G\in G^j,\\ m\in \zn}} \lambda^{j,G}_m \, h^j_{G,m}, \qquad \lambda \in b^s_{p,q} (\rn),
\end{\eq}
the unconditional convergence being in $S' (\rn)$. The representation \eqref{6.11} is unique,
\begin{\eq}   \label{6.12}
\lambda^{j,G}_m =  \lambda^{j,G}_m (f) = 2^{jn} \, \big( f, h^j_{G,m} \big)
\end{\eq}
and
\begin{\eq}   \label{6.13}
I: \quad f \mapsto \big\{ \lambda^{j,G}_m (f) \big\}
\end{\eq}
is an isomorphic map of $\Bs (\rn)$ onto $b^s_{p,q} (\rn)$,
\begin{\eq}   \label{6.14}
\big\| f \, | \Bs (\rn) \big\| \sim \big\| \lambda (f) \, | b^s_{p,q} (\rn) \big\|.
\end{\eq}
\end{proposition}

\begin{remark}   \label{R6.2}
This is the $B$--part of {\cite[Theorem 3.18, p.\,99]{T20}}. There one finds also explanations and detailed references. The restriction
\eqref{6.10} is natural: If $\big( \frac{1}{p},s \big)$ does not belong to the set
\begin{\eq}   \label{6.15}
0<p \le \infty \quad \text{and} \quad \max \Big( n \big(\frac{1}{p}-1), \ \frac{1}{p} - 1 \Big) \le s \le \min \big( \frac{1}{p}, 1 \big)
\end{\eq}
in the $\big\{ \big[\frac{1}{p}, s\big): \ 0 \le \frac{1}{p} <\infty, \ s\in \real \big\}$ half--plane in the Figure~\ref{fig-3}, then
the above proposition is no longer valid. In recent times it has been studied in detail what can be said about Haar expansions in
$\Bs (\rn)$ (and also in $\Fs (\rn)$) if $\big( \frac{1}{p},s \big)$ belongs to the difference of the two sets in \eqref{6.15} and
\eqref{6.10} (the related boundary of the set in \eqref{6.15}).
\end{remark}

\noindent
\begin{minipage}{\textwidth}
  ~\hspace*{\fill}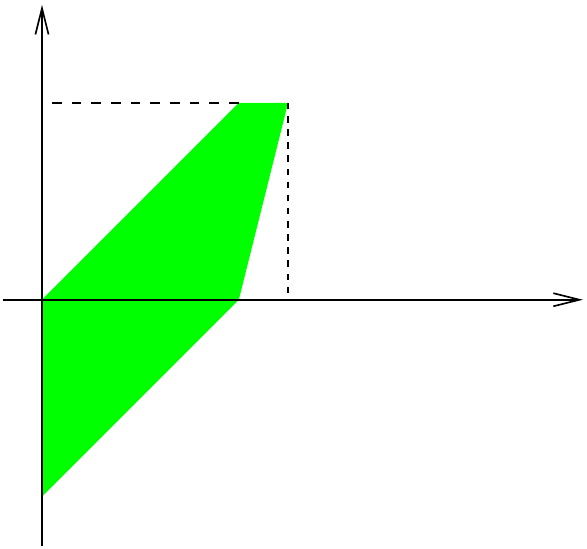\hspace*{\fill}~\\[0ex]    
~\hspace*{\fill}\unterbild{fig-3}
\end{minipage}
\smallskip~

\noindent
One may ask for a counterpart of the above proposition for the $B$--spaces $\rhoBs (\rn)$ in \eqref{2.44} of the $\vr$--clan in
$\rn$ according to Definition~\ref{D2.11}. Let $n\in \nat$, $-n <\vr <0$, $s\in \real$, $0<p<\infty$ and $0<q \le \infty$. Let 
$\lambda$ be the same sequence as in \eqref{6.8}. Then the quasi--Banach space $\La^{\vr} b^s_{p,q} (\rn)$ collects all sequences
$\lambda$ such that
\begin{\eq}    \label{6.16}
\begin{aligned}
&\| \lambda \, | \La^{\vr} b^s_{p,q} (\rn) \| \\
&= \sup_{J\in \ganz, M\in \zn} 2^{\frac{J}{p} (n+ \vr)} \bigg( \sum_{j \ge J^+} 2^{j(s- \frac{n}{p})q} 
\Big( \sum_{\substack{m:Q_{j,m}
\subset Q_{J,M}, \\ G\in G^j}} \big| \lambda^{j,G}_m \big|^p \Big)^{q/p} \bigg)^{1/q}
\end{aligned}
\end{\eq}
is finite and the quasi--Banach space $\La_{\vr} b^s_{p,q} (\rn)$ collects all sequences $\lambda$ such that
\begin{\eq}    \label{6.17}
\begin{aligned}
&\| f \, | \La_{\vr} b^s_{p,q}(\rn) \| \\
&= \bigg( \sum^\infty_{j=0} 2^{j(s- \frac{n}{p})q} \sup_{J\in \ganz, M \in \zn} 2^{\frac{J}{p} (n+\vr)q}
\Big( \sum_{\substack{m:Q_{j,m}
\subset Q_{J,M}, \\ G\in G^j}} \big| \lambda^{j,G}_m \big|^p \Big)^{q/p} \bigg)^{1/q}
\end{aligned}
\end{\eq}
is finite. Here all notation have the same meaning as in Proposition~\ref{P3.1} and as already indicated in Remark~\ref{R3.2}. 
According to \cite[Theorem 3.41, p.\,74]{T14} one has the following counterpart of Proposition~\ref{P6.1} for the spaces $\La^{\vr} 
\Bs (\rn)$. Let $n\in \nat$, $-n <\vr <0$,
\begin{\eq}   \label{6.18}
0<p<\infty, \quad 0<q \le \infty, \quad \max \Big( n \big( \frac{1}{p} -1 \big), \frac{1}{p} - 1 \Big) <s< \min \big(
\frac{1}{p}, \frac{|\vr|}{p}, 1 \big).
\end{\eq}
Let $f\in S'(\rn)$. Then $f \in \La^{\vr} \Bs (\rn)$ if, and only if, it can be represented as
\begin{\eq}   \label{6.19}
f = \sum_{\substack{j\in \no, G\in G^j,\\ m\in \zn}} \lambda^{j,G}_m \, h^j_{G,m}, \qquad \lambda \in \La^{\vr}b^s_{p,q} (\rn),
\end{\eq}
the unconditional convergence being in $S'(\rn)$. The representation \eqref{6.19} is unique with $\lambda^{j,G}_m$ as in \eqref{6.12}
and $I$ in \eqref{6.13} is an isomorphic map of $\La^{\vr} \Bs (\rn)$ onto $\La^{\vr} b^s_{p,q} (\rn)$. Theorem~\ref{T5.24} shows that
the restriction $s< \frac{1}{p} \min (|\vr|,1)$ in \eqref{6.18}, compared with $s <\frac{1}{p}$ in \eqref{6.10}, is natural. The 
restriction $s<1$ in \eqref{6.10} comes from the use of local means with the Haar functions $h^j_{G,m}$ as kernels and their 
restricted cancellation properties. One may consult  the proof of \cite[Proposition 2.8, p.\,79]{T10} based on \cite[Theorem 1.15, 
p.\,7]{T10}. Although we have no reference one can expect that this argument applies also to the spaces $\La^{\vr}\Bs (\rn)$. In other
words, the right--hand sides of \eqref{6.10} and \eqref{6.18} are natural. If one asks for Haar expansions in $\Bs (\rn)$, then the
related conditions for $\chi_Q$ as described in Proposition~\ref{P5.22} and Remark~\ref{R5.30}, especially \eqref{5.72}, are 
necessary. According to Proposition~\ref{P6.1} they are also sufficient for the spaces $\Bs (\rn)$ at least as far the breaking lines
are concerned, in sharp contrast to Haar expansions for $\Fs (\rn)$ where $q$ comes in, \cite[Theorem 3.18, p.\,99]{T20}. One can
take this observation as a guide for $B$--spaces. Then it is clear that the left--hand side of \eqref{6.18} as natural restriction for
Haar expansions in $\La^{\vr} \Bs (\rn)$, $-n<\vr <0$, is questionable. Based on the Theorems~\ref{T5.24} and \ref{T5.29} the 
following expectation is at least reasonable. Let
\begin{\eq}   \label{6.20}
R_\sigma = \Big\{ \big( \frac{1}{p},s \big): \ 0<p<\infty, \ \frac{\sigma}{p} -1 <s< \min \big( 1, \frac{\sigma}{p} \big) \Big\},
\quad 0<\sigma \le 1,
\end{\eq}
\begin{\eq}   \label{6.21}
R_\sigma = \Big\{ \big(\frac{1}{p},s \big): \ 0<p<\infty, \ \max \Big( \frac{1}{p}-1, \sigma \big( \frac{1}{p} -1 \big) \Big) <s<
\min \big(1, \frac{1}{p} \big) \Big\}, \ \sigma >1,
\end{\eq}
and 
\begin{\eq}   \label{6.22}
\ol{R_\sigma} = \Big\{ \big( \frac{1}{p},s \big): \ 0<p<\infty, \ \frac{\sigma}{p} -1 \le s \le \min \big( 1, \frac{\sigma}{p} \big) \Big\}, \quad 0<\sigma \le 1,
\end{\eq}
\begin{\eq}   \label{6.23}
\ol{R_\sigma} = \Big\{ \big(\frac{1}{p},s \big): \ 0<p<\infty, \ \max \Big( \frac{1}{p}-1, \sigma \big( \frac{1}{p} -1 \big) \Big) 
\le s \le \min \big(1, \frac{1}{p} \big) \Big\}, \ \sigma >1,
\end{\eq}
in the $\big\{ \big[ \frac{1}{p}, s \big): \ 0 \le \frac{1}{p} <\infty, \ s\in \real \big\}$ half--plane.

\noindent
\begin{minipage}[b]{\textwidth}
~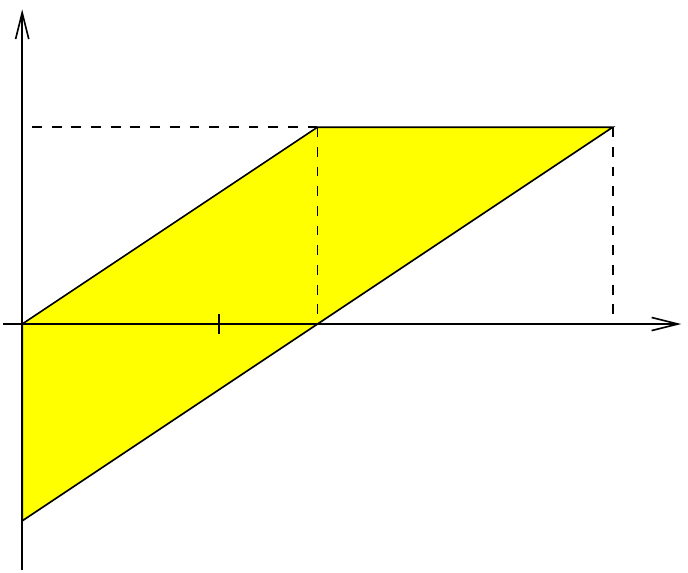\hfill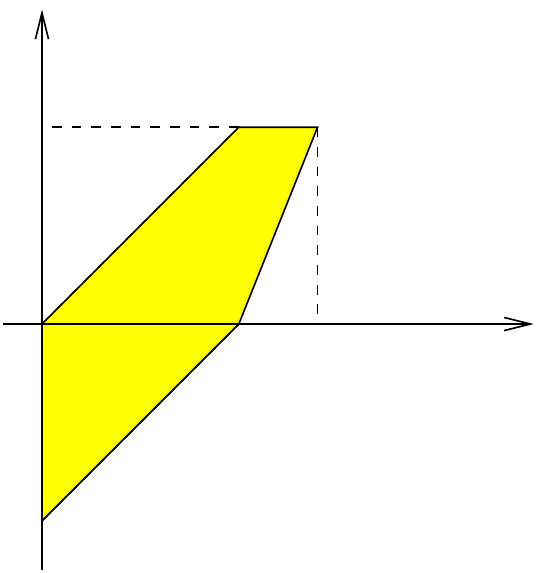\\[1ex]    
~\hspace*{\fill} $R_{|\vr|}: 0<|\vr|\leq 1$ \hfill\hspace*{\fill}~~~$R_{|\vr|}: 1<|\vr|<n$ \qquad~\\[2ex]
~\hspace*{\fill}\unterbild{fig-5}
\end{minipage}

\begin{conjecture}   \label{C6.3}
Let $n\in \nat$, $-n <\vr <0$, $s\in \real$, $0<p<\infty$ and $0<q\le \infty$. Let $ \rhoBs (\rn)$ be the $B$--spaces $\La^{\vr}
\Bs (\rn)$ or $\La_{\vr} \Bs (\rn)$ of the $\vr$--clan according to Definition~\ref{D2.11}{\em (ii)} and let $ \rhobs(\rn)$ be the respective sequence spaces $\La^\vr \bs(\rn)$ or $\La_{\vr} \bs(\rn)$ as introduced above. Let $f\in 
S'(\rn)$ and $\big( \frac{1}{p},s\big) \in R_{|\vr|}$. Then $f\in  \rhoBs (\rn)$ if, and only if, it can be represented as
\begin{\eq}   \label{6.24}
f = \sum_{\substack{j\in \no, G\in G^j,\\ m\in \zn}} \lambda^{j,G}_m \, h^j_{G,m}, \qquad \lambda \in \rhobs(\rn),
\end{\eq}
the unconditional convergence being in $S'(\rn)$. The representation in \eqref{6.24} is unique with $\lambda^{j,G}_m$ as in 
\eqref{6.12} and $I$ in \eqref{6.13} is an isomorphic map of $ \rhoBs (\rn)$ onto $\rhobs(\rn)$. The spaces
$ \rhoBs (\rn)$ with $\big( \frac{1}{p},s \big) \not\in \ol{R_{|\vr|}}$ do not admit characterisations in terms of Haar wavelets.
\end{conjecture}

\begin{remark}   \label{R6.4}
This is the $\vr$--counterpart of the corresponding assertion for the spaces $\Bs (\rn)$ as described in Proposition~\ref{P6.1} and
Remark~\ref{R6.2}.

\noindent
\begin{minipage}{0.45\textwidth}
  ~\hspace*{\fill}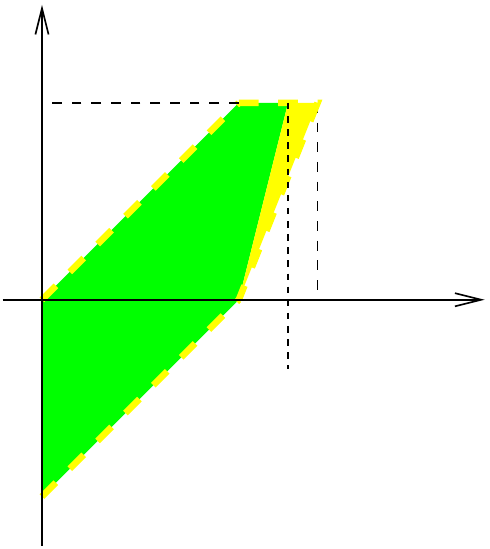\hspace*{\fill}~\\[0ex]    
~\hspace*{\fill}$R_n$ and $R_{|\vr|}$ for $1<|\vr|<n$ \hspace*{\fill}~\\[1ex]
\unterbild{fig-6}
\end{minipage}\hfill
\begin{minipage}{0.53\textwidth}
Taking for granted that $s\le 1$ is always necessary for Haar expansions, then the above--quoted assertion, based 
on \eqref{6.18}, and the preceding considerations confirm the Conjecture~\ref{C6.3} for the spaces
\begin{align}
\nonumber
 \La^{\vr} \Bs (\rn), &  \\
 n \ge 2, \quad & -n<\vr \le -1, \label{6.25}\\
 \nonumber  0<q \le \infty & \quad  \text{if} \quad 1\le p < \infty,
\end{align}

whereas $R_n$ and $R_{|\vr|}$ differ if $p<1$, see Figure~\ref{fig-6} aside.\\
\end{minipage}
\medskip~

\noindent
The situation for the low--slope spaces $\La^{\vr} \Bs (\rn)$, $n\in \nat$,
$-1 <\vr <0$ is different and related Haar expansions are ensured so far only in the region
\begin{\eq}   \label{6.26}
0<p<\infty, \ 0<q\le \infty \quad \text{and} \quad \max \Big(n \big( \frac{1}{p}-1\big), \frac{1}{p} -1 \Big) <s< \min \big(
\frac{|\vr|}{p}, 1 \big)
\end{\eq}
instead of $R_{|\vr|}$ as conjectured (depending on $n$ and $|\vr|$ and not only on $|\vr|$ exclusively).
\end{remark}

\begin{remark}    \label{R6.5}
It is sufficient to confirm the above conjecture for the spaces $\La^{\vr}\Bs (\rn)$. Using \eqref{2.55} and \eqref{3.14} it follows from
\begin{\eq}   \label{6.27}
\La_{\vr} \Bs (\rn) = \big( \La^{\vr} B^{s_1}_{p,p} (\rn), \La^{\vr} B^{s_2}_{p,p} (\rn) \big)_{\theta,q}
\end{\eq}
and its sequence counterpart based on the indicated wavelet isomorphisms, that any affirmative assertion about Haar expansions for the
spaces $\La^{\vr} \Bs (\rn)$ can be transferred to the spaces $\La_{\vr} \Bs (\rn)$. The sharpness of the breaking lines in the above
conjecture for the spaces $\La_{\vr} \Bs (\rn)$ follows afterwards from \eqref{3.13} with $q=\infty$ and \eqref{2.57}.
\end{remark}

\subsection{A proposal: Faber expansions}   \label{S6.2}
Expansions by Haar systems is a central topic in the theory of function spaces. It is of interest for its own sake. But they may also 
serve as a starting point for numerous other observations. In one dimension one can step from Haar systems to related Faber systems
on the unit interval $I = (0,1)$ or on $\real$, roughly speaking, by integration. This includes Faber expansions in suitable function
spaces on $\real$, sampling and so--called numerical integration. We refer to \cite{T10} and, in particular, to \cite[Section 
3.2]{T12}. One may ask what happens if one replaces $\Bs (\real)$ with
\begin{\eq}   \label{6.28}
0<p,q \le \infty, \qquad \frac{1}{p} <s< 1 + \min \big( \frac{1}{p}, 1 \big)
\end{\eq}
by, say, $\La^{\vr} \Bs (\real)$, $-1<\vr <0$. Taking for granted that one has related affirmative answers for the above Conjecture~\ref{C6.3}, then
\begin{\eq}   \label{6.29}
0<p<\infty, \quad 0<q\le \infty, \quad \frac{|\vr|}{p} <s< 1 + \min \big(1, \frac{|\vr|}{p} \big)
\end{\eq}
is the counterpart of \eqref{6.28}, as indicated in Figure~\ref{fig-6a} below.\\

\noindent
\begin{minipage}{\textwidth}
  ~\hspace*{\fill}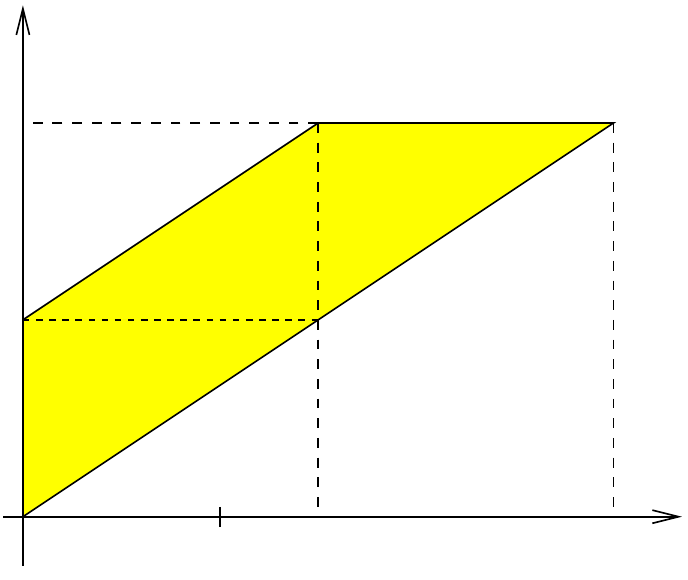\hspace*{\fill}~\\[1ex]    
~\hspace*{\fill} $(0,1) + R_{|\vr|}$ for $0<|\vr|\leq 1$\hspace*{\fill}~\\[0ex]
~\hspace*{\fill}\unterbild{fig-6a}
\end{minipage}
\smallskip~

\noindent
Faber systems can be extended from one dimension to higher dimensions by tensor products in the
context of spaces with dominating mixed smoothness
\begin{\eq}   \label{6.30}
S^r_{p,q} B (\rn) \quad \text{with} \quad 0<p,q \le \infty \quad \text{and} \quad r \in \real.
\end{\eq}
This, in turn, can be used for numerical integration and discrepancy. As for basic definitions and related properties  one may consult
\cite{ST87}, \cite{T10}, \cite{T12}, \cite{T19} and the references within. One may ask for Morrey versions 
\begin{\eq}   \label{6.31}
 \vr\text{-} \! S^r_{p,q} B(\rn), \quad -1<\vr<0, \ 0<p<\infty, \ 0<q \le \infty \ \text{and} \ r\in \real,
\end{\eq}
of these spaces and whether they are of some use in connection with sampling, discrepancy and numerical integration. We are not aware of any efforts in these directions so far.

\section{Embeddings -- revisited}   \label{S7}

In Sections~\ref{S5.1}, \ref{S5.2} we already studied embeddings of spaces $\rhoAs(\rn)$ into certain target spaces, in addition to the more basic embeddings result which can be found in Section~\ref{S2}, in particular in Proposition~\ref{P2.9} and Theorem~\ref{T2.16}. It is our intention now to return to the topic and modify it appropriately. At first we pose the question of continuous embeddings within the same $\vr$-clan to spaces, defined on bounded domains. As is well known, then we may also ask for compact embeddings -- unlike in case of spaces on $\rn$. This will be our first goal. Secondly we shed some light on the so-called growth envelope functions, some tool to `measure' unboundedness of tempered distributions by very classical tools. It turns out, that again the Slope Rules~\ref{slope_rules} can be observed. Finally, in a last part, we leave the realm of embeddings {\em within one and the same $\vr$-clan} and briefly discuss some phenomena when `crossing borders' (between $\vr$-clans).

\subsection{Spaces on domains}    \label{S5.7}
Similarly as in Definition~\ref{D4.7} one introduces spaces on domains $\Om$ in $\rn$ by restriction. As usual, $D'(\Om)$ denotes the
set of all distributions on \Om. Furthermore, $g|\Om \in D'(\Om)$ stands for the restriction of $g\in S'(\rn)$ to \Om.

\begin{definition}   \label{D5.35}
Let $\Om$ be an {\em (}arbitrary{\em )} domain {\em (}= open set{\em )} on \rn, $n\in \nat$. Let $A(\rn)$ be a space covered by
Definition~\ref{D2.11}. Then
\begin{\eq}  \label{5.81}
A(\Om) = \big\{ f\in D'(\Om): \ \text{$f= g|\Om$ for some $g\in A(\rn)$} \big\},
\end{\eq}
\begin{\eq}    \label{5.82}
\| f \, |A(\Om) \| = \inf \| g \, | A(\rn) \|
\end{\eq}
where the infimum is taken over all $g \in A(\rn)$ with $g|\Om = f$.
\end{definition}

\begin{remark}   \label{R5.36}
It follows from standard arguments that $A(\Om)$ is a quasi--Banach space (and a Banach space if $p\ge 1$, $q\ge 1$), continuously 
embedded in $D'(\Om)$ and in the restriction of $S'(\rn)$ to \Om.
\end{remark}

The classical spaces $\As (\Om)$ are the restrictions  of the spaces in \eqref{2.38}, the $n$--clan, now with respect to \Om. They
have been studied in numerous papers and books, including {\cite[Section 2.6, pp.\,62--74]{T20}} where one finds some recent results
(covering also $F^s_{\infty,q} (\Om)$) and (historical) references. One may ask for counterparts of related assertions for the spaces
\begin{\eq}   \label{5.83}
\La_{\vr} \Bs(\Om), \quad \La^{\vr} \Bs (\Om) \quad \text{and} \quad \La_{\vr} \Fs (\Om) = \La^{\vr} \Fs (\Om)
\end{\eq}
with
\begin{\eq}   \label{5.84} 
s\in \real, \quad   0<p<\infty, \quad 0<q \le \infty,
\end{\eq}
and $-n <\vr <0$, the $\vr$--clan according to Definition~\ref{D2.11}(ii) now with respect to \Om. But this is not our topic. We are
interested here exclusively in the description of compact embeddings between these spaces measured in terms of entropy numbers.

Let $A$ and $B$ be two quasi--Banach spaces and let $U_A = \{a \in A: \ \|a \, | A \| \le 1 \}$ and $U_B = \{b \in B: \ \|b \, | B \| \le 1 \}$ be the related unit balls. Let $T: A \hra B$ be a linear and compact map of $A$ into $B$. Then the $k$th entropy number
$e_k (T)$, $k\in \nat$, is the infimum of all $\ve >0$ such that
\begin{\eq}  \label{5.85}
T(U_A) \subset \bigcup^{2^{k-1}}_{j=1} (b_j + \ve U_B ) \quad \text{for} \quad \{b_l \}^{2^{k-1}}_{l=1} \subset B.
\end{\eq}
Basic properties and references may be found \new{in the monographs \cite{CS,EE,Konig,Pie-s} (restricted to the case of Banach spaces), and \cite{ET96} for some extensions to quasi-Banach spaces}, see also \cite[Section 1.10, pp.\,55--58]{T06}.

\begin{proposition}   \label{P5.37}
Let $\Om$ be a bounded domain in \rn, $n\in \nat$. Let $0<p_1, p_2, q \le \infty$ and $s_1 \in \real$, $s_2 \in \real$.
\bli
\item
Then there is a continuous embedding
\begin{\eq}   \label{5.86}
\id_{\Om}: \quad B^{s_1}_{p_1,q} (\Om) \hra B^{s_2}_{p_2,q} (\Om)
\end{\eq}
if, and only if,
\begin{\eq}   \label{5.87}
s_1 - s_2 \ge n \cdot \max \Big( \frac{1}{p_1} - \frac{1}{p_2}, 0 \Big).
\end{\eq}
\item
The embedding $\id_{\Om}$ is compact if, and only if,
\begin{\eq}   \label{5.88}
s_1 - s_2 > n \cdot \max \Big( \frac{1}{p_1} - \frac{1}{p_2}, 0 \Big).
\end{\eq}
Furthermore,
\begin{\eq}   \label{5.89}
e_k (\id_{\Om}) \sim k^{- \frac{s_1 - s_2}{n}}, \qquad k\in \nat.
\end{\eq}
\eli
\end{proposition} 

\begin{remark}   \label{R5.38}
Detailed discussions  and references about embeddings of the spaces $\As (\Om)$ may be found in {\cite[Section 2.6.5, 
pp.\,69--72]{T20}}. The equivalence \eqref{5.89} for bounded $C^\infty$ domains $\Om$ goes back to \cite[Section 3.3, 
pp.\,105--118]{ET96} and the references given there. Its extension to arbitrary bounded domains $\Om$ in $\rn$ may be found in
\cite[Theorem 1.97, Remark 1.98, p.\,61]{T06} and related references. It is based on the observation that the quality of the bounded
domain $\Om$ does not play any r\^ole. Everything can be reduced to the embedding
\begin{\eq} \label{5.90}
\big\{ f \in B^{s_1}_{p_1,q} (\rn) : \ \supp f \subset [0,1]^n \big\} \hra B^{s_2}_{p_2,q} (\rn).
\end{\eq}
This applies also to the counterparts of the above assertions for the spaces in \eqref{5.83}, \eqref{5.84} as discussed below. 
Nowadays everything can be reduced to wavelet representations for these spaces and mapping properties between related sequence
spaces. In this way one can also justify immediately the only--if--parts of the above proposition. They are more or less obvious but
not explicitly mentioned in the quoted literature. 
\end{remark}

The above proposition can be extended from $\Bs (\Om)$ to the spaces $\La_{\vr} \Bs (\Om)$ and $\La^{\vr} \Bs (\Om)$ with $-n <\vr
<0$ in \eqref{5.83}, \eqref{5.84} as follows.

\new{
\begin{theorem}  \label{T5.39}
Let $\Om$ be a bounded domain {\em (=} open set{\em )} in \rn, $n\in \nat$. Let $-n < \vr <0$, $0<q \le \infty$,
\begin{\eq}   \label{5.91}
0<p_1, p_2 <\infty \quad \text{and} \quad s_1 \in \real, \quad s_2 \in \real.
\end{\eq}
\bli
\item
Then there is a continuous embedding
\begin{\eq}   \label{5.92}
\id_{\Om}: \quad  \La_{\vr} A^{s_1}_{p_1,q} (\Om) \hra \La_{\vr} A^{s_2}_{p_2,q} (\Om) 
\end{\eq}
if, and only if,
\begin{\eq}   \label{5.93}
s_1 - s_2 \ge  |\vr| \cdot\max \Big( \frac{1}{p_1} - \frac{1}{p_2}, 0 \Big).
\end{\eq}
The same is true for the embedding
\begin{\eq}   \label{DDH-1}
\id_{\Om}: \quad  \La^{\vr} F^{s_1}_{p_1,q} (\Om) \hra \La^{\vr} F^{s_2}_{p_2,q} (\Om). 
\end{\eq}
\item The embedding
  \begin{equation}
 \new{   \id_{\Om} : \rhoAx{s_1}{p_1}{q}(\Omega) \hookrightarrow \rhoAx{s_2}{p_2}{q}(\Omega)}
 \end{equation}
    is compact if, and only if,
\begin{\eq}   \label{5.94}
s_1 - s_2 > |\vr| \cdot \max \Big( \frac{1}{p_1} - \frac{1}{p_2}, 0 \Big).
\end{\eq}
Furthermore,
\begin{\eq}   \label{5.95}
e_k (\id_{\Om} ) \sim k^{- \frac{s_1 - s_2}{n}}, \qquad k\in \nat.
\end{\eq}
\eli
\end{theorem}

\begin{proof}
The assertions about the continuity and the compactness of the embedding 
$\id_{\Om}$ in \eqref{5.92} in case of $\rhoAs(\Omega)=\La_{\vr}\As(\Omega)$
are covered by \cite[Theorems 3.1, 4.1, pp.\,126, 135]{HaS13} (when $A=B$) and \cite[Theorem~5.2]{HaS14} (when $A=F$), always reformulated according to \eqref{2.24}, \eqref{2.25}. The asymptotic \eqref{5.95} for spaces  $\rhoAs(\Omega)=\La_{\vr}\As(\Omega)$ can be found in  \cite[Case (a), Theorem 4.2, pp.\,3, 19--20]{HaS20}. The compactness and entropy number result for $\rhoAs(\Omega)=\La^{\vr}\As(\Omega)$ can be found in \cite[Theorems 3.2, 4.1]{GHS21}. Finally, the characterisation for the continuity of $\id_\Omega$ in case of $\La^{\vr}\Fs(\Omega)$ is a special case of \cite[Corollary~4.14]{GHS21b}.
\end{proof}
}

\begin{remark}   \label{R5.40}
 \new{As far as the $B$-counterpart of \eqref{DDH-1} is concerned, i.e., 
\begin{\eq}  
\id_{\Om}: \quad  \La^{\vr} B^{s_1}_{p_1,q} (\Om) \hra \La^{\vr} B^{s_2}_{p_2,q} (\Om), 
\end{\eq}
we have no final answer yet. More precisely, there is a gap concerning the limiting situation of \eqref{5.93} when $p_1<p_2$, that is,
\begin{\eq}   \label{DDH-3}
s_1-s_2 = |\vr| \Big( \frac{1}{p_1} - \frac{1}{p_2} \Big) > 0.
\end{\eq}
Then, dealing with different parameters $q_1$ and $q_2$, we could prove in \cite[Theorem~4.9]{GHS21b} that
\begin{\eq}   \label{DDH-2}
\id_{\Om}:  \La^{\vr} \Be (\Om) \hra \La^{\vr} \Bz (\Om) 
\end{\eq}
is continuous in that case, if $q_1 \leq \frac{p_1}{p_2}q_2$ which excludes $q_1=q_2$. Conversely, the continuity of $\id_\Omega$ given by \eqref{DDH-2} in the limiting situation \eqref{DDH-3} implies $q_1\leq q_2$, as desired. So apart from this small gap in this particular limiting case \eqref{DDH-3} for spaces of type $\La^{\vr}\Bs(\Omega)$, there is good reason to assume that in all settings according to Definition~\ref{D2.11}, 
\[
\id_{\Om} : \rhoAx{s_1}{p_1}{q}(\Omega) \hookrightarrow \rhoAx{s_2}{p_2}{q}(\Omega)
\]
is continuous  if, and only if, \eqref{5.93} is satisfied. }
  \ignore{
By elementary embeddings based on \eqref{2.44} and their restrictions to $\Om$ one can extend \eqref{5.95} to
\begin{\eq}   \label{5.97}
\id_{\Om}: \quad \La^{\vr} A^{s_1}_{p_1,q_1} (\Om) \hra \La^{\vr} A^{s_2}_{p_2,q_2} (\Om)
\end{\eq}
and 
\begin{\eq}   \label{5.98}
\id_{\Om}: \quad \La_{\vr} B^{s_1}_{p_1,q_1} (\Om) \hra \La_{\vr} B^{s_2}_{p_2,q_2} (\Om)
\end{\eq}
with $-n < \vr <0$, \eqref{5.91}, \eqref{5.94}
and $0<q_1, q_2 \le \infty$.}
  
  The above theorem \new{(together with our preceding remarks)} compared with Proposition~\ref{P5.37}
shows that continuity and compactness of $\id_{\Om}$
within a $\vr$--clan obeys the Slope--$n$--Rule, recall \new{Slope Rules~\ref{slope_rules}(ii)}.
But it is very remarkable  that \eqref{5.95} = \eqref{5.89} is the same for all $\vr$--clans, $-n <\vr <0$ and also the $n$--clan.
 The assertions in \new{the papers
\cite{HaS13,HaS20,GHS21,GHS21b}} are proved in the context of the spaces on the right--hand side of \eqref{2.24}, covering also continuous
and compact embeddings between spaces belonging to different $\vr$--clans. There are also some estimates of the
entropy numbers of corresponding compact embeddings, as well as their approximation numbers, cf. \cite{HaS19,GHS21}. They are \new{in some sense remarkable, but} less final and \new{more technical} 
than in the above cases within a fixed
$\vr$--clan. Crossing the border between different clans is also in the theory of function spaces a brave undertaking, we return to this topic in Section~\ref{S6.3} below.
\end{remark}

\noindent
\begin{minipage}{\textwidth}
  ~\hspace*{\fill}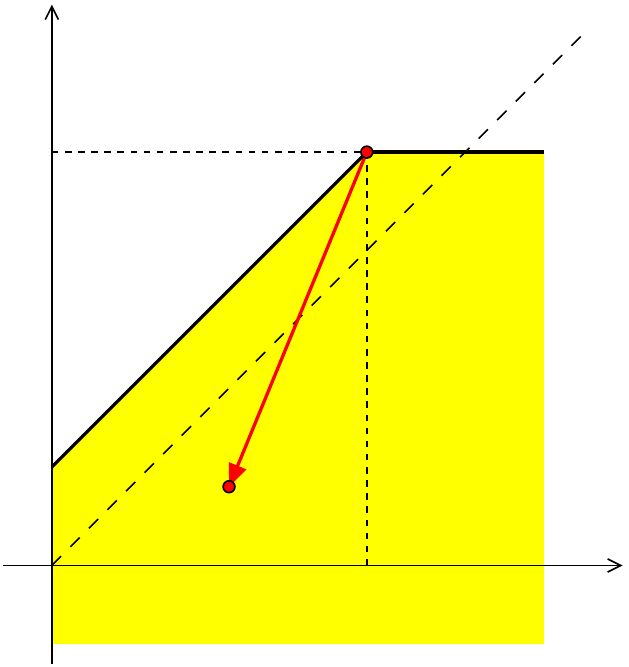\hspace*{\fill}~\\[0ex]    
~\hspace*{\fill}\unterbild{fig-8}
\end{minipage}

\subsection{Growth envelope functions}
\new{
As mentioned in the beginning, we shall prove some specific unboundedness feature of function spaces of Morrey type. This will be done in the context of growth envelope functions. First we briefly recall this concept, before we come to our main results.

If $f$ is an extended complex-valued measurable function on
$\rn$ which is finite a.e., then the decreasing
rearrangement of $f$ is the function defined on $[0,\infty)$ by
\begin{equation}  \label{f*}
f^{*}(t):= \inf \{\sigma > 0: \mu (f,\sigma)\leq t\},\quad t \geq
0,
\end{equation}
with $\mu (f,\sigma)$ being the distribution function given by
$$
\mu (f,\sigma):= \big|\{x\in\rn:|f(x)|>\sigma\}\big|,
\quad \sigma \geq 0.
$$
As usual, the convention $\inf \emptyset =\infty$ is assumed. 

\begin{definition}
\label{defi_eg}
Let $X=X(\rn)$ be some quasi-normed function space on $\rn$. 
The {\em growth envelope function} $\egX : (0,\infty) \rightarrow [0,\infty]$ of $\ X\ $
is defined by
\[
\egX(t) := \sup_{\|f\mid X\|\leq
1} f^*(t), \qquad t\in (0,\infty).
\]
\end{definition}

This concept was first introduced and  studied in \cite[Chapter~2]{T01} and
\cite{HaHabil}; see also \cite{Ha-crc}. In a light abuse of notation we shall not distinguish between representative and equivalence class (of growth envelope functions) and stick at the notation introduced above. We refer to the above references for further details and only recall that  $X(\rn)\hookrightarrow L_{\infty}(\rn)$ holds if, and only if, $\egX$ is bounded. This is the case we are not further interested in. Let us also mention that in rearrangement-invariant Banach function spaces $X(\rn)$ there exists the notion of the fundamental function $\fundfn{X}$, that is, for $t>0$, and $ A_t
\subset \rn $ with $|A_t| = t$, then  $\ 
\fundfn{X}(t) = \big\| \chi_{A_t} \big| X\big\| $. In
\cite[Sect.~2.3]{Ha-crc} it is proved that in this case
\[\egX(t) \sim  \frac{1}{\fundfn{X}(t)}, \quad t>0.\]

\begin{example}\label{exm-Lpq}
Basic examples for spaces $X$ are the well-known Lorentz spaces $L_{p,q}(\rn)$; for
  definitions and further details we refer to \cite[Ch.~4, 
  Defs.~4.1]{BS}, for instance. This scale represents a natural refinement of the scale of Lebesgue spaces. The following result is proved in \cite[Sect. 2.2]{Ha-crc}. Let $0<p<\infty$, $0<q\leq \infty$. Then
\begin{equation}
\egv{L_{p,q}(\rn)}(t) \;\sim\; t^{-\frac1p} \; ,\qquad t >0.
\label{Lp-global}
\end{equation}
\label{Prop-Lpq}
\end{example}

  In general, if we consider spaces of distributions, the concept of growth envelopes makes sense only for spaces $ X(\rn) \subset \Lloc(\rn) $, i.e., 
when we deal with locally integrable functions. For such spaces  we shall thus assume $X(\rn) \subset \Lloc(\rn)$ and, moreover, $X(\rn)
\not\hookrightarrow L_\infty(\rn)$, as we are interested in questions of unboundedness.

\begin{example}\label{envg-A}
Let $X=\As(\rn)$ and recall Propositions~\ref{P5.1} and \ref{P5.5}. Then the results for growth envelope functions in spaces $\As(\rn)$ for $0<q\leq \infty$, $0<p < \infty$,  and $s> \sigma_p^n$ read as follows.
\begin{itemize}
\item[{\bfseries (i)}]
Assume $s<\frac{n}{p}$. Then
\begin{equation}\label{eg-A}
\egv{\As(\rn)}(t) \ \sim \  t^{- \frac1p + \frac{s}{n}}\ ,\quad t\to 0. 
\end{equation}
\item[{\bfseries (ii)}]
Let $ 1<q\leq \infty  $ and $  \frac1q + \frac{1}{q'} =1$, as usual. Then
\begin{equation}
\egv{B^{n/p}_{p,q}(\rn)}(t) \ \sim \ \left|\log t\right|^\frac{1}{q'}, \quad t\to 0.
\end{equation}
\item[{\bfseries (iii)}]
Let $ 1<p< \infty  $ and $  \frac1p + \frac{1}{p'} =1$, as usual. Then
\begin{equation}
\egv{F^{n/p}_{p,q}(\rn)}(t) \ \sim \ \left|\log t\right|^\frac{1}{p'}, \quad t\to 0.
\end{equation}
\end{itemize}
For proofs and further discussion  we refer to \cite[Thms.~8.1, 8.16]{Ha-crc}, \cite[Sects. 13, 15]{T01}. Results for the case $s=\sigma^n_p$ can be found in \cite[Props.~8.12, 8.14, 8.24]{Ha-crc}, \cite{vybiral-4} and \cite{Se-Tr}.
\end{example}

In contrast to Examples~\ref{exm-Lpq}  and \ref{envg-A} it turns out that we have
\begin{equation}\label{eg_M_rn}
\egv{\La^\vr_p(\rn)}(t)= \infty, \quad t>0,\quad -n<\vr<0,
\end{equation}
for the Morrey spaces $\La^\vr_p(\rn)$ according to Definition~\ref{D2.3} with 
 $0<p<\infty$ and $-n<\vr<0$, recall notation \eqref{2.20}, cf. \cite[Thm.~3.7]{Ha-SM-3}, in sharp contrast to
\[
\egv{\La^{-n}_p(\rn)}(t) \sim t^{-\frac1p}, \quad t>0,\quad \vr=-n, 
\]
using \eqref{2.18}. 
Similarly, whenever $\rhoAs(\rn)\subset \Lloc(\rn)$ and $\rhoAs(\rn)\not\hra L_\infty(\rn)$, we proved in \cite[Thms.~4.9, 4.11, Cor.~4.14]{HMS16} (with some forerunner in \cite{Ha-SM-3}) that 
\begin{equation}\label{eg_MA_rn}
\egv{\rhoAs(\rn)}(t)= \infty,\quad t>0, \quad -n<\vr<0.
\end{equation}
This is in good agreement with the earlier result in \cite[Sect.~13.7]{T01} that
\[
\egv{\bmo(\rn)}(t)=\infty, \quad t>0,
\]
recall the coincidence $\bmo(\rn)=\La^0 F^0_{p,2}(\rn)=F^0_{\infty,2}(\rn)$ in \eqref{2.35} and Remark~\ref{bmo-def}. However, for a bounded domain $\Omega$ one has 
\[
\egv{\bmo(\Omega)}(t) \sim |\log t|, \quad t\to 0,
\]
see \cite[Sect.~13.7]{T01}. So these assertions are connected with the underlying $\rn$.

In \cite{HSS-morrey,HMSS-morrey} we studied the growth envelope function for spaces on bounded domains, e.g.
\begin{equation}\label{la0}
\egv{\La^\vr_p(\Omega)}(t) \sim t^{-\frac1p}, \quad t\to 0, \quad -n\leq \vr<0,
\end{equation}
cf. \cite{HSS-morrey} (where some care is needed to define the spaces $\La^\vr_p(\Omega)$, also in dependence on the quality of $\Omega$). Recall that we have parallel results to Theorem~\ref{T5.3} and \ref{T5.7} now. So we are left to consider the strip
\[ \left\{\Big(\frac1p,s\Big): 0<p<\infty, \sigma_p^{|\vr|} \leq s\leq \frac{|\vr|}{p}\right\}
\]
as shown below, recall Figures~\ref{fig-1a} on p.~\pageref{fig-1a} and \ref{fig-1b} on p.~\pageref{fig-1b}.\\
}

\noindent
\begin{minipage}{\textwidth}
  ~\hspace*{\fill}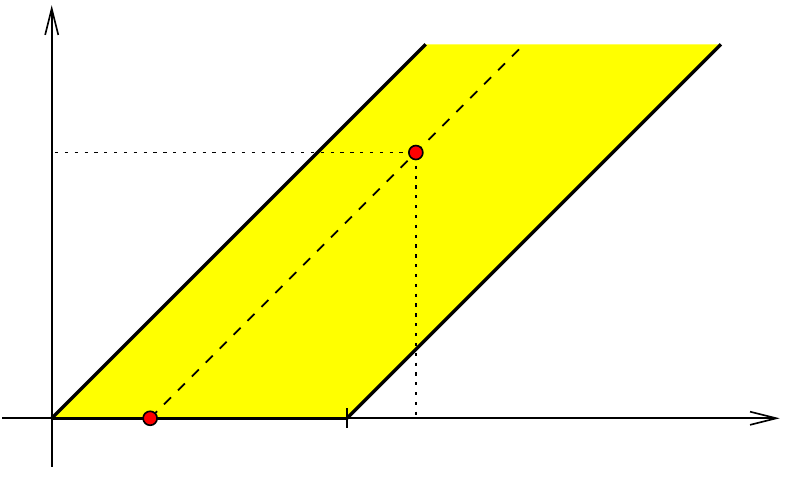\hspace*{\fill}~\\[1ex]    
~\hspace*{\fill}\unterbild{fig-15}
\end{minipage}
\smallskip~

\noindent
\new{
For convenience, we restrict ourselves to the so-called {\em sub-critical case} of spaces $\rhoAs(\Omega)$ omitting the limiting situations $s=\sigma_p^{|\vr|}$ and $s=\frac{|\vr|}{p}$ here. Then we could prove in \cite{HMSS-morrey} the following.

\begin{theorem}\label{eg-rhoAs}
  Let $-n< \vr<0$, $0<p<\infty$, $0<q\leq\infty$, $\sigma_p^{|\vr|}<s<\frac{|\vr|}{p}$, $\Omega\subset\rn$ a bounded  smooth domain. Then
\begin{equation}\label{lar}
\egv{\rhoAs(\Omega)}(t) \sim t^{-\frac 1  p  + \frac{s}{|\vr|}},\quad t\to 0.
\end{equation}
\end{theorem}

We refer to \cite[Proposition~4.1, Corollary~4.2]{HMSS-morrey}. There one can also find further results on the borderline cases $s=\sigma_p^{|\vr|}$ and $s=\frac{|\vr|}{p}$.

Let us denote the number $r\in (1,\infty)$ given by $\frac1r=\frac1p-\frac{s}{|\vr|}$, when $(\frac1p,s)$ belongs to the above strip. Then our results \eqref{la0} and \eqref{lar} yield that
\begin{equation}
  \egv{\La^\vr_r(\Omega)}(t) \sim \egv{\rhoAs(\Omega)}(t) \sim t^{-\frac 1  p  + \frac{s}{|\vr|}} = t^{-\frac1r}, \quad t\to 0,\quad -n< \vr<0,
\end{equation}
that is, all function spaces $\rhoAs(\Omega)$ depicted on the line $s=|\vr|(\frac1p-\frac1r)>0$ with foot point $\frac1r$ possess the same growth envelope function -- and in that sense the same `quality' of unboundedness -- like the space $\La^\vr_r(\Omega)$. This reflects exactly the outcome in case of spaces $\As(\rn)$ and $L_r(\rn)$, that is, when $\vr=-n$, as described in \eqref{eg-A} which remains valid if one replaces $\rn$ by a bounded domain $\Omega$. In other words, this result represents another fine incidence of the Slope--$n$--Rule as formulated in \new{Slope Rules~\ref{slope_rules}(ii)}.

\begin{remark}
  In the monographs \cite{Ha-crc,T01} as well as ion the above mentioned papers \cite{HSS-morrey,HMSS-morrey} we studied, in addition, the so-called {\em growth envelope} $\envg(X) = (\egX(t), \uGx)$, consisting of the growth envelope functions $\egX(t)$, together with some fine index $\uGx\in (0,\infty]$, refining the characterisation. In case of $L_{p,q}(\rn)$ the result reads, for instance, as $\envg(L_{p,q}(\rn))= (t^{-\frac1p}, q)$, recall Example~\ref{exm-Lpq}, and in case  of spaces $\As(\rn)$ we have usually $\uGindex{\Bs(\rn)}=q$ and $\uGindex{\Fs(\rn)}=p$; we refer to the above mentioned literature for further details.
  \end{remark}
}

  \subsection{Crossing borders}\label{S6.3}

  As already announced, our idea is to leave the situation of (only) one $\vr$-clan and ask for new phenomena then. But first we would like to complement the sharp embeddings in the Sections~\ref{S5.1}, \ref{S5.2} as follows.

\new{
\begin{theorem}   \label{T6.5}
  Let $n\in \nat$ and $-n <\vr <0$. Let $0<p_1,p_2<\infty$, $0<q_1, q_2 \le \infty$ and $s_1 \in \real$, $s_2 \in \real$.
\bli
\item[{\upshape\bfseries (i)}]
Then there is a continuous embedding
\begin{\eq}   \label{6.33}
\id_{\rn}: \rhoBe(\rn) \hra \rhoBz(\rn)
\end{\eq}
if, and only if,
\begin{\eq}   \label{6.34}
p_2 \ge p_1 \quad \text{and} \quad
\begin{cases}
\text{either} & s_1 - \frac{|\vr|}{p_1} > s_2 - \frac{|\vr|}{p_2}, \quad 0<q_1, q_2 \le \infty, \\
\text{or} & s_1 - \frac{|\vr|}{p_1} = s_2 - \frac{|\vr|}{p_2}, \quad 0<q_1\le q_2.
\end{cases}
\end{\eq}
\item[{\upshape\bfseries (ii)}]
Then there is a continuous embedding
\begin{\eq}   
  \id_{\rn}: \rhoFe(\rn) \hra \rhoFz(\rn)
\end{\eq}
if, and only if,
\begin{\eq}   
p_2 \ge p_1 \quad \text{and} \quad
\begin{cases}
\text{either} & s_1 - \frac{|\vr|}{p_1} > s_2 - \frac{|\vr|}{p_2}, \quad 0<q_1, q_2 \le \infty, \\
\text{or} & s_1 - \frac{|\vr|}{p_1} = s_2 - \frac{|\vr|}{p_2}, \quad p_1<p_2, \quad 0<q_1, q_2 \le \infty, \\
\text{or} & s_1=s_2, \quad p_1=p_2,\quad  0<q_1\le q_2.
\end{cases}
\end{\eq}
\ignore{
\item[{\upshape\bfseries (iii)}]
Then
\begin{\eq}   \label{DDH-1}
\La^{\vr} B^{s_1}_{p_1,q_1} (\rn) \hra \La^{\vr} B^{s_2}_{p_2,q_2} (\rn)
\end{\eq}
if, and only if,
\begin{\eq}   
p_2 \ge p_1 \quad \text{and} \quad
\begin{cases}
\text{either} & s_1 - \frac{|\vr|}{p_1} > s_2 - \frac{|\vr|}{p_2}, \quad 0<q_1, q_2 \le \infty, \\
 \text{or} & s_1 - \frac{|\vr|}{p_1} = s_2 - \frac{|\vr|}{p_2}, \quad \text{and}\quad  q_1\leq q_2.
  \end{cases}
\end{\eq}
\item[{\upshape\bfseries (iv)}]
  Then
\begin{\eq}   \label{6.33}
\La^{\vr} F^{s_1}_{p_1,q_1} (\rn) \hra \La^{\vr} F^{s_2}_{p_2,q_2} (\rn)
\end{\eq}
if, and only if,
\begin{\eq}   
p_2 \ge p_1 \quad \text{and} \quad
\begin{cases}
\text{either} & s_1 - \frac{|\vr|}{p_1} > s_2 - \frac{|\vr|}{p_2}, \quad 0<q_1, q_2 \le \infty, \\
\text{or} & s_1 - \frac{|\vr|}{p_1} = s_2 - \frac{|\vr|}{p_2}, \quad p_1<p_2, \quad 0<q_1, q_2 \le \infty, \\
\text{or} & s_1=s_2, \quad p_1=p_2,\quad  0<q_1\le q_2.
\end{cases}
\end{\eq}
}
\eli
There are no  compact embeddings $\rhoAx{s_1}{p_1}{q_1}(\rn)\hra\rhoAx{s_2}{p_2}{q_2}(\rn)$.
\end{theorem}

\begin{proof} Part (i) in case of $\rhoBs(\rn)=\La_\vr\Bs(\rn)$ is a special case of \cite[Theorem~3.3]{HaS12}, whereas  the case $\rhoBs(\rn)=\La^\vr\Bs(\rn)$ is covered by \cite[Theorem~2.5]{YHSY15} and \cite[Corollary~4.21]{GHS21b}. The result (ii) is covered by \cite[Theorem~3.1]{HaS14} when $\rhoFs(\rn)=\La_\vr\Fs(\rn)$, and by
\cite[Corollary~5.9]{YHSY15} in case of $\rhoFs(\rn)=\La^\vr\Fs(\rn)$ and also by \eqref{2.42}. 
\end{proof}

\begin{remark}
  We refer to the papers \cite{HaS12,HaS14,YHSY15,YHMSY15,GHS21b} for further embedding results of $\rhoAs$ spaces on $\rn$, in particular when $\vr_1\neq \vr_2$ (as is always assumed in our setting here). The difference between  Theorem~\ref{T6.5} for spaces on $\rn$ and Theorem~\ref{T5.39} for spaces on a bounded domain $\Omega$ concerns (apart from the obviously missing compactness assertion) essentially the extension to parameters $p_1>p_2$ which is a matter of H\"older's inequality. This is well known from the classical spaces already and also reflected in the comparison of the corresponding diagrams Figure~\ref{fig-8} on p.~\pageref{fig-8} and Figure~\ref{fig-7} below.  
\end{remark}  
}

\noindent
\begin{minipage}{\textwidth}
  ~\hspace*{\fill}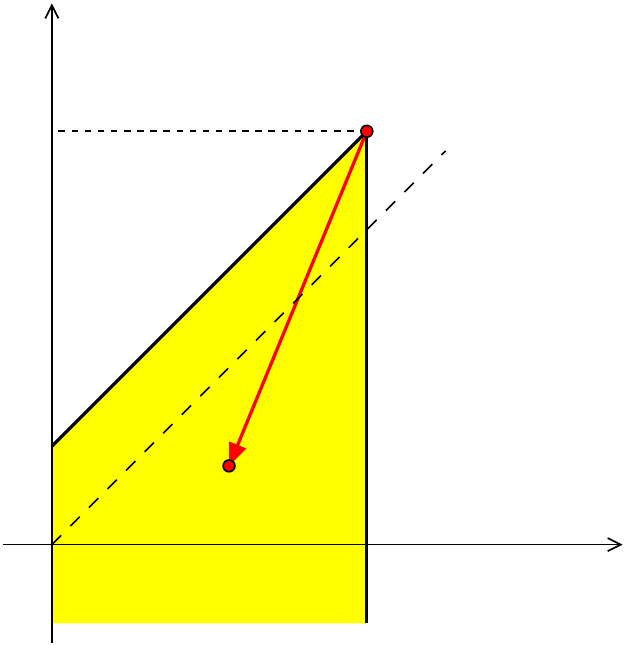\hspace*{\fill}~\\[1ex]    
~\hspace*{\fill}\unterbild{fig-7}
\end{minipage}
\smallskip~

\new{
  \begin{remark}\label{R-non-Lr}
  Let us briefly mention some phenomenon which essentially distinguishes the $\vr$-clan $\rhoAs(\rn)$ with $\varrho>-n$ from its classical counterpart $\As(\rn)$, corresponding to $\vr=-n$, recall Proposition~\ref{P2.9}(i).
 Whenever $\vr>-n$, then for the embedding into a Lebesgue space $L_r(\rn)$, $0<r\leq\infty$, we get 
  \[
  \rhoAs(\rn) \hra L_r(\rn) \quad \text{if, and only if,}\quad r=\infty,
  \]
for appropriately chosen $s$, $p$ and $q$. This somewhat surprising result has been proved in \cite[Corollary 3.6]{HaS14}  for spaces of type $\La_\vr\As(\rn)$, and in \cite[Proposition~5.7]{YHSY15}  for spaces of type $\La^\vr\As(\rn)$.

In contrast to this situation, we have embeddings $\rhoAs(\rn)\hookrightarrow \bmo(\rn)$, recall Remark~\ref{bmo-def} for the definition of $\bmo(\rn)$. Then
for $-n\leq \vr<0$,
\[
\La_\vr \As(\rn) \hra \bmo(\rn)\quad \text{if, and only if,}\quad s\geq \frac{|\vr|}{p},
\]
cf. \cite[Corollary~5.1]{HMS20}. The parallel result for the case $\rhoAs(\rn)=\La^\vr \As(\rn)$ can be found in \cite[Propositions 5.12, 5.13]{YHSY15}. 
\end{remark}

  So far we dealt with properties of the spaces $\rhoAs(\rn)$ according to Definition~\ref{D2.11} {\em within a fixed $\vr$-clan}. One may ask how spaces belonging to different $\vr$-clans are related to each other. Here we summarise our results extending Theorem~\ref{T6.5} for possibly different parameters $\vr_1$ and $\vr_2$.

  \begin{theorem}   \label{T7.1}
  Let $n\in \nat$ and $-n < \vr_1,\vr_2 <0$. Let $0<p_1, p_2<\infty$ and $0< q_1, q_2 \le \infty$. Let $s_1, s_2 \in \real$.
  \bli
\item[{\upshape\bfseries (i)}]
Then there is a continuous embedding
\begin{\eq}
  \id_{\rn}: 
\La_{\vr_1} B^{s_1}_{p_1,q_1} (\rn) \hra \La_{\vr_2} B^{s_2}_{p_2,q_2} (\rn)
\end{\eq}
if, and only if,
\begin{\eq}   \label{DDH-7}
  |\vr_2|\leq |\vr_1|,  \quad \frac{|\vr_1|}{p_1} \geq \frac{|\vr_2|}{p_2},
  \end{\eq}
  \text{and}
  \begin{\eq}
  \begin{cases}
\text{either} & s_1 - \frac{|\vr_1|}{p_1} > s_2 - \frac{|\vr_2|}{p_2}, \quad 0<q_1, q_2 \le \infty, \\
\text{or} & s_1 - \frac{|\vr_1|}{p_1} = s_2 - \frac{|\vr_2|}{p_2}, \quad 0<q_1\le q_2.
\end{cases}
\end{\eq}
  \item[{\upshape\bfseries (ii)}]
  Then there is a continuous embedding
\begin{\eq}   
 \La_{\vr_1} F^{s_1}_{p_1,q_1} (\rn) \hra \La_{\vr_2} F^{s_2}_{p_2,q_2} (\rn) 
\end{\eq}
if, and only if, \eqref{DDH-7} is satisfied and
\begin{\eq}\label{DDH-8}
\begin{cases}
\text{either} & s_1 - \frac{|\vr_1|}{p_1} > s_2 - \frac{|\vr_2|}{p_2}, \quad 0<q_1, q_2 \le \infty, \\
\text{or} & s_1 - \frac{|\vr_1|}{p_1} = s_2 - \frac{|\vr_2|}{p_2}, \quad p_1<p_2, \quad 0<q_1, q_2 \le \infty, \\
\text{or} & s_1=s_2, \quad p_1=p_2,\quad  0<q_1\le q_2.
\end{cases}
\end{\eq}
\item[{\upshape\bfseries (iii)}]
  Then there is a continuous embedding
\begin{\eq}\label{DDH-9}
 \La^{\vr_1} B^{s_1}_{p_1,q_1} (\rn) \hra \La^{\vr_2} B^{s_2}_{p_2,q_2} (\rn) 
\end{\eq}
if \eqref{DDH-7} is  satisfied and
\begin{\eq}
  s_1 - \frac{|\vr_1|}{p_1} > s_2 - \frac{|\vr_2|}{p_2}.   
\end{\eq}
Conversely, if \eqref{DDH-9} is continuous, then
\eqref{DDH-7} is satisfied and $s_1\geq s_2$, $s_1 - \frac{|\vr_1|}{p_1}\geq s_2 - \frac{|\vr_2|}{p_2}$ and $q_1\leq q_2$ if $s_1=s_2$.
\item[{\upshape\bfseries (iv)}]
  Then there is a continuous embedding
\begin{\eq}   
 \La^{\vr_1} F^{s_1}_{p_1,q_1} (\rn) \hra \La^{\vr_2} F^{s_2}_{p_2,q_2} (\rn) 
\end{\eq}
if, and only if, \eqref{DDH-7} and \eqref{DDH-8} are satisfied.
\eli
  \end{theorem}

  \begin{remark}\label{R6.16}
    The above result (i) coincides with \cite[Theorem~3.3]{HaS12}, and (ii) with \cite[Theorem~3.1]{HaS14}. In particular, there are no continuous embeddings if $|\vr_2|>|\vr_1|$. Part (iii) is covered by \cite[Theorem~2.5]{YHSY15}. In fact, there are further sufficient conditions for the continuity of \eqref{DDH-9} in the limiting case which we omitted here for simplicity, we also refer to some extension of this result in \cite[Corollary~4.21]{GHS21b}. Finally, part (iv) coincides with \cite[Corollary~5.9]{YHSY15}, now also covered by part (ii) and \eqref{2.42}. Note that Theorem~\ref{T6.5} is just a special case of Theorem~\ref{T7.1} for the case $\vr=\vr_1=\vr_2$ where,  in addition, we have a complete result also for the case (iii) above. 
    So -- apart from the small gap in the limiting case of (iii) -- it is known that
    \[
    \id_{\rn} : \rhoeAe(\rn) \hra \rhozAz(\rn)
    \]
    if, and only if, \eqref{DDH-7} is satisfied and
    \[
     s_1 - \frac{|\vr_1|}{p_1} \geq s_2 - \frac{|\vr_2|}{p_2},
    \]
    with some additional assumptions in the limiting case. 
    We illustrate the situation in Figure~\ref{fig-12} below. Clearly this diagram coincides with the above Figure~\ref{fig-7} on p.~\pageref{fig-7} in case of $\vr_1=\vr_2=\vr$.
  \end{remark}

  \noindent
\begin{minipage}{\textwidth}
  ~\hspace*{\fill}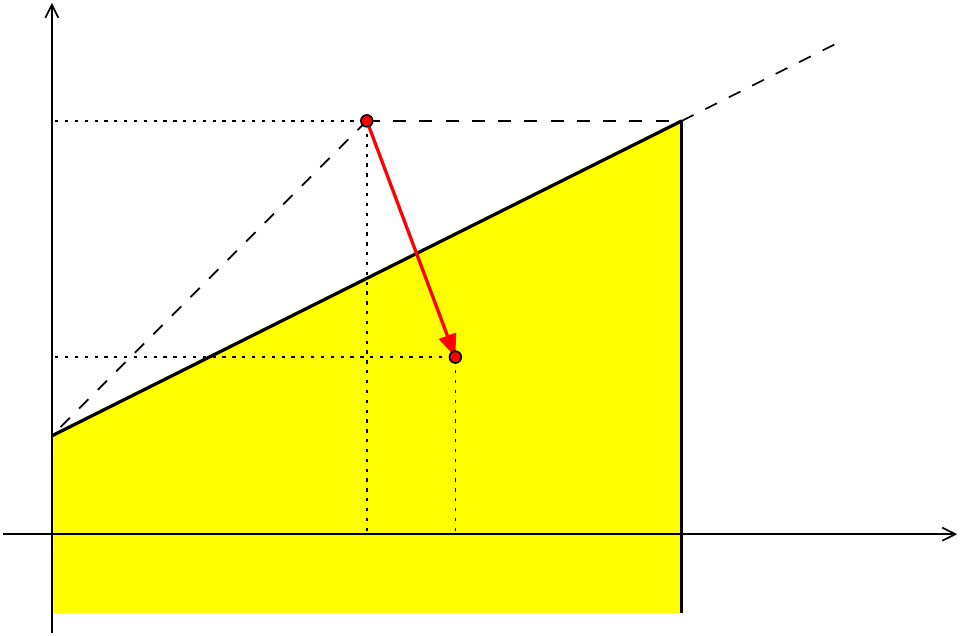\hspace*{\fill}~\\[1ex]    
~\hspace*{\fill}$\id_{\rn}$~ for the case $|\vr_1|\geq |\vr_2|$~\hspace*{\fill}\unterbild{fig-12}
\end{minipage}
\smallskip~

\begin{remark}\label{R6.18}
  We excluded so far $\As(\rn)$ according to \eqref{2.38}, that is, the $n$-clan in terms of Definition~\ref{D2.11}(i), as target spaces. However, in view of the first condition in \eqref{DDH-7}, one may have doubts whether
there is a continuous embedding of type $\rhoAe(\rn) \hookrightarrow \Az(\rn)$ when $-n<\vr<0$. In fact, this turns out to be wrong, more precisely,  
  \begin{equation}
\rhoAe(\rn) \not\subseteq \Az(\rn)
   \end{equation} 
if $-n<\vr<0$, $0<p_1,p_2<\infty$, $0<q_1,q_2\leq\infty$, $s_1,s_2\in\real$. This -- at first glance surprising -- result is covered by the references given in Remark~\ref{R6.16}.  
\end{remark}

\begin{remark}
  The case $F^s_{\infty,q}(\rn)$ can be incorporated by embeddings, but is also partly covered by \eqref{ftbt} together with \cite[Corollary~5.8]{YHSY15} and \cite[Corollary~4.23]{GHS21b}.    
  \end{remark}

We return to Section~\ref{S5.7}, in particular to Theorem~\ref{T5.39}, and ask for continuous and compact embeddings between different $\vr$-clans. For convenience we restrict ourselves to the question of compactness only.

\begin{theorem}  \label{T5.39a}
Let $\Om$ be a bounded domain in \rn, $n\in \nat$. Let $-n < \vr_1, \vr_2 <0$, $0<p_1, p_2<\infty$, $0< q_1, q_2 \le \infty$, and $s_1, s_2\in\real$.
\bli
\item
Assume $|\vr_1|\leq |\vr_2|$. Then 
  \begin{equation}
   \id_{\Om} : \rhoeAe(\Omega) \hookrightarrow \rhozAz(\Omega)
 \end{equation}
    is compact if, and only if,
\begin{\eq}   
s_1 - s_2 > |\vr_1| \cdot \max \Big( \frac{1}{p_1} - \frac{1}{p_2}, 0 \Big).
\end{\eq}
\item
Assume $|\vr_1|\geq |\vr_2|$. Then 
  \begin{equation}
   \id_{\Om} : \rhoeAe(\Omega) \hookrightarrow \rhozAz(\Omega)
 \end{equation}
    is compact if, and only if,
\begin{\eq}   
s_1 - s_2 > \max\Big( \frac{|\vr_1|}{p_1} -\frac{|\vr_2|}{p_2}, 0 \Big).
\end{\eq}
\eli
\end{theorem}

\begin{remark}\label{R6.20}
The above results are covered by \cite{HaS13,HaS14,HaS20,GHS21,GHS21b}, including also further limiting cases for continuous embeddings. The extension to parameters $p_2<p_1$ is due to H\"older's inequality and the boundedness of $\Omega$. 
\end{remark}

\noindent
\begin{minipage}{\textwidth}
  ~\hspace*{\fill}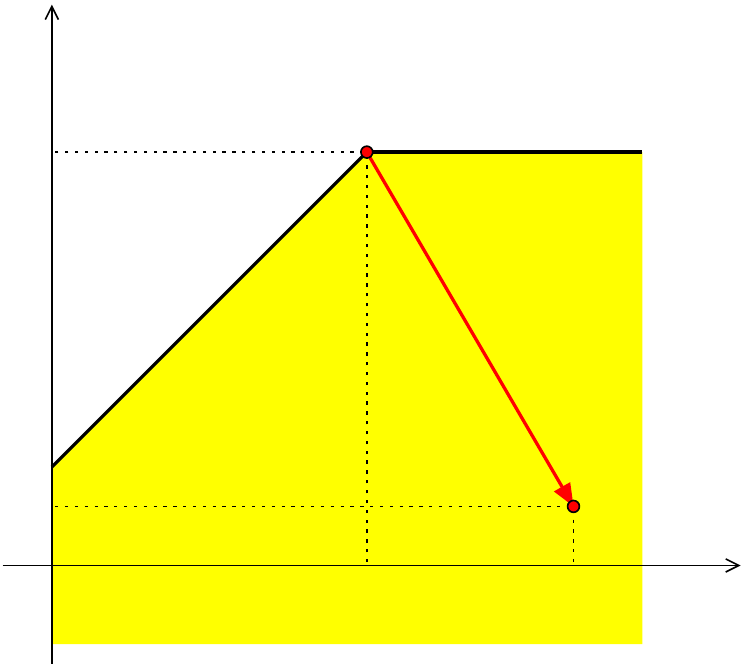\hspace*{\fill}~\\[1ex]    
~\hspace*{\fill}$\id_\Omega$~ in case of $|\vr_1|\leq |\vr_2|$~\hspace*{\fill} \unterbild{fig-14}
\end{minipage}
\smallskip~

\noindent
In case of $|\vr_1|\leq |\vr_2|$ the picture remains essentially the same, compare Figure~\ref{fig-8} on p.~\pageref{fig-8} and Figure~\ref{fig-14}. However, when $|\vr_1|\geq |\vr_2|$, then we meet again Figure~\ref{fig-12} on p.~\pageref{fig-12}, extended to the right by H\"older's inequality again.\\

  \noindent
\begin{minipage}{\textwidth}
  ~\hspace*{\fill}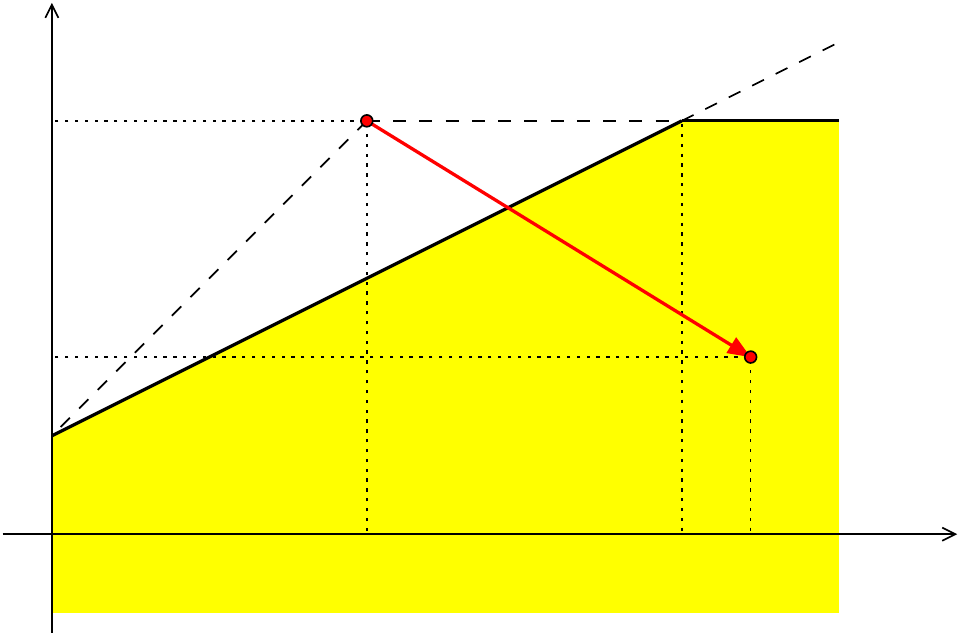\hspace*{\fill}~\\[1ex]    
~\hspace*{\fill}$\id_\Omega$~ in case of $|\vr_1|\geq |\vr_2|$~\hspace*{\fill} \unterbild{fig-13}
\end{minipage}
\medskip~
}

  In contrast to Remark~\ref{R6.18} we may now replace the target spaces $\rhozAz(\Omega)$ by their classical counterparts $\Az(\Omega)$. This corresponds to Theorem~\ref{T5.39a}(i) and is covered by the references listed in Remark~\ref{R6.20}.

  \begin{corollary}
Let $\Om$ be a bounded domain in \rn, $n\in \nat$. Let $-n < \vr <0$, $0<p_1, p_2<\infty$, $0< q_1, q_2 \le \infty$, and $s_1, s_2\in\real$. Then 
  \begin{equation}
   \id_{\Om} : \rhoAe(\Omega) \hookrightarrow \Az(\Omega)
 \end{equation}
    is compact if, and only if,
\begin{\eq}   
s_1 - s_2 > |\vr| \cdot \max \Big( \frac{1}{p_1} - \frac{1}{p_2}, 0 \Big).
\end{\eq}    
  \end{corollary}

\begin{remark}
  One can further derive entropy number results  parallel to \eqref{5.95} from \cite{HaS20,GHS21}, for instance, for $\id_\Omega:\rhoeAe(\Omega)\hra \rhozAz(\Omega)$ we have
  \[
  e_k(\id_\Omega) \sim k^{-\frac{s_1-s_2}{n}},\quad k\in\nat,
  \]
  in all cases except 
  \[
  |\vr_1|<|\vr_2| \quad\text{and}\quad 0<|\vr_1|\left(\frac{1}{p_1}-\frac{1}{p_2}\right) < s_1-s_2 \leq n\left(\frac{1}{p_1}-\frac{1}{p_2}\right).
  \]
  In that latter case we obtained in \cite{HaS20,GHS21} that there exists some $c>0$ and for any $\varepsilon>0$ a number $c_\varepsilon>0$ such that
  \[
c\ k^{-\alpha} \leq   e_k(\id_\Omega) \leq c_\varepsilon \ k^{-\alpha+\varepsilon},\quad k\in\nat,
  \]
  with
  \[ \alpha = \frac{1}{n-|\vr_1|}\left(s_1-s_2-|\vr_1|\left(\frac{1}{p_1}-\frac{1}{p_2}\right)\right)>0.
  \]
  Note that this case cannot appear in the classical setting when $\vr_1=-n$, recall Proposition~\ref{P5.37}. Further results in the context of Theorem~\ref{T5.39a} can be found in   \cite{HaS20,GHS21}.
\end{remark}


  

\ignore{
\subsection{Some further approaches / Lorentz smoothness spaces}   \label{S6.4}
This paper deals with the extension of distinguished properties  of the classical spaces $\As (\rn)$ in \eqref{2.38} to the Morrey
smoothness spaces $\rhoAs(\rn)$ in Definition~\ref{D2.11}. If one replaces the  Lebesgue spaces $L_p (\rn)$ in \eqref{2.10},
\eqref{2.12} by the Morrey spaces $\La^{\vr}_p (\rn)$, then one obtains the spaces $\La_{\vr} \As (\rn)$ as introduced in Definition
\ref{D2.5}. But instead of $\La^{\vr}_p (\rn)$ as basic spaces one may choose other distinguished spaces, especially the very 
classical well--known Lorentz spaces $L_{p,r} (\rn)$ resulting in,
\begin{\eq}   \label{6.35}
B^s_q L_{p,r} (\rn) \ \text{and} \ F^s_q L_{p,r} (\rn) \quad \text{where $0<p<\infty$, $s\in \real$ and $0<q,r \le \infty$}.
\end{\eq}
This attracted some attention over the years. Substantial embeddings results have been obtained recently in \cite{SeT19}. The first
three of the four key problems as described in Section~\ref{S4.1} for these spaces have been treated recently in \cite{BCT21}. 
However traces for these spaces on hyper--planes seems to be a rather tricky subject. First assertions may be found in \cite[Section 
4]{BHT21}. We do not go into details and refer the reader for definitions, (historical) references and results to the quoted papers.
Final assertions about atomic and wavelet expansions for the spaces in \eqref{6.35} have been derived recently in \cite{BeC21}. 
Furthermore it comes out that these spaces can be obtained by iterated real interpolation of the spaces $\Fs (\rn)$. This ensures that
breaking lines of distinguished properties for the these spaces are the same as for the classical spaces $\As (\rn)$, in sharp 
contrast to the Morrey smoothness spaces.

\open{remarks and some references about weighted spaces, generalised Besov-Morrey spaces \cite{HMS20a} \dots ?}
}

\newpage
\addcontentsline{toc}{section}{References}
\bibliographystyle{alpha}
\bibliography{HT_6}

\addcontentsline{toc}{section}{List of figures}
\listoffigures

\end{document}

%% file: fig_rho_1a.pdf_t
\begin{picture}(0,0)%
\includegraphics{fig_rho_1a.pdf}%
\end{picture}%
\setlength{\unitlength}{4144sp}%
\begingroup\makeatletter\ifx\SetFigFont\undefined%
\gdef\SetFigFont#1#2#3#4#5{%
  \reset@font\fontsize{#1}{#2pt}%
  \fontfamily{#3}\fontseries{#4}\fontshape{#5}%
  \selectfont}%
\fi\endgroup%
\begin{picture}(3489,3522)(664,-3898)
\put(1846,-2986){\makebox(0,0)[b]{\smash{{\SetFigFont{10}{12.0}{\familydefault}{\mddefault}{\updefault}$\frac{n}{|\vr|}$}}}}
\put(1261,-2986){\makebox(0,0)[b]{\smash{{\SetFigFont{12}{14.4}{\familydefault}{\mddefault}{\updefault}$\frac1p$}}}}
\put(811,-1051){\makebox(0,0)[rb]{\smash{{\SetFigFont{12}{14.4}{\familydefault}{\mddefault}{\updefault}$s$}}}}
\put(901,-511){\makebox(0,0)[lb]{\smash{{\SetFigFont{12}{14.4}{\familydefault}{\mddefault}{\updefault}$\rhoAs$}}}}
\put(1576,-1411){\makebox(0,0)[b]{\smash{{\SetFigFont{12}{14.4}{\familydefault}{\mddefault}{\updefault}$\rhoAs\hra L_\infty$}}}}
\put(811,-1681){\makebox(0,0)[rb]{\smash{{\SetFigFont{12}{14.4}{\familydefault}{\mddefault}{\updefault}$s$}}}}
\put(2791,-3391){\makebox(0,0)[b]{\smash{{\SetFigFont{12}{14.4}{\familydefault}{\mddefault}{\updefault}$\delta\in\rhoAs$}}}}
\put(2836,-1861){\makebox(0,0)[lb]{\smash{{\SetFigFont{12}{14.4}{\familydefault}{\mddefault}{\updefault}$s=\frac{|\vr|}{p}-n$}}}}
\put(1801,-1996){\makebox(0,0)[lb]{\smash{{\SetFigFont{12}{14.4}{\familydefault}{\mddefault}{\updefault}$s=\frac{|\vr|}{p}$}}}}
\put(811,-3706){\makebox(0,0)[rb]{\smash{{\SetFigFont{12}{14.4}{\familydefault}{\mddefault}{\updefault}$-n$}}}}
\put(4051,-2986){\makebox(0,0)[b]{\smash{{\SetFigFont{12}{14.4}{\familydefault}{\mddefault}{\updefault}$\frac1p$}}}}
\put(1486,-2986){\makebox(0,0)[b]{\smash{{\SetFigFont{10}{12.0}{\familydefault}{\mddefault}{\updefault}$1$}}}}
\end{picture}%

%% file: fig_rho_1b.pdf_t
\begin{picture}(0,0)%
\includegraphics{fig_rho_1b.pdf}%
\end{picture}%
\setlength{\unitlength}{4144sp}%
\begingroup\makeatletter\ifx\SetFigFont\undefined%
\gdef\SetFigFont#1#2#3#4#5{%
  \reset@font\fontsize{#1}{#2pt}%
  \fontfamily{#3}\fontseries{#4}\fontshape{#5}%
  \selectfont}%
\fi\endgroup%
\begin{picture}(3579,2196)(664,-3055)
\put(2251,-2986){\makebox(0,0)[b]{\smash{{\SetFigFont{10}{12.0}{\familydefault}{\mddefault}{\updefault}$1$}}}}
\put(4141,-2986){\makebox(0,0)[b]{\smash{{\SetFigFont{12}{14.4}{\familydefault}{\mddefault}{\updefault}$\frac1p$}}}}
\put(1351,-2986){\makebox(0,0)[b]{\smash{{\SetFigFont{12}{14.4}{\familydefault}{\mddefault}{\updefault}$\frac1p$}}}}
\put(811,-1051){\makebox(0,0)[rb]{\smash{{\SetFigFont{12}{14.4}{\familydefault}{\mddefault}{\updefault}$s$}}}}
\put(811,-1681){\makebox(0,0)[rb]{\smash{{\SetFigFont{12}{14.4}{\familydefault}{\mddefault}{\updefault}$s$}}}}
\put(1576,-1411){\makebox(0,0)[b]{\smash{{\SetFigFont{12}{14.4}{\familydefault}{\mddefault}{\updefault}$\rhoAs\subset L_1^{\loc}$}}}}
\put(1711,-2086){\makebox(0,0)[lb]{\smash{{\SetFigFont{12}{14.4}{\familydefault}{\mddefault}{\updefault}$s=\frac{|\vr|}{p}$}}}}
\put(3151,-1996){\makebox(0,0)[lb]{\smash{{\SetFigFont{12}{14.4}{\familydefault}{\mddefault}{\updefault}$s=\sigma_p^{|\vr|}$}}}}
\end{picture}%

%% file: fig_rho_2a.pdf_t
\begin{picture}(0,0)%
\includegraphics{fig_rho_2a.pdf}%
\end{picture}%
\setlength{\unitlength}{4144sp}%
\begingroup\makeatletter\ifx\SetFigFont\undefined%
\gdef\SetFigFont#1#2#3#4#5{%
  \reset@font\fontsize{#1}{#2pt}%
  \fontfamily{#3}\fontseries{#4}\fontshape{#5}%
  \selectfont}%
\fi\endgroup%
\begin{picture}(4119,2454)(664,-3313)
\put(811,-1051){\makebox(0,0)[rb]{\smash{{\SetFigFont{12}{14.4}{\familydefault}{\mddefault}{\updefault}$s$}}}}
\put(2476,-1096){\makebox(0,0)[rb]{\smash{{\SetFigFont{12}{14.4}{\familydefault}{\mddefault}{\updefault}$s=\frac1p$}}}}
\put(2026,-2986){\makebox(0,0)[b]{\smash{{\SetFigFont{12}{14.4}{\familydefault}{\mddefault}{\updefault}$1$}}}}
\put(811,-1636){\makebox(0,0)[rb]{\smash{{\SetFigFont{12}{14.4}{\familydefault}{\mddefault}{\updefault}$1$}}}}
\put(2251,-2221){\makebox(0,0)[lb]{\smash{{\SetFigFont{12}{14.4}{\familydefault}{\mddefault}{\updefault}$s=\frac{|\vr|}{p}$, $|\vr|< 1$}}}}
\put(4636,-2986){\makebox(0,0)[b]{\smash{{\SetFigFont{12}{14.4}{\familydefault}{\mddefault}{\updefault}$\frac1p$}}}}
\put(3241,-1231){\makebox(0,0)[b]{\smash{{\SetFigFont{12}{14.4}{\familydefault}{\mddefault}{\updefault}$|\vr|\geq 1$}}}}
\put(3151,-1636){\makebox(0,0)[b]{\smash{{\SetFigFont{12}{14.4}{\familydefault}{\mddefault}{\updefault}$\chi_Q\in \rhoAs(\rn)$}}}}
\end{picture}%

%% file: fig_rho_2b.pdf_t
\begin{picture}(0,0)%
\includegraphics{fig_rho_2b.pdf}%
\end{picture}%
\setlength{\unitlength}{4144sp}%
\begingroup\makeatletter\ifx\SetFigFont\undefined%
\gdef\SetFigFont#1#2#3#4#5{%
  \reset@font\fontsize{#1}{#2pt}%
  \fontfamily{#3}\fontseries{#4}\fontshape{#5}%
  \selectfont}%
\fi\endgroup%
\begin{picture}(2724,3534)(664,-4483)
\put(1801,-2986){\makebox(0,0)[b]{\smash{{\SetFigFont{12}{14.4}{\familydefault}{\mddefault}{\updefault}$1$}}}}
\put(811,-3661){\makebox(0,0)[rb]{\smash{{\SetFigFont{12}{14.4}{\familydefault}{\mddefault}{\updefault}$-1$}}}}
\put(3061,-1411){\makebox(0,0)[rb]{\smash{{\SetFigFont{12}{14.4}{\familydefault}{\mddefault}{\updefault}$s=\frac1p-1$}}}}
\put(2701,-2986){\makebox(0,0)[b]{\smash{{\SetFigFont{12}{14.4}{\familydefault}{\mddefault}{\updefault}$\frac{1}{|\vr|}$}}}}
\put(811,-1186){\makebox(0,0)[rb]{\smash{{\SetFigFont{12}{14.4}{\familydefault}{\mddefault}{\updefault}$s$}}}}
\put(1576,-3436){\makebox(0,0)[lb]{\smash{{\SetFigFont{12}{14.4}{\familydefault}{\mddefault}{\updefault}$s=\frac{|\vr|}{p}-1$}}}}
\put(3286,-2986){\makebox(0,0)[b]{\smash{{\SetFigFont{12}{14.4}{\familydefault}{\mddefault}{\updefault}$\frac1p$}}}}
\put(1576,-2086){\makebox(0,0)[b]{\smash{{\SetFigFont{12}{14.4}{\familydefault}{\mddefault}{\updefault}\eqref{5.69}}}}}
\end{picture}%

%% file: fig_rho_2c.pdf_t
\begin{picture}(0,0)%
\includegraphics{fig_rho_2c.pdf}%
\end{picture}%
\setlength{\unitlength}{4144sp}%
\begingroup\makeatletter\ifx\SetFigFont\undefined%
\gdef\SetFigFont#1#2#3#4#5{%
  \reset@font\fontsize{#1}{#2pt}%
  \fontfamily{#3}\fontseries{#4}\fontshape{#5}%
  \selectfont}%
\fi\endgroup%
\begin{picture}(2499,3534)(664,-4483)
\put(1801,-2986){\makebox(0,0)[b]{\smash{{\SetFigFont{12}{14.4}{\familydefault}{\mddefault}{\updefault}$1$}}}}
\put(811,-3661){\makebox(0,0)[rb]{\smash{{\SetFigFont{12}{14.4}{\familydefault}{\mddefault}{\updefault}$-1$}}}}
\put(811,-4336){\makebox(0,0)[rb]{\smash{{\SetFigFont{12}{14.4}{\familydefault}{\mddefault}{\updefault}$-|\vr|$}}}}
\put(811,-1186){\makebox(0,0)[rb]{\smash{{\SetFigFont{12}{14.4}{\familydefault}{\mddefault}{\updefault}$s$}}}}
\put(2926,-2986){\makebox(0,0)[b]{\smash{{\SetFigFont{12}{14.4}{\familydefault}{\mddefault}{\updefault}$\frac1p$}}}}
\put(2341,-2356){\makebox(0,0)[lb]{\smash{{\SetFigFont{12}{14.4}{\familydefault}{\mddefault}{\updefault}$s=\frac1p-1$}}}}
\put(1351,-3796){\makebox(0,0)[lb]{\smash{{\SetFigFont{12}{14.4}{\familydefault}{\mddefault}{\updefault}$s=|\vr|(\frac1p-1)$}}}}
\put(1576,-2086){\makebox(0,0)[b]{\smash{{\SetFigFont{12}{14.4}{\familydefault}{\mddefault}{\updefault}\eqref{5.69}}}}}
\end{picture}%

%% file: fig_rho_4a.pdf_t
\begin{picture}(0,0)%
\includegraphics{fig_rho_4a.pdf}%
\end{picture}%
\setlength{\unitlength}{4144sp}%
\begingroup\makeatletter\ifx\SetFigFont\undefined%
\gdef\SetFigFont#1#2#3#4#5{%
  \reset@font\fontsize{#1}{#2pt}%
  \fontfamily{#3}\fontseries{#4}\fontshape{#5}%
  \selectfont}%
\fi\endgroup%
\begin{picture}(3129,2551)(259,-1700)
\put(3151,-1636){\makebox(0,0)[lb]{\smash{{\SetFigFont{12}{14.4}{\familydefault}{\mddefault}{\updefault}{\color[rgb]{0,0,0}$\frac1p$}%
}}}}
\put(1351,-1636){\makebox(0,0)[b]{\smash{{\SetFigFont{12}{14.4}{\familydefault}{\mddefault}{\updefault}{\color[rgb]{0,0,0}$1$}%
}}}}
\put(361,-331){\makebox(0,0)[rb]{\smash{{\SetFigFont{12}{14.4}{\familydefault}{\mddefault}{\updefault}{\color[rgb]{0,0,0}$1$}%
}}}}
\put(2251,-1096){\makebox(0,0)[lb]{\smash{{\SetFigFont{12}{14.4}{\familydefault}{\mddefault}{\updefault}{\color[rgb]{0,0,0}$s=|\vr|(\frac1p-1)$}%
}}}}
\put(1351,299){\makebox(0,0)[rb]{\smash{{\SetFigFont{12}{14.4}{\familydefault}{\mddefault}{\updefault}{\color[rgb]{0,0,0}$s=1+\frac{|\vr|}{p}$}%
}}}}
\put(361,704){\makebox(0,0)[rb]{\smash{{\SetFigFont{12}{14.4}{\familydefault}{\mddefault}{\updefault}{\color[rgb]{0,0,0}$s$}%
}}}}
\end{picture}%

%% file: fig_rho_4b.pdf_t
\begin{picture}(0,0)%
\includegraphics{fig_rho_4b.pdf}%
\end{picture}%
\setlength{\unitlength}{4144sp}%
\begingroup\makeatletter\ifx\SetFigFont\undefined%
\gdef\SetFigFont#1#2#3#4#5{%
  \reset@font\fontsize{#1}{#2pt}%
  \fontfamily{#3}\fontseries{#4}\fontshape{#5}%
  \selectfont}%
\fi\endgroup%
\begin{picture}(2409,2551)(304,-2825)
\put(361,-1906){\makebox(0,0)[rb]{\smash{{\SetFigFont{12}{14.4}{\familydefault}{\mddefault}{\updefault}$1$}}}}
\put(1126,-2761){\makebox(0,0)[b]{\smash{{\SetFigFont{12}{14.4}{\familydefault}{\mddefault}{\updefault}$1$}}}}
\put(361,-511){\makebox(0,0)[rb]{\smash{{\SetFigFont{12}{14.4}{\familydefault}{\mddefault}{\updefault}$s$}}}}
\put(1396,-781){\makebox(0,0)[rb]{\smash{{\SetFigFont{12}{14.4}{\familydefault}{\mddefault}{\updefault}$s=1+\frac1p$}}}}
\put(2476,-2761){\makebox(0,0)[b]{\smash{{\SetFigFont{12}{14.4}{\familydefault}{\mddefault}{\updefault}$\frac1p$}}}}
\put(1486,-1861){\makebox(0,0)[lb]{\smash{{\SetFigFont{12}{14.4}{\familydefault}{\mddefault}{\updefault}$s=|\vr|(\frac1p-1)$}}}}
\put(1756,-2761){\makebox(0,0)[b]{\smash{{\SetFigFont{12}{14.4}{\familydefault}{\mddefault}{\updefault}$\frac{|\vr|+1}{|\vr|-1}$}}}}
\end{picture}%

%% file: fig_rho_5.pdf_t
\begin{picture}(0,0)%
\includegraphics{fig_rho_5.pdf}%
\end{picture}%
\setlength{\unitlength}{4144sp}%
\begingroup\makeatletter\ifx\SetFigFont\undefined%
\gdef\SetFigFont#1#2#3#4#5{%
  \reset@font\fontsize{#1}{#2pt}%
  \fontfamily{#3}\fontseries{#4}\fontshape{#5}%
  \selectfont}%
\fi\endgroup%
\begin{picture}(2679,2499)(709,-2548)
\put(2161,-1636){\makebox(0,0)[b]{\smash{{\SetFigFont{12}{14.4}{\familydefault}{\mddefault}{\updefault}{\color[rgb]{0,0,0}$1\!+\!\frac{1}{n}$}%
}}}}
\put(1801,-1636){\makebox(0,0)[b]{\smash{{\SetFigFont{12}{14.4}{\familydefault}{\mddefault}{\updefault}{\color[rgb]{0,0,0}$1$}%
}}}}
\put(856,-556){\makebox(0,0)[rb]{\smash{{\SetFigFont{12}{14.4}{\familydefault}{\mddefault}{\updefault}{\color[rgb]{0,0,0}$1$}%
}}}}
\put(856,-2356){\makebox(0,0)[rb]{\smash{{\SetFigFont{12}{14.4}{\familydefault}{\mddefault}{\updefault}{\color[rgb]{0,0,0}$-1$}%
}}}}
\put(856,-286){\makebox(0,0)[rb]{\smash{{\SetFigFont{12}{14.4}{\familydefault}{\mddefault}{\updefault}{\color[rgb]{0,0,0}$s$}%
}}}}
\put(3061,-1636){\makebox(0,0)[lb]{\smash{{\SetFigFont{12}{14.4}{\familydefault}{\mddefault}{\updefault}{\color[rgb]{0,0,0}$\frac1p$}%
}}}}
\put(1261,-736){\makebox(0,0)[b]{\smash{{\SetFigFont{12}{14.4}{\familydefault}{\mddefault}{\updefault}{\color[rgb]{0,0,0}$s=\frac1p$}%
}}}}
\put(1216,-2131){\makebox(0,0)[lb]{\smash{{\SetFigFont{12}{14.4}{\familydefault}{\mddefault}{\updefault}{\color[rgb]{0,0,0}$s=\frac1p-1$}%
}}}}
\put(1936,-1006){\makebox(0,0)[lb]{\smash{{\SetFigFont{12}{14.4}{\familydefault}{\mddefault}{\updefault}{\color[rgb]{0,0,0}$s=n(\frac1p-1)$}%
}}}}
\put(1486,-1231){\makebox(0,0)[b]{\smash{{\SetFigFont{12}{14.4}{\familydefault}{\mddefault}{\updefault}{\color[rgb]{0,0,0}\eqref{6.10}}%
}}}}
\end{picture}%

%% file: fig_rho_5a.pdf_t
\begin{picture}(0,0)%
\includegraphics{fig_rho_5a.pdf}%
\end{picture}%
\setlength{\unitlength}{4144sp}%
\begingroup\makeatletter\ifx\SetFigFont\undefined%
\gdef\SetFigFont#1#2#3#4#5{%
  \reset@font\fontsize{#1}{#2pt}%
  \fontfamily{#3}\fontseries{#4}\fontshape{#5}%
  \selectfont}%
\fi\endgroup%
\begin{picture}(3129,2610)(799,-2548)
\put(1801,-1636){\makebox(0,0)[b]{\smash{{\SetFigFont{12}{14.4}{\familydefault}{\mddefault}{\updefault}{\color[rgb]{0,0,0}$1$}%
}}}}
\put(2251,-1636){\makebox(0,0)[b]{\smash{{\SetFigFont{12}{14.4}{\familydefault}{\mddefault}{\updefault}{\color[rgb]{0,0,0}$\frac{1}{|\vr|}$}%
}}}}
\put(1711,-781){\makebox(0,0)[rb]{\smash{{\SetFigFont{12}{14.4}{\familydefault}{\mddefault}{\updefault}{\color[rgb]{0,0,0}$s=\frac{|\vr|}{p}$}%
}}}}
\put(3781,-1636){\makebox(0,0)[lb]{\smash{{\SetFigFont{12}{14.4}{\familydefault}{\mddefault}{\updefault}{\color[rgb]{0,0,0}$\frac1p$}%
}}}}
\put(3601,-1636){\makebox(0,0)[b]{\smash{{\SetFigFont{12}{14.4}{\familydefault}{\mddefault}{\updefault}{\color[rgb]{0,0,0}$\frac{2}{|\vr|}$}%
}}}}
\put(856,-556){\makebox(0,0)[rb]{\smash{{\SetFigFont{12}{14.4}{\familydefault}{\mddefault}{\updefault}{\color[rgb]{0,0,0}$1$}%
}}}}
\put(856,-2356){\makebox(0,0)[rb]{\smash{{\SetFigFont{12}{14.4}{\familydefault}{\mddefault}{\updefault}{\color[rgb]{0,0,0}$-1$}%
}}}}
\put(856,-61){\makebox(0,0)[rb]{\smash{{\SetFigFont{12}{14.4}{\familydefault}{\mddefault}{\updefault}{\color[rgb]{0,0,0}$s$}%
}}}}
\put(2701,-1231){\makebox(0,0)[lb]{\smash{{\SetFigFont{12}{14.4}{\familydefault}{\mddefault}{\updefault}{\color[rgb]{0,0,0}$s=\frac{|\vr|}{p}-1$}%
}}}}
\end{picture}%

%% file: fig_rho_5b.pdf_t
\begin{picture}(0,0)%
\includegraphics{fig_rho_5b.pdf}%
\end{picture}%
\setlength{\unitlength}{4144sp}%
\begingroup\makeatletter\ifx\SetFigFont\undefined%
\gdef\SetFigFont#1#2#3#4#5{%
  \reset@font\fontsize{#1}{#2pt}%
  \fontfamily{#3}\fontseries{#4}\fontshape{#5}%
  \selectfont}%
\fi\endgroup%
\begin{picture}(2454,2610)(709,-2548)
\put(2206,-1636){\makebox(0,0)[b]{\smash{{\SetFigFont{12}{14.4}{\familydefault}{\mddefault}{\updefault}{\color[rgb]{0,0,0}$1\!+\!\frac{1}{|\vr|}$}%
}}}}
\put(1801,-1636){\makebox(0,0)[b]{\smash{{\SetFigFont{12}{14.4}{\familydefault}{\mddefault}{\updefault}{\color[rgb]{0,0,0}$1$}%
}}}}
\put(1351,-1996){\makebox(0,0)[lb]{\smash{{\SetFigFont{12}{14.4}{\familydefault}{\mddefault}{\updefault}{\color[rgb]{0,0,0}$s=\frac1p-1$}%
}}}}
\put(1486,-736){\makebox(0,0)[rb]{\smash{{\SetFigFont{12}{14.4}{\familydefault}{\mddefault}{\updefault}{\color[rgb]{0,0,0}$s=\frac1p$}%
}}}}
\put(1981,-1096){\makebox(0,0)[lb]{\smash{{\SetFigFont{12}{14.4}{\familydefault}{\mddefault}{\updefault}{\color[rgb]{0,0,0}$s=|\vr|(\frac1p-1)$}%
}}}}
\put(2926,-1636){\makebox(0,0)[lb]{\smash{{\SetFigFont{12}{14.4}{\familydefault}{\mddefault}{\updefault}{\color[rgb]{0,0,0}$\frac1p$}%
}}}}
\put(856,-61){\makebox(0,0)[rb]{\smash{{\SetFigFont{12}{14.4}{\familydefault}{\mddefault}{\updefault}{\color[rgb]{0,0,0}$s$}%
}}}}
\put(856,-556){\makebox(0,0)[rb]{\smash{{\SetFigFont{12}{14.4}{\familydefault}{\mddefault}{\updefault}{\color[rgb]{0,0,0}$1$}%
}}}}
\put(856,-2356){\makebox(0,0)[rb]{\smash{{\SetFigFont{12}{14.4}{\familydefault}{\mddefault}{\updefault}{\color[rgb]{0,0,0}$-1$}%
}}}}
\end{picture}%

%% file: fig_rho_5c.pdf_t
\begin{picture}(0,0)%
\includegraphics{fig_rho_5c.pdf}%
\end{picture}%
\setlength{\unitlength}{4144sp}%
\begingroup\makeatletter\ifx\SetFigFont\undefined%
\gdef\SetFigFont#1#2#3#4#5{%
  \reset@font\fontsize{#1}{#2pt}%
  \fontfamily{#3}\fontseries{#4}\fontshape{#5}%
  \selectfont}%
\fi\endgroup%
\begin{picture}(2229,2499)(709,-2548)
\put(2161,-1636){\makebox(0,0)[lb]{\smash{{\SetFigFont{12}{14.4}{\familydefault}{\mddefault}{\updefault}{\color[rgb]{0,0,0}$1\!+\!\frac{1}{|\vr|}$}%
}}}}
\put(2026,-1861){\makebox(0,0)[b]{\smash{{\SetFigFont{12}{14.4}{\familydefault}{\mddefault}{\updefault}{\color[rgb]{0,0,0}$1\!+\!\frac{1}{n}$}%
}}}}
\put(1801,-1636){\makebox(0,0)[b]{\smash{{\SetFigFont{12}{14.4}{\familydefault}{\mddefault}{\updefault}{\color[rgb]{0,0,0}$1$}%
}}}}
\put(856,-556){\makebox(0,0)[rb]{\smash{{\SetFigFont{12}{14.4}{\familydefault}{\mddefault}{\updefault}{\color[rgb]{0,0,0}$1$}%
}}}}
\put(856,-2356){\makebox(0,0)[rb]{\smash{{\SetFigFont{12}{14.4}{\familydefault}{\mddefault}{\updefault}{\color[rgb]{0,0,0}$-1$}%
}}}}
\put(856,-286){\makebox(0,0)[rb]{\smash{{\SetFigFont{12}{14.4}{\familydefault}{\mddefault}{\updefault}{\color[rgb]{0,0,0}$s$}%
}}}}
\put(1261,-1681){\makebox(0,0)[b]{\smash{{\SetFigFont{12}{14.4}{\familydefault}{\mddefault}{\updefault}{\color[rgb]{0,0,0}$R_n$}%
}}}}
\put(2206,-421){\makebox(0,0)[b]{\smash{{\SetFigFont{12}{14.4}{\familydefault}{\mddefault}{\updefault}{\color[rgb]{0,0,0}$R_{|\vr|}$}%
}}}}
\put(2836,-1636){\makebox(0,0)[lb]{\smash{{\SetFigFont{12}{14.4}{\familydefault}{\mddefault}{\updefault}{\color[rgb]{0,0,0}$\frac1p$}%
}}}}
\end{picture}%

%% file: fig_rho_6.pdf_t
\begin{picture}(0,0)%
\includegraphics{fig_rho_6.pdf}%
\end{picture}%
\setlength{\unitlength}{4144sp}%
\begingroup\makeatletter\ifx\SetFigFont\undefined%
\gdef\SetFigFont#1#2#3#4#5{%
  \reset@font\fontsize{#1}{#2pt}%
  \fontfamily{#3}\fontseries{#4}\fontshape{#5}%
  \selectfont}%
\fi\endgroup%
\begin{picture}(3132,2641)(796,-2600)
\put(811,-556){\makebox(0,0)[rb]{\smash{{\SetFigFont{12}{14.4}{\familydefault}{\mddefault}{\updefault}{\color[rgb]{0,0,0}$2$}%
}}}}
\put(811,-196){\makebox(0,0)[rb]{\smash{{\SetFigFont{12}{14.4}{\familydefault}{\mddefault}{\updefault}{\color[rgb]{0,0,0}$s$}%
}}}}
\put(2701,-1231){\makebox(0,0)[lb]{\smash{{\SetFigFont{12}{14.4}{\familydefault}{\mddefault}{\updefault}{\color[rgb]{0,0,0}$s=\frac{|\vr|}{p}$}%
}}}}
\put(2251,-2536){\makebox(0,0)[b]{\smash{{\SetFigFont{12}{14.4}{\familydefault}{\mddefault}{\updefault}{\color[rgb]{0,0,0}$\frac{1}{|\vr|}$}%
}}}}
\put(3781,-2536){\makebox(0,0)[lb]{\smash{{\SetFigFont{12}{14.4}{\familydefault}{\mddefault}{\updefault}{\color[rgb]{0,0,0}$\frac1p$}%
}}}}
\put(3601,-2536){\makebox(0,0)[b]{\smash{{\SetFigFont{12}{14.4}{\familydefault}{\mddefault}{\updefault}{\color[rgb]{0,0,0}$\frac{2}{|\vr|}$}%
}}}}
\put(1801,-2536){\makebox(0,0)[b]{\smash{{\SetFigFont{12}{14.4}{\familydefault}{\mddefault}{\updefault}{\color[rgb]{0,0,0}$1$}%
}}}}
\put(811,-1456){\makebox(0,0)[rb]{\smash{{\SetFigFont{12}{14.4}{\familydefault}{\mddefault}{\updefault}{\color[rgb]{0,0,0}$1$}%
}}}}
\put(1756,-736){\makebox(0,0)[rb]{\smash{{\SetFigFont{12}{14.4}{\familydefault}{\mddefault}{\updefault}{\color[rgb]{0,0,0}$s=1+\frac{|\vr|}{p}$}%
}}}}
\end{picture}%

%% file: fig_emb_dom.pdf_t
\begin{picture}(0,0)%
\includegraphics{fig_emb_dom.pdf}%
\end{picture}%
\setlength{\unitlength}{4144sp}%
\begingroup\makeatletter\ifx\SetFigFont\undefined%
\gdef\SetFigFont#1#2#3#4#5{%
  \reset@font\fontsize{#1}{#2pt}%
  \fontfamily{#3}\fontseries{#4}\fontshape{#5}%
  \selectfont}%
\fi\endgroup%
\begin{picture}(2859,3039)(1114,-4123)
\put(2341,-2176){\makebox(0,0)[rb]{\smash{{\SetFigFont{10}{12.0}{\familydefault}{\mddefault}{\updefault}{\color[rgb]{0,0,0}$s_1-s_2=|\vr|(\frac{1}{p_1}-\frac{1}{p_2})$}%
}}}}
\put(1306,-3256){\makebox(0,0)[rb]{\smash{{\SetFigFont{10}{12.0}{\familydefault}{\mddefault}{\updefault}{\color[rgb]{0,0,0}$s_1-\frac{|\vr|}{p_1}$}%
}}}}
\put(1306,-1816){\makebox(0,0)[rb]{\smash{{\SetFigFont{10}{12.0}{\familydefault}{\mddefault}{\updefault}{\color[rgb]{0,0,0}$s_1$}%
}}}}
\put(1306,-1231){\makebox(0,0)[rb]{\smash{{\SetFigFont{10}{12.0}{\familydefault}{\mddefault}{\updefault}{\color[rgb]{0,0,0}$s$}%
}}}}
\put(3826,-3841){\makebox(0,0)[b]{\smash{{\SetFigFont{10}{12.0}{\familydefault}{\mddefault}{\updefault}{\color[rgb]{0,0,0}$\frac1p$}%
}}}}
\put(2521,-1636){\makebox(0,0)[lb]{\smash{{\SetFigFont{10}{12.0}{\familydefault}{\mddefault}{\updefault}{\color[rgb]{0,0,0}$\rhoAe(\Omega)$}%
}}}}
\put(3601,-1366){\makebox(0,0)[rb]{\smash{{\SetFigFont{10}{12.0}{\familydefault}{\mddefault}{\updefault}{\color[rgb]{0,0,0}$s=\frac{|\vr|}{p}$}%
}}}}
\put(2791,-3841){\makebox(0,0)[b]{\smash{{\SetFigFont{10}{12.0}{\familydefault}{\mddefault}{\updefault}{\color[rgb]{0,0,0}$\frac{1}{p_1}$}%
}}}}
\put(3196,-2716){\makebox(0,0)[b]{\smash{{\SetFigFont{10}{12.0}{\familydefault}{\mddefault}{\updefault}{\color[rgb]{0,0,0}compact}%
}}}}
\put(2386,-2941){\makebox(0,0)[lb]{\smash{{\SetFigFont{10}{12.0}{\familydefault}{\mddefault}{\updefault}{\color[rgb]{1,0,0}$\id_\Omega$}%
}}}}
\put(1936,-3526){\makebox(0,0)[lb]{\smash{{\SetFigFont{10}{12.0}{\familydefault}{\mddefault}{\updefault}{\color[rgb]{0,0,0}$\rhoAz(\Omega)$}%
}}}}
\end{picture}%

%% file: fig_rho_env.pdf_t
\begin{picture}(0,0)%
\includegraphics{fig_rho_env.pdf}%
\end{picture}%
\setlength{\unitlength}{4144sp}%
\begingroup\makeatletter\ifx\SetFigFont\undefined%
\gdef\SetFigFont#1#2#3#4#5{%
  \reset@font\fontsize{#1}{#2pt}%
  \fontfamily{#3}\fontseries{#4}\fontshape{#5}%
  \selectfont}%
\fi\endgroup%
\begin{picture}(3579,2196)(664,-3055)
\put(2251,-2986){\makebox(0,0)[b]{\smash{{\SetFigFont{10}{12.0}{\familydefault}{\mddefault}{\updefault}$1$}}}}
\put(4141,-2986){\makebox(0,0)[b]{\smash{{\SetFigFont{12}{14.4}{\familydefault}{\mddefault}{\updefault}$\frac1p$}}}}
\put(811,-1051){\makebox(0,0)[rb]{\smash{{\SetFigFont{12}{14.4}{\familydefault}{\mddefault}{\updefault}$s$}}}}
\put(3151,-1996){\makebox(0,0)[lb]{\smash{{\SetFigFont{12}{14.4}{\familydefault}{\mddefault}{\updefault}$s=\sigma_p^{|\vr|}$}}}}
\put(2296,-1321){\makebox(0,0)[rb]{\smash{{\SetFigFont{12}{14.4}{\familydefault}{\mddefault}{\updefault}$s=\frac{|\vr|}{p}$}}}}
\put(811,-1591){\makebox(0,0)[rb]{\smash{{\SetFigFont{12}{14.4}{\familydefault}{\mddefault}{\updefault}$s$}}}}
\put(2611,-2986){\makebox(0,0)[b]{\smash{{\SetFigFont{12}{14.4}{\familydefault}{\mddefault}{\updefault}$\frac1p$}}}}
\put(1351,-2986){\makebox(0,0)[b]{\smash{{\SetFigFont{12}{14.4}{\familydefault}{\mddefault}{\updefault}$\frac1r$}}}}
\put(2656,-1546){\makebox(0,0)[lb]{\smash{{\SetFigFont{12}{14.4}{\familydefault}{\mddefault}{\updefault}$\rhoAs(\Omega)$}}}}
\put(1576,-2671){\makebox(0,0)[lb]{\smash{{\SetFigFont{12}{14.4}{\familydefault}{\mddefault}{\updefault}$\La^\vr_r(\Omega)$}}}}
\put(2836,-1051){\makebox(0,0)[lb]{\smash{{\SetFigFont{12}{14.4}{\familydefault}{\mddefault}{\updefault}$s=|\vr|(\frac1p\!-\!\frac1r)$}}}}
\end{picture}%

%% file: fig_emb_rn.pdf_t
\begin{picture}(0,0)%
\includegraphics{fig_emb_rn.pdf}%
\end{picture}%
\setlength{\unitlength}{4144sp}%
\begingroup\makeatletter\ifx\SetFigFont\undefined%
\gdef\SetFigFont#1#2#3#4#5{%
  \reset@font\fontsize{#1}{#2pt}%
  \fontfamily{#3}\fontseries{#4}\fontshape{#5}%
  \selectfont}%
\fi\endgroup%
\begin{picture}(2859,2949)(1114,-4123)
\put(1306,-3256){\makebox(0,0)[rb]{\smash{{\SetFigFont{10}{12.0}{\familydefault}{\mddefault}{\updefault}{\color[rgb]{0,0,0}$s_1-\frac{|\vr|}{p_1}$}%
}}}}
\put(1306,-1816){\makebox(0,0)[rb]{\smash{{\SetFigFont{10}{12.0}{\familydefault}{\mddefault}{\updefault}{\color[rgb]{0,0,0}$s_1$}%
}}}}
\put(3826,-3841){\makebox(0,0)[b]{\smash{{\SetFigFont{10}{12.0}{\familydefault}{\mddefault}{\updefault}{\color[rgb]{0,0,0}$\frac1p$}%
}}}}
\put(2521,-1636){\makebox(0,0)[lb]{\smash{{\SetFigFont{10}{12.0}{\familydefault}{\mddefault}{\updefault}{\color[rgb]{0,0,0}$\rhoAe(\rn)$}%
}}}}
\put(2791,-3841){\makebox(0,0)[lb]{\smash{{\SetFigFont{10}{12.0}{\familydefault}{\mddefault}{\updefault}{\color[rgb]{0,0,0}$\frac{1}{p_1}$}%
}}}}
\put(3016,-2131){\makebox(0,0)[lb]{\smash{{\SetFigFont{10}{12.0}{\familydefault}{\mddefault}{\updefault}{\color[rgb]{0,0,0}$s=\frac{|\vr|}{p}$}%
}}}}
\put(1306,-1411){\makebox(0,0)[rb]{\smash{{\SetFigFont{10}{12.0}{\familydefault}{\mddefault}{\updefault}{\color[rgb]{0,0,0}$s$}%
}}}}
\put(1891,-3526){\makebox(0,0)[lb]{\smash{{\SetFigFont{10}{12.0}{\familydefault}{\mddefault}{\updefault}{\color[rgb]{0,0,0}$\rhoAz(\rn)$}%
}}}}
\put(2386,-2941){\makebox(0,0)[lb]{\smash{{\SetFigFont{10}{12.0}{\familydefault}{\mddefault}{\updefault}{\color[rgb]{1,0,0}$\id_{\rn}$}%
}}}}
\end{picture}%

%% file: fig_cross_1.pdf_t
\begin{picture}(0,0)%
\includegraphics{fig_cross_1.pdf}%
\end{picture}%
\setlength{\unitlength}{4144sp}%
\begingroup\makeatletter\ifx\SetFigFont\undefined%
\gdef\SetFigFont#1#2#3#4#5{%
  \reset@font\fontsize{#1}{#2pt}%
  \fontfamily{#3}\fontseries{#4}\fontshape{#5}%
  \selectfont}%
\fi\endgroup%
\begin{picture}(4389,2904)(1114,-4123)
\put(1306,-1816){\makebox(0,0)[rb]{\smash{{\SetFigFont{10}{12.0}{\familydefault}{\mddefault}{\updefault}{\color[rgb]{0,0,0}$s_1$}%
}}}}
\put(2521,-1636){\makebox(0,0)[lb]{\smash{{\SetFigFont{10}{12.0}{\familydefault}{\mddefault}{\updefault}{\color[rgb]{0,0,0}$\rhoeAe(\rn)$}%
}}}}
\put(1306,-1411){\makebox(0,0)[rb]{\smash{{\SetFigFont{10}{12.0}{\familydefault}{\mddefault}{\updefault}{\color[rgb]{0,0,0}$s$}%
}}}}
\put(2791,-3841){\makebox(0,0)[b]{\smash{{\SetFigFont{10}{12.0}{\familydefault}{\mddefault}{\updefault}{\color[rgb]{0,0,0}$\frac{1}{p_1}$}%
}}}}
\put(3196,-3841){\makebox(0,0)[b]{\smash{{\SetFigFont{10}{12.0}{\familydefault}{\mddefault}{\updefault}{\color[rgb]{0,0,0}$\frac{1}{p_2}$}%
}}}}
\put(1306,-2851){\makebox(0,0)[rb]{\smash{{\SetFigFont{10}{12.0}{\familydefault}{\mddefault}{\updefault}{\color[rgb]{0,0,0}$s_2$}%
}}}}
\put(5401,-3841){\makebox(0,0)[b]{\smash{{\SetFigFont{10}{12.0}{\familydefault}{\mddefault}{\updefault}{\color[rgb]{0,0,0}$\frac1p$}%
}}}}
\put(1306,-3256){\makebox(0,0)[rb]{\smash{{\SetFigFont{10}{12.0}{\familydefault}{\mddefault}{\updefault}{\color[rgb]{0,0,0}$\sigma_1=s_1\!-\!\frac{|\vr_1|}{p_1}$}%
}}}}
\put(2386,-2086){\makebox(0,0)[rb]{\smash{{\SetFigFont{10}{12.0}{\familydefault}{\mddefault}{\updefault}{\color[rgb]{0,0,0}$s=\frac{|\vr_1|}{p}+\sigma_1$}%
}}}}
\put(4636,-1501){\makebox(0,0)[rb]{\smash{{\SetFigFont{10}{12.0}{\familydefault}{\mddefault}{\updefault}{\color[rgb]{0,0,0}$s=\frac{|\vr_2|}{p}+\sigma_1$}%
}}}}
\put(4231,-3841){\makebox(0,0)[lb]{\smash{{\SetFigFont{10}{12.0}{\familydefault}{\mddefault}{\updefault}{\color[rgb]{0,0,0}$\frac{1}{p_\vr}\!=\!\frac{|\vr_1|}{|\vr_2|} \frac{1}{p_1}$}%
}}}}
\put(3151,-2536){\makebox(0,0)[lb]{\smash{{\SetFigFont{10}{12.0}{\familydefault}{\mddefault}{\updefault}{\color[rgb]{1,0,0}$\id_{\rn}$}%
}}}}
\put(3241,-2986){\makebox(0,0)[lb]{\smash{{\SetFigFont{10}{12.0}{\familydefault}{\mddefault}{\updefault}{\color[rgb]{0,0,0}$\rhozAz(\rn)$}%
}}}}
\end{picture}%

%% file: fig_cross_2a.pdf_t
\begin{picture}(0,0)%
\includegraphics{fig_cross_2a.pdf}%
\end{picture}%
\setlength{\unitlength}{4144sp}%
\begingroup\makeatletter\ifx\SetFigFont\undefined%
\gdef\SetFigFont#1#2#3#4#5{%
  \reset@font\fontsize{#1}{#2pt}%
  \fontfamily{#3}\fontseries{#4}\fontshape{#5}%
  \selectfont}%
\fi\endgroup%
\begin{picture}(3399,3039)(1114,-4123)
\put(2341,-2176){\makebox(0,0)[rb]{\smash{{\SetFigFont{10}{12.0}{\familydefault}{\mddefault}{\updefault}{\color[rgb]{0,0,0}$s=\frac{|\vr_1|}{p} + \sigma_1$}%
}}}}
\put(1306,-1816){\makebox(0,0)[rb]{\smash{{\SetFigFont{10}{12.0}{\familydefault}{\mddefault}{\updefault}{\color[rgb]{0,0,0}$s_1$}%
}}}}
\put(1306,-1231){\makebox(0,0)[rb]{\smash{{\SetFigFont{10}{12.0}{\familydefault}{\mddefault}{\updefault}{\color[rgb]{0,0,0}$s$}%
}}}}
\put(2521,-1636){\makebox(0,0)[lb]{\smash{{\SetFigFont{10}{12.0}{\familydefault}{\mddefault}{\updefault}{\color[rgb]{0,0,0}$\rhoeAe(\Omega)$}%
}}}}
\put(2791,-3841){\makebox(0,0)[b]{\smash{{\SetFigFont{10}{12.0}{\familydefault}{\mddefault}{\updefault}{\color[rgb]{0,0,0}$\frac{1}{p_1}$}%
}}}}
\put(4276,-3841){\makebox(0,0)[b]{\smash{{\SetFigFont{10}{12.0}{\familydefault}{\mddefault}{\updefault}{\color[rgb]{0,0,0}$\frac1p$}%
}}}}
\put(1306,-3211){\makebox(0,0)[rb]{\smash{{\SetFigFont{10}{12.0}{\familydefault}{\mddefault}{\updefault}{\color[rgb]{0,0,0}$\sigma_1=s_1-\frac{|\vr_1|}{p_1}$}%
}}}}
\put(1306,-3436){\makebox(0,0)[rb]{\smash{{\SetFigFont{10}{12.0}{\familydefault}{\mddefault}{\updefault}{\color[rgb]{0,0,0}$s_2$}%
}}}}
\put(3736,-3841){\makebox(0,0)[b]{\smash{{\SetFigFont{10}{12.0}{\familydefault}{\mddefault}{\updefault}{\color[rgb]{0,0,0}$\frac{1}{p_2}$}%
}}}}
\put(3286,-2581){\makebox(0,0)[lb]{\smash{{\SetFigFont{10}{12.0}{\familydefault}{\mddefault}{\updefault}{\color[rgb]{1,0,0}$\id_\Omega$}%
}}}}
\put(3556,-3256){\makebox(0,0)[rb]{\smash{{\SetFigFont{10}{12.0}{\familydefault}{\mddefault}{\updefault}{\color[rgb]{0,0,0}$\rhozAz(\Omega)$}%
}}}}
\end{picture}%

%% file: fig_cross_2b.pdf_t
\begin{picture}(0,0)%
\includegraphics{fig_cross_2b.pdf}%
\end{picture}%
\setlength{\unitlength}{4144sp}%
\begingroup\makeatletter\ifx\SetFigFont\undefined%
\gdef\SetFigFont#1#2#3#4#5{%
  \reset@font\fontsize{#1}{#2pt}%
  \fontfamily{#3}\fontseries{#4}\fontshape{#5}%
  \selectfont}%
\fi\endgroup%
\begin{picture}(4389,2904)(1114,-4123)
\put(1306,-1816){\makebox(0,0)[rb]{\smash{{\SetFigFont{10}{12.0}{\familydefault}{\mddefault}{\updefault}{\color[rgb]{0,0,0}$s_1$}%
}}}}
\put(2521,-1636){\makebox(0,0)[lb]{\smash{{\SetFigFont{10}{12.0}{\familydefault}{\mddefault}{\updefault}{\color[rgb]{0,0,0}$\rhoeAe(\Omega)$}%
}}}}
\put(1306,-1411){\makebox(0,0)[rb]{\smash{{\SetFigFont{10}{12.0}{\familydefault}{\mddefault}{\updefault}{\color[rgb]{0,0,0}$s$}%
}}}}
\put(2791,-3841){\makebox(0,0)[b]{\smash{{\SetFigFont{10}{12.0}{\familydefault}{\mddefault}{\updefault}{\color[rgb]{0,0,0}$\frac{1}{p_1}$}%
}}}}
\put(1306,-2851){\makebox(0,0)[rb]{\smash{{\SetFigFont{10}{12.0}{\familydefault}{\mddefault}{\updefault}{\color[rgb]{0,0,0}$s_2$}%
}}}}
\put(5401,-3841){\makebox(0,0)[b]{\smash{{\SetFigFont{10}{12.0}{\familydefault}{\mddefault}{\updefault}{\color[rgb]{0,0,0}$\frac1p$}%
}}}}
\put(1306,-3256){\makebox(0,0)[rb]{\smash{{\SetFigFont{10}{12.0}{\familydefault}{\mddefault}{\updefault}{\color[rgb]{0,0,0}$\sigma_1=s_1\!-\!\frac{|\vr_1|}{p_1}$}%
}}}}
\put(2386,-2086){\makebox(0,0)[rb]{\smash{{\SetFigFont{10}{12.0}{\familydefault}{\mddefault}{\updefault}{\color[rgb]{0,0,0}$s=\frac{|\vr_1|}{p}+\sigma_1$}%
}}}}
\put(4636,-1501){\makebox(0,0)[rb]{\smash{{\SetFigFont{10}{12.0}{\familydefault}{\mddefault}{\updefault}{\color[rgb]{0,0,0}$s=\frac{|\vr_2|}{p}+\sigma_1$}%
}}}}
\put(4546,-3841){\makebox(0,0)[b]{\smash{{\SetFigFont{10}{12.0}{\familydefault}{\mddefault}{\updefault}{\color[rgb]{0,0,0}$\frac{1}{p_2}$}%
}}}}
\put(3961,-3841){\makebox(0,0)[b]{\smash{{\SetFigFont{10}{12.0}{\familydefault}{\mddefault}{\updefault}{\color[rgb]{0,0,0}$\frac{1}{p_\vr}\!=\!\frac{|\vr_1|}{|\vr_2|} \frac{1}{p_1}$}%
}}}}
\put(3826,-2536){\makebox(0,0)[rb]{\smash{{\SetFigFont{10}{12.0}{\familydefault}{\mddefault}{\updefault}{\color[rgb]{1,0,0}$\id_{\Omega}$}%
}}}}
\put(4501,-2986){\makebox(0,0)[rb]{\smash{{\SetFigFont{10}{12.0}{\familydefault}{\mddefault}{\updefault}{\color[rgb]{0,0,0}$\rhozAz(\Omega)$}%
}}}}
\end{picture}%

%% file: morrey4.bbl
\newcommand{\etalchar}[1]{$^{#1}$}
\begin{thebibliography}{SDFH20b}

\bibitem[Bar02a]{baraka-1}
A.~El Baraka.
\newblock An embedding theorem for {C}ampanato spaces.
\newblock {\em Electron. J. Differential Equations}, pages No. 66, 17 pp.
  (electronic), 2002.

\bibitem[Bar02b]{baraka-2}
A.~El Baraka.
\newblock Function spaces of {BMO} and {C}ampanato type.
\newblock In {\em Proceedings of the 2002 {F}ez {C}onference on {P}artial
  {D}ifferential {E}quations}, volume~9 of {\em Electron. J. Differ. Equ.
  Conf.}, pages 109--115 (electronic). Southwest Texas State Univ., San Marcos,
  TX, 2002.

\bibitem[Bar06]{baraka-3}
A.~El Baraka.
\newblock Littlewood-{P}aley characterization for {C}ampanato spaces.
\newblock {\em J. Funct. Spaces Appl.}, 4(2):193--220, 2006.

\bibitem[BL76]{BL}
J.~Bergh and J.~L{\"o}fstr{\"o}m.
\newblock {\em Interpolation spaces}.
\newblock Springer, Berlin, 1976.

\bibitem[BS88]{BS}
C.~Bennett and R.~Sharpley.
\newblock {\em Interpolation of operators}.
\newblock Academic Press, Boston, 1988.

\bibitem[CS90]{CS}
B.~Carl and I.~Stephani.
\newblock {\em Entropy, compactness and the approximation of operators}.
\newblock Cambridge Univ. Press, Cambridge, 1990.

\bibitem[Dau92]{Dau92}
I.~Daubechies.
\newblock {\em Ten lectures on wavelets}, volume~61 of {\em CBMS-NSF Regional
  Conference Series in Applied Mathematics}.
\newblock Society for Industrial and Applied Mathematics (SIAM), Philadelphia,
  PA, 1992.

\bibitem[EE87]{EE}
D.E. Edmunds and W.D. Evans.
\newblock {\em Spectral theory and differential operators}.
\newblock Clarendon Press, Oxford, 1987.

\bibitem[ET96]{ET96}
D.E. Edmunds and H.~Triebel.
\newblock {\em Function spaces, entropy numbers, differential operators}.
\newblock Cambridge Univ. Press, Cambridge, 1996.

\bibitem[FJ90]{FrJ90}
M.~Frazier and B.~Jawerth.
\newblock A discrete transform and decomposition of distribution spaces.
\newblock {\em J. Funct. Anal.}, 93(1):34--170, 1990.

\bibitem[FP19]{FP}
L.C.F. Ferreira and M.~Postigo.
\newblock Global well-posedness and asymptotic behavior in {B}esov-{M}orrey
  spaces for chemotaxis-{N}avier-{S}tokes fluids.
\newblock {\em J. Math. Phys.}, 60(6):061502, 19, 2019.

\bibitem[Fra86]{Franke86}
J.~Franke.
\newblock On the spaces {$F^s_{p,q}$} of {T}riebel-{L}izorkin type: Pointwise
  multipliers and spaces on domains.
\newblock {\em Math. Nachr.}, 125:29--68, 1986.

\bibitem[GHS21a]{GHS21}
H.F. Gon\c{c}alves, D.D. Haroske, and L.~Skrzypczak.
\newblock {Compact embeddings of Besov-type and Triebel-Lizorkin-type spaces on
  bounded domains}.
\newblock {\em Rev. Mat. Complut.}, 34:761--795, 2021.

\bibitem[GHS21b]{GHS21b}
H.F. Gon\c{c}alves, D.D. Haroske, and L.~Skrzypczak.
\newblock {Limiting embeddings of Besov-type and Triebel-Lizorkin-type spaces
  on bounded domains, and extension operator}.
\newblock Preprint, arXiv:2109.12015, 2021.

\bibitem[GP77]{GuP77}
J.~Gustavsson and J.~Peetre.
\newblock Interpolation of {O}rlicz spaces.
\newblock {\em Studia Math.}, 60(1):33--59, 1977.

\bibitem[Har02]{HaHabil}
D.D. Haroske.
\newblock Limiting embeddings, entropy numbers and envelopes in function
  spaces.
\newblock Habi\-li\-ta\-tions\-schrift, Friedrich-Schiller-Universit\"at Jena,
  Germany, 2002.

\bibitem[Har07]{Ha-crc}
D.D. Haroske.
\newblock {\em Envelopes and sharp embeddings of function spaces}, volume 437
  of {\em Chapman \& Hall/CRC Research Notes in Mathematics}.
\newblock Chapman \& Hall/CRC, Boca Raton, FL, 2007.

\bibitem[HM16]{Ha-SM-3}
D.D. Haroske and S.D. Moura.
\newblock {Some specific unboundedness property in smoothness Morrey spaces.
  The non-existence of growth envelopes in the subcritical case}.
\newblock {\em Acta Math. Sinica}, 32(2):137--152, 2016.

\bibitem[HMS16]{HMS16}
D.D. Haroske, S.D. Moura, and L.~Skrzypczak.
\newblock {Smoothness Morrey spaces of regular distributions, and some
  unboundedness property}.
\newblock {\em Nonlinear Anal.}, 139:218--244, 2016.

\bibitem[HMS20]{HMS20}
D.D. Haroske, S.D. Moura, and L.~Skrzypczak.
\newblock {Some embeddings of Morrey spaces with critical smoothness}.
\newblock {\em J. Fourier Anal. Appl.}, 26(3):50, 2020.

\bibitem[HMSS17]{HMSS-morrey}
D.D. Haroske, S.D. Moura, C.~Schneider, and L.~Skrzypczak.
\newblock {Unboundedness properties of smoothness Morrey spaces of regular
  distributions on domains}.
\newblock {\em Sci. China Math.}, 60(12):2349--2376, 2017.

\bibitem[HNS17]{HNaS}
D.I. Hakim, S.~Nakamura, and Y.~Sawano.
\newblock Complex interpolation of smoothness {M}orrey subspaces.
\newblock {\em Constr. Approx.}, 46(3):489--563, 2017.

\bibitem[HNS19]{HNoS}
D.I. Hakim, T.~Nogayama, and Y.~Sawano.
\newblock Complex interpolation of smoothness {T}riebel-{L}izorkin-{M}orrey
  spaces.
\newblock {\em Math. J. Okayama Univ.}, 61:99--128, 2019.

\bibitem[Hov21a]{Hov20b}
M.~Hovemann.
\newblock Besov-{M}orrey spaces and differences.
\newblock {\em Math. Rep. (Bucur.)}, 23(73)(1-2):175--192, 2021.

\bibitem[Hov21b]{Hov20a}
M.~Hovemann.
\newblock Truncation in {B}esov-{M}orrey and {T}riebel-{L}izorkin-{M}orrey
  spaces.
\newblock {\em Nonlinear Anal.}, 204:Paper No. 112239, 30, 2021.

\bibitem[HS12]{HaS12}
D.D. Haroske and L.~Skrzypczak.
\newblock Continuous embeddings of {B}esov-{M}orrey function spaces.
\newblock {\em Acta Math. Sinica}, 28(7):1307--1328, 2012.

\bibitem[HS13]{HaS13}
D.D. Haroske and L.~Skrzypczak.
\newblock Embeddings of {B}esov-{M}orrey spaces on bounded domains.
\newblock {\em Studia Math.}, 218:119--144, 2013.

\bibitem[HS14]{HaS14}
D.D. Haroske and L.~Skrzypczak.
\newblock {On Sobolev and Franke-Jawerth embeddings of smoothness Morrey
  spaces}.
\newblock {\em Rev. Mat. Complut.}, 27(2):541--573, 2014.

\bibitem[HS17a]{HaSa-2}
D.I. Hakim and Y.~Sawano.
\newblock Calder\'{o}n's first and second complex interpolations of closed
  subspaces of {M}orrey spaces.
\newblock {\em J. Fourier Anal. Appl.}, 23(5):1195--1226, 2017.

\bibitem[HS17b]{HaSa-3}
D.I. Hakim and Y.~Sawano.
\newblock Complex interpolation of {M}orrey spaces.
\newblock In {\em Function spaces and inequalities}, volume 206 of {\em
  Springer Proc. Math. Stat.}, pages 85--115. Springer, Singapore, 2017.

\bibitem[HS19]{HaS19}
D.D. Haroske and L.~Skrzypczak.
\newblock {Some quantitative result on compact embeddings in smoothness Morrey
  spaces on bounded domains; an approach via interpolation}.
\newblock In {\em Function spaces XII}, volume 119 of {\em Banach Center
  Publ.}, pages 181--191, Warsaw, 2019. Polish Acad. Sci., Warsaw.

\bibitem[HS20a]{HaS20}
D.D. Haroske and L.~Skrzypczak.
\newblock {Entropy numbers of compact embeddings of smoothness Morrey spaces on
  bounded domains}.
\newblock {\em J. Approx. Theory}, 256:105424, 2020.

\bibitem[HS20b]{HoS20}
M.~Hovemann and W.~Sickel.
\newblock Besov-type spaces and differences.
\newblock {\em Eurasian Math. J.}, 11(1):25--56, 2020.

\bibitem[HSS18]{HSS-morrey}
D.D. Haroske, C.~Schneider, and L.~Skrzypczak.
\newblock {Morrey spaces on domains: Different approaches and growth
  envelopes}.
\newblock {\em J. Geom. Anal.}, 28(2):817--841, 2018.

\bibitem[Jaw77]{Jaw77}
B.~Jawerth.
\newblock Some observations on {B}esov and {L}izorkin-{T}riebel spaces.
\newblock {\em Math. Scand.}, 40(1):94--104, 1977.

\bibitem[K{\"o}n86]{Konig}
H.~K{\"o}nig.
\newblock {\em Eigenvalue distribution of compact operators}.
\newblock Birkh\"auser, Basel, 1986.

\bibitem[KY95]{KY}
H.~Kozono and M.~Yamazaki.
\newblock The stability of small stationary solutions in {M}orrey spaces of the
  {N}avier-{S}tokes equation.
\newblock {\em Indiana Univ. Math. J.}, 44(4):1307--1336, 1995.

\bibitem[LR07]{LMR-5}
P.G. Lemari\'{e}-Rieusset.
\newblock The {N}avier-{S}tokes equations in the critical {M}orrey-{C}ampanato
  space.
\newblock {\em Rev. Mat. Iberoam.}, 23(3):897--930, 2007.

\bibitem[LR12]{LMR-4}
P.~G. Lemari\'{e}-Rieusset.
\newblock The role of {M}orrey spaces in the study of {N}avier-{S}tokes and
  {E}uler equations.
\newblock {\em Eurasian Math. J.}, 3(3):62--93, 2012.

\bibitem[LR13]{Lem13}
P.G. Lemari\'{e}-Rieusset.
\newblock Multipliers and {M}orrey spaces.
\newblock {\em Potential Anal.}, 38(3):741--752, 2013.

\bibitem[LR14]{Lem14}
P.G. Lemari\'{e}-Rieusset.
\newblock Erratum to: {M}ultipliers and {M}orrey spaces [mr3034598].
\newblock {\em Potential Anal.}, 41(4):1359--1362, 2014.

\bibitem[LYY{\etalchar{+}}13]{LYYSU}
Y.~Liang, D.~Yang, W.~Yuan, Y.~Sawano, and T.~Ullrich.
\newblock {A new framework for generalized Besov-type and
  Triebel–Lizorkin-type spaces}.
\newblock {\em Dissertationes Math.}, 489:114 pp., 2013.

\bibitem[Mal98]{Mal98}
S.~Mallat.
\newblock {\em A wavelet tour of signal processing}.
\newblock Academic Press, Inc., San Diego, CA, 1998.

\bibitem[Mar87]{Mar87}
J.~Marschall.
\newblock Some remarks on {T}riebel spaces.
\newblock {\em Studia Math.}, 87:79--92, 1987.

\bibitem[Maz03a]{mazzu-2}
A.L. Mazzucato.
\newblock Besov-{M}orrey spaces: function space theory and applications to
  non-linear {PDE}.
\newblock {\em Trans. Amer. Math. Soc.}, 355(4):1297--1364 (electronic), 2003.

\bibitem[Maz03b]{Maz03}
A.L. Mazzucato.
\newblock Decomposition of {B}esov-{M}orrey spaces.
\newblock In {\em Harmonic analysis at {M}ount {H}olyoke ({S}outh {H}adley,
  {MA}, 2001)}, volume 320 of {\em Contemp. Math.}, pages 279--294. Amer. Math.
  Soc., Providence, RI, 2003.

\bibitem[Mey92]{Mey92}
Y.~Meyer.
\newblock {\em Wavelets and operators}, volume~37 of {\em Cambridge Studies in
  Advanced Mathematics}.
\newblock Cambridge Univ. Press, Cambridge, 1992.

\bibitem[MNS19]{MNS19}
S.D. Moura, J.S. Neves, and C.~Schneider.
\newblock Traces and extensions of generalized smoothness {M}orrey spaces on
  domains.
\newblock {\em Nonlinear Anal.}, 181:311--339, 2019.

\bibitem[{Mor}38]{Mor}
Ch.~B. {Morrey, Jr.}
\newblock On the solutions of quasi-linear elliptic partial differential
  equations.
\newblock {\em Trans. Amer. Math. Soc.}, 43(1):126--166, 1938.

\bibitem[MS19]{MaSa}
M.~Masty{\l}o and Y.~Sawano.
\newblock Complex interpolation and {C}alder\'{o}n-{M}ityagin couples of
  {M}orrey spaces.
\newblock {\em Anal. PDE}, 12(7):1711--1740, 2019.

\bibitem[Net87]{Ne-10}
Yu. Netrusov.
\newblock Some imbedding theorems for spaces of {B}esov-{M}orrey type.
\newblock {\em J. Sov. Math.}, 36:270--276, 1987.
\newblock Transl. from Zap. Nauchn. Sem. Leningrad. Otdel. Mat. Inst. Steklov.
  (LOMI) 139, 139--147 (1984).

\bibitem[Pee69]{Pee}
J.~Peetre.
\newblock On the theory of {${\mathcal L}_{p,\lambda }$} spaces.
\newblock {\em J. Funct. Anal.}, 4:71--87, 1969.

\bibitem[Pie87]{Pie-s}
A.~Pietsch.
\newblock {\em Eigenvalues and $s$-numbers}.
\newblock Akad. Verlagsgesellschaft Geest \& Portig, Leipzig, 1987.

\bibitem[Ros13]{MR-1}
M.~Rosenthal.
\newblock {Local means, wavelet bases and wavelet isomorphisms in Besov-Morrey
  and Triebel-Lizorkin-Morrey spaces}.
\newblock {\em Math. Nachr.}, 286(1):59--87, 2013.

\bibitem[Saw08]{Saw08}
Y.~Sawano.
\newblock Wavelet characterizations of {B}esov-{M}orrey and
  {T}riebel-{L}izorkin-{M}orrey spaces.
\newblock {\em Funct. Approx. Comment. Math.}, 38:93--107, 2008.

\bibitem[Saw09]{Saw09}
Y.~Sawano.
\newblock A note on {B}esov-{M}orrey spaces and {T}riebel-{L}izorkin-{M}orrey
  spaces.
\newblock {\em Acta Math. Sin. (Engl. Ser.)}, 25(8):1223--1242, 2009.

\bibitem[Saw10]{Saw10}
Y.~Sawano.
\newblock {Besov-Morrey spaces and Triebel-Lizorkin-Morrey spaces on domains}.
\newblock {\em Math. Nachr.}, 283:1456--1487, 2010.

\bibitem[Saw18]{Saw18}
Y.~Sawano.
\newblock {\em Theory of {B}esov spaces}, volume~56 of {\em Developments in
  Mathematics}.
\newblock Springer, Singapore, 2018.

\bibitem[SDFH20a]{FHS-MS-1}
Y.~Sawano, G.~Di~Fazio, and D.I. Hakim.
\newblock {\em {M}orrey spaces. {Introduction and Applications to Integral
  Operators and PDE's}. Vol. {I}}.
\newblock Monographs and Research Notes in Mathematics. Chapman \& Hall CRC
  Press, Boca Raton, FL, 2020.

\bibitem[SDFH20b]{FHS-MS-2}
Y.~Sawano, G.~Di~Fazio, and D.I. Hakim.
\newblock {\em {M}orrey spaces. {Introduction and Applications to Integral
  Operators and PDE's}. Vol. {II}}.
\newblock Monographs and Research Notes in Mathematics. Chapman \& Hall CRC
  Press, Boca Raton, FL, 2020.

\bibitem[Sic12]{Sic12}
W.~Sickel.
\newblock {Smoothness spaces related to Morrey spaces - a survey, I}.
\newblock {\em Eurasian Math. J.}, 3(3):110--149, 2012.

\bibitem[Sic13]{Sic13}
W.~Sickel.
\newblock {Smoothness spaces related to Morrey spaces - a survey. II}.
\newblock {\em Eurasian Math. J.}, 4(1):82--124, 2013.

\bibitem[ST87]{ST87}
H.-J. Schmeisser and H.~Triebel.
\newblock {\em Topics in {F}ourier analysis and function spaces}.
\newblock Wiley, Chichester, 1987.

\bibitem[ST95]{SiT95}
W.~Sickel and H.~Triebel.
\newblock H\"older inequalities and sharp embeddings in function spaces of
  {$B^s_{p,q}$} and {$F^s_{p,q}$} type.
\newblock {\em Z. Anal. Anwendungen}, 14:105--140, 1995.

\bibitem[ST07]{SawTan}
Y.~Sawano and H.~Tanaka.
\newblock Decompositions of {B}esov-{M}orrey spaces and
  {T}riebel-{L}izorkin-{M}orrey spaces.
\newblock {\em Math. Z.}, 257(4):871--905, 2007.

\bibitem[ST09a]{SawTan-2}
Y.~Sawano and H.~Tanaka.
\newblock Besov-{M}orrey spaces and {T}riebel-{L}izorkin-{M}orrey spaces for
  non-doubling measures.
\newblock {\em Math. Nachr.}, 282(12):1788--1810, 2009.

\bibitem[ST09b]{Se-Tr}
A.~Seeger and W.~Trebels.
\newblock Low regularity classes and entropy numbers.
\newblock {\em Arch. Math.}, 92:147--157, 2009.

\bibitem[SYY10]{SYY10}
Y.~Sawano, D.~Yang, and W.~Yuan.
\newblock New applications of {B}esov-type and {T}riebel-{L}izorkin-type
  spaces.
\newblock {\em J. Math. Anal. Appl.}, 363(1):73--85, 2010.

\bibitem[Tri78]{T78}
H.~Triebel.
\newblock {\em Interpolation theory, function spaces, differential operators}.
\newblock North-Holland, Amsterdam, 1978.
\newblock (2nd ed., Barth, Heidelberg, 1995).

\bibitem[Tri83]{T83}
H.~Triebel.
\newblock {\em Theory of function spaces}.
\newblock Birkh\"auser, Basel, 1983.
\newblock Reprint (Modern Birkh\"auser Classics) 2010.

\bibitem[Tri92]{T92}
H.~Triebel.
\newblock {\em Theory of function spaces {II}}.
\newblock Birkh\"auser, Basel, 1992.
\newblock Reprint (Modern Birkh\"auser Classics) 2010.

\bibitem[Tri01]{T01}
H.~Triebel.
\newblock {\em The structure of functions}.
\newblock Birkh\"auser, Basel, 2001.

\bibitem[Tri06]{T06}
H.~Triebel.
\newblock {\em Theory of function spaces {III}}.
\newblock Birkh\"auser, Basel, 2006.

\bibitem[Tri08]{T08}
H.~Triebel.
\newblock {\em Function spaces and wavelets on domains}, volume~7 of {\em EMS
  Tracts in Mathematics (ETM)}.
\newblock European Mathematical Society (EMS), Z\"urich, 2008.

\bibitem[Tri10]{T10}
H.~Triebel.
\newblock {\em Bases in function spaces, sampling, discrepancy, numerical
  integration}, volume~11 of {\em EMS Tracts in Mathematics (ETM)}.
\newblock European Mathematical Society (EMS), Z\"urich, 2010.

\bibitem[Tri12]{T12}
H.~Triebel.
\newblock {\em {Faber systems and their use in sampling, discrepancy, numerical
  integration}}.
\newblock EMS Series of Lectures in Mathematics. EMS Publishing House,
  Z\"urich, 2012.

\bibitem[Tri13]{T13}
H.~Triebel.
\newblock {\em {Local function spaces, heat and Navier-Stokes equations}},
  volume~20 of {\em EMS Tracts in Mathematics (ETM)}.
\newblock European Mathematical Society (EMS), Z\"urich, 2013.

\bibitem[Tri14]{T14}
H.~Triebel.
\newblock {\em {Hybrid function spaces, heat and Navier-Stokes equations}},
  volume~24 of {\em EMS Tracts in Mathematics (ETM)}.
\newblock European Mathematical Society (EMS), Z\"urich, 2014.

\bibitem[Tri19]{T19}
H.~Triebel.
\newblock {\em {Function spaces with dominating mixed smoothness}}.
\newblock EMS Series of Lectures in Mathematics. EMS Publishing House,
  Z\"urich, 2019.

\bibitem[Tri20]{T20}
H.~Triebel.
\newblock {\em Theory of function spaces {IV}}.
\newblock Birkh\"auser, Basel, 2020.

\bibitem[TX05]{TX}
L.~Tang and J.~Xu.
\newblock Some properties of {M}orrey type {B}esov-{T}riebel spaces.
\newblock {\em Math. Nachr.}, 278(7-8):904--917, 2005.

\bibitem[Vyb08]{Vyb08}
J.~Vyb{\'\i}ral.
\newblock A new proof of the {J}awerth-{F}ranke embedding.
\newblock {\em Rev. Mat. Complut.}, 21(1):75--82, 2008.

\bibitem[Vyb10]{vybiral-4}
J.~Vyb{\'\i}ral.
\newblock On sharp embeddings of {B}esov and {T}riebel-{L}izorkin spaces in the
  subcritical case.
\newblock {\em Proc. Amer. Math. Soc.}, 138(1):141--146, 2010.

\bibitem[Woj97]{Woj97}
P.~Wojtaszczyk.
\newblock {\em A mathematical introduction to wavelets}, volume~37 of {\em
  London Math. Soc. Student Texts}.
\newblock Cambridge Univ. Press, Cambridge, 1997.

\bibitem[WYY17]{WYY}
S.~Wu, D.~Yang, and W.~Yuan.
\newblock Equivalent quasi-norms of {B}esov-{T}riebel-{L}izorkin-type spaces
  via derivatives.
\newblock {\em Results Math.}, 72(1-2):813--841, 2017.

\bibitem[YFS19]{YFS}
M.~Yang, Z.~Fu, and J.~Sun.
\newblock Existence and large time behavior to coupled chemotaxis-fluid
  equations in {B}esov-{M}orrey spaces.
\newblock {\em J. Differential Equations}, 266(9):5867--5894, 2019.

\bibitem[YHM{\etalchar{+}}15]{YHMSY15}
W.~Yuan, D.D. Haroske, S.D. Moura, L.~Skrzypczak, and D.~Yang.
\newblock Limiting embeddings in smoothness {M}orrey spaces, continuity
  envelopes and applications.
\newblock {\em J. Approx. Theory}, 192:306--335, 2015.

\bibitem[YHSY15]{YHSY15}
W.~Yuan, D.D. Haroske, L.~Skrzypczak, and D.~Yang.
\newblock Embeddings properties of {B}esov type spaces.
\newblock {\em Appl. Anal.}, 94(2):318--340, 2015.

\bibitem[YSY10]{YSY10}
W.~Yuan, W.~Sickel, and D.~Yang.
\newblock {\em Morrey and {C}ampanato meet {B}esov, {L}izorkin and {T}riebel},
  volume 2005 of {\em Lecture Notes in Mathematics}.
\newblock Springer, Berlin, 2010.

\bibitem[YSY13]{YSY13}
W.~Yuan, W.~Sickel, and D.~Yang.
\newblock On the coincidence of certain approaches to smoothness spaces related
  to {M}orrey spaces.
\newblock {\em Math. Nachr.}, 286:1571--1584, 2013.

\bibitem[YSY15]{YSY15}
W.~Yuan, W.~Sickel, and D.~Yang.
\newblock Interpolation of {M}orrey-{C}ampanato and related smoothness spaces.
\newblock {\em Sci. China Math.}, 58(9):1835--1908, 2015.

\bibitem[YSY20]{YSY20}
W.~Yuan, W.~Sickel, and D.~Yang.
\newblock The {H}aar system in {B}esov-type spaces.
\newblock {\em Studia Math.}, 253(2):129--162, 2020.

\bibitem[YY08]{YY}
D.~Yang and W.~Yuan.
\newblock A new class of function spaces connecting {T}riebel-{L}izorkin spaces
  and {$Q$} spaces.
\newblock {\em J. Funct. Anal.}, 255(10):2760--2809, 2008.

\bibitem[YY10]{YY-2}
D.~Yang and W.~Yuan.
\newblock New {B}esov-type spaces and {T}riebel-{L}izorkin-type spaces
  including {$Q$} spaces.
\newblock {\em Math. Z.}, 265(2):451--480, 2010.

\bibitem[YY13]{YY-4}
D.~Yang and W.~Yuan.
\newblock Relations among {B}esov-type spaces, {T}riebel-{L}izorkin-type spaces
  and generalized {C}arleson measure spaces.
\newblock {\em Appl. Anal.}, 92(3):549--561, 2013.

\bibitem[YYZ13]{YYZ}
D.~Yang, W.~Yuan, and C.~Zhuo.
\newblock Complex interpolation on {B}esov-type and {T}riebel-{L}izorkin-type
  spaces.
\newblock {\em Anal. Appl. (Singap.)}, 11(5):1350021, 45, 2013.

\bibitem[ZHS20]{ZHS20}
C.~Zhuo, M.~Hovemann, and W.~Sickel.
\newblock Complex interpolation of {L}izorkin-{T}riebel-{M}orrey spaces on
  domains.
\newblock {\em Anal. Geom. Metr. Spaces}, 8(1):268--304, 2020.

\bibitem[Zhu21]{Z21}
C.~Zhuo.
\newblock Complex interpolation of {B}esov-type spaces on domains.
\newblock {\em Z. Anal. Anwend.}, 40(3):313--347, 2021.

\end{thebibliography}
